\documentclass[12pt]{amsart}
\usepackage{amsmath,amsthm,amsfonts,amssymb}

\usepackage{colonequals} 
\usepackage{stmaryrd} 
\usepackage[
margin=2.5cm]{geometry}
\usepackage[usenames,dvipsnames]{color}
\usepackage{mathrsfs}
\usepackage[colorlinks,allcolors=blue]{hyperref}
\usepackage{enumitem}
\setlist[enumerate]{%
	leftmargin=*,%
	label=(\roman*)} 
\setlist[itemize]{leftmargin=*}
\usepackage[mathscr]{eucal}
\usepackage{multirow, array}

\usepackage{tikz,xcolor}
\usetikzlibrary{positioning, backgrounds}

\usepackage{
	tikz-qtree}
\usetikzlibrary{matrix,arrows,trees,shapes}
\usepackage{xcolor}

\usepackage{fancyvrb} 
\RecustomVerbatimEnvironment{Verbatim}{Verbatim}%
{xleftmargin=0pt,baselinestretch=0.85}

\usepackage{microtype}

\usepackage{bm} 
\usepackage{bbm} 

\usepackage[all]{xy} 
\usepackage[T1]{fontenc} 

\usepackage{graphics} 

\usepackage[alphabetic]{amsrefs} 

\DefineSimpleKey{bib}{primaryclass}{}
\DefineSimpleKey{bib}{archiveprefix}{}
\newcommand*{\bracketize}[1]{[#1]}

\BibSpec{arXiv}{%
	+{}{\PrintAuthors}{author}
	+{,}{ \textit}{title}
	+{}{ \parenthesize}{date}
	+{,}{ arXiv:}{eprint}
	+{}{ \bracketize}{primaryclass}
}

\usepackage[yyyymmdd]{datetime}

\usepackage{fancyhdr}
\newtheorem{theorem}{Theorem}[section]
\newtheorem{lemma}[theorem]{Lemma}
\newtheorem{proposition}[theorem]{Proposition}
\newtheorem{corollary}[theorem]{Corollary}
\newtheorem{question}[theorem]{Question}

\theoremstyle{definition}%
\newtheorem{definition}[theorem]{Definition}
\newtheorem{example}[theorem]{Example}
\newtheorem{remark}[theorem]{Remark}

\DeclareMathOperator{\Spec}{Spec}

\DeclareMathOperator{\OO}{\mathscr{O}}
\DeclareMathOperator{\Supp}{Supp}%
\newcommand{\fmm}{\mathfrak{m}}

\newcommand{\sF}{\mathscr{F}}%

\newcommand{\PP}{\mathbb{P}}

\renewcommand{\setminus}{\smallsetminus}

\newcommand{\QQ}{\mathbb{Q}}

\newcommand{\cX}{\mathcal{X}}
\newcommand{\cY}{\mathcal{Y}}
\newcommand{\cZ}{\mathcal{Z}}

\newcommand{\cA}{\mathcal{A}}
\newcommand{\cC}{\mathcal{C}}

\newcommand{\cU}{\mathcal{U}}

\newcommand{\cJ}{\mathscr{J}}

\newcommand{\cL}{\mathcal{L}}
\newcommand{\cM}{\mathcal{M}}

\DeclareMathOperator{\Aut}{Aut}

\DeclareMathOperator{\Div}{div}

\newif\ifhascomments \hascommentstrue
\ifhascomments
	\newcommand{\matt}[1]{{\color{red}[[\ensuremath{\spadesuit\spadesuit\spadesuit} #1]]}}
	\newcommand{\jesse}[1]{{\color{blue}[[\ensuremath{\clubsuit\clubsuit\clubsuit} #1]]}}
	\newcommand{\changho}[1]{{\color{teal}[[\ensuremath{\diamondsuit\diamondsuit\diamondsuit} #1]]}}
\else
	\newcommand{\jesse}[1]{}
	\newcommand{\matt}[1]{}
	\newcommand{\changho}[1]{}
\fi

\begin{document}

	\begin{abstract}
		Determining the limiting behaviour of the Jacobian as the underlying curve degenerates has been the subject of much interest. 
		For nodal singularities, there are beautiful constructions of Caporaso as well as Pandharipande of compactified universal Jacobians over the moduli space of stable curves. 
		Alexeev later obtained a canonical such compactification by extending the Torelli map out of the Deligne--Mumford compactification of $\mathcal{M}_{g,n}$. 
		In contrast, Alexeev and Brunyate proved that the Torelli map does not extend over the cuspidal locus in Schubert's alternative compactification of pseudostable curves. 

		In this paper, we consider curves with singularities that locally look like the axes in $m$-space, which we call axis-like singularities. 
		We construct an alternative compactification of $\mathcal{M}_{g,n}$ consisting of curves with such singularities and prove that the Torelli map extends out of this compactification. 
		Furthermore, for every alternative compactification in the sense of Smyth, we identify an axis-like locus over which the Torelli map extends.
	\end{abstract}

	\title{Extending the Torelli map to alternative compactifications of the moduli space of curves}
	
	\author{Changho Han}
	\address{Changho Han, Department of Mathematics, Korea University}
	\email{changho\_han@korea.ac.kr}
	
	\author{Jesse Leo Kass}
	\address{Jesse Leo Kass, Department of Mathematics, UC Santa Cruz}
	\email{jelkass@ucsc.edu}
	
	\author{Matthew~Satriano}
	\thanks{MS was partially supported by a Discovery Grant from the
		National Science and Engineering Research Council of Canada and a Mathematics Faculty Research Chair.}
	\address{Matthew Satriano, Department of Pure Mathematics, University
		of Waterloo}
	\email{msatriano@uwaterloo.ca}
	
	\maketitle

	\section{Introduction} \label{sec:intro}

	The Jacobian of a smooth genus $g$ curve $C$ is a classical object in algebraic geometry which intertwines geometric, arithmetic, and Hodge-theoretic information of $C$. 
	Defined as the moduli space of degree zero line bundles on $C$, it carries the structure of a $g$-dimension Abelian variety. 
	The Torelli Theorem, which says that $C$ can be reconstructed from its Jacobian $\mathrm{Jac}(C)$, led to advances in the construction of moduli spaces via period domains and period maps, such as moduli spaces of 
	abelian varieties (see \cite{AMRT10}), K3 surfaces (see \cite{Kon20}), and hyperK\"ahler manifolds (see \cite{Ver13, Ver20}). Furthermore, Jacobians played a key role in the proof of important arithmetic results such as Faltings's Theorem (over number fields) in \cite{Fal83,Fal84} and the Riemann Hypothesis over finite fields (proved by Weil in \cite{Wei48}).
	
	Motivated by Faltings's usage of integral models of Jacobians in \cite{Fal83,Fal84}, it became of interest to understand both how $\mathrm{Jac}(C)$ varies in moduli as well as how it degenerates. 
	The first of these considerations led to the study of {\it universal Jacobians} $\cJ^d_g$, which parameterize pairs $(L,C)$ with $L$ a degree $d$ line bundle on a curve $C$; 
	there is a natural map $\cJ^d_g\to\cM_g$ realizing $\cJ^d_g$ as a family of Abelian varieties over the moduli space of genus $g$ curves. 
	

	The question of how $\mathrm{Jac}(C)$ degenerates as $C$ does has been a subject of great interest \cite{Igu56, Ish78, D'S79, OS79, AK80I, AK79II, Cap94, Sim94I, Pan96, Est01, Cap08, Mel09, Mel11, KP19, Mel19}. 
	One way to study universal Jacobians and their compactifications comes from Torelli maps. In \cite{Ale04}, Alexeev gave a modular description of the Torelli map as a morphism 
	\begin{align*}
		t_g \colon& \cM_g \to \cA_g\\
		&C \mapsto (\mathrm{Pic}^{g-1}(C),\Theta(C))
	\end{align*}
	where $\cA_g$ is moduli stack of principally polarized Abelian varieties of dimension $g$ and $\Theta(C)$ is the {\it Theta divisor} corresponding to the principal polarization of an Abelian variety $\mathrm{Pic}^{g-1}(C)$. 
	The canonical universal Jacobian $\cJ^{g-1}_g$ can thus by viewed as the pullback of the universal family of Abelian varieties over $\cA_g$. 
	Understanding how $\textrm{Pic}^{g-1}(C)$ behaves as $C$ degenerates to a stable (nodal) curve is then a question of how to construct a compactification of $\cJ^{g-1}_g$ over the Deligne--Mumford compactification $\overline{\cM}_g$; 
	constructing such a compactification then amounts to extending the Torelli map to a modular compactification of $\cA_g$. 
	In \cite{Ale02}, Alexeev constructed a compactification $\overline{\cA}_g^{\mathrm{Ale}}$ of $\cA_g$ parameterizing so-called principally polarized stable semi-abelic pairs. 
	Furthermore, he showed in \cite[\S 5]{Ale04}\footnote{On the level of coarse spaces, this map was known to exist by Namikawa \cite[Corollary 18.9]{Nam76II}.} that the Torelli map extends to 
	\[
		\overline{t}_g \colon \overline{\cM}_g \to \overline{\cA}_g^{\mathrm{Ale}}.
	\]
	As a consequence, $\overline{t}_g$ induces a {\it canonically compactified universal Jacobian} $\overline{\cJ}^{g-1}_g\to\overline{\cM}_g$ whose fiber over any stable curve $C$ is the {\it canonically compactified Jacobian} $\overline{\mathrm{Pic}}^{g-1}(C)$ of $C$.
	\\
	\\
	\indent	From this perspective, the question of how to consistently define a limit of $\mathrm{Pic}^{g-1}(C)$ as the underlying smooth curve $C$ degenerates to a \emph{non-nodal} curve then becomes a question of how to extend the Torelli map to \emph{alternative compactifications} of $\cM_g$ which incorporate non-nodal singularities.

	Already when one moves from nodes to cusps, constructing such limits presents problems. 
	Specifically, there is a natural compactification $\overline{\cM}_g^{\mathrm{ps}}$ of $\cM_g$ introduced by Schubert \cite{Sch91} which disallows elliptic tails and replaces with them cusps; 
	such curves are called \emph{pseudostable curves}. 
	In \cite{AB12}, Alexeev and Brunyate proved that the Torelli map does \emph{not} extend to $\overline{\cM}_g^{\mathrm{ps}}$ precisely because it fails to extend over the cuspidal locus. 
	The moduli stack of pseudostable curves $\overline{\cM}_g^{\mathrm{ps}}$ was historically the first alternative compactification of $\cM_g$, and fits more generally within the Hassett--Keel program \cite{Has05, HH09, HL10, Fed12, CMJL12, HH13, HL14, CMJL14, AFS16, AFSvdW17I, AFS17II, AFS17III}; 
	this program uses the minimal model program to construct alternative compactifications $\overline{\cM}_g(\alpha)$ depending on a parameter $\alpha$. 
	In light of Alexeev and Brunyate's result, in order to understand the limiting behaviour of $\mathrm{Pic}^{g-1}(C)$, we must look beyond cuspidal curves and, correspondingly, beyond the Hassett--Keel program.

	
	In this paper, we work over an algebraically closed field of characteristic zero. 
	We introduce the following class of singularities which are locally given by the axes in $m$-space. 
	Note that nodes precisely correspond to the case where $m=2$.
	\begin{definition} \label{def:axis_sing}
		Let $C$ be a curve defined over an algebraically closed field $k$ and let $m \ge 2$ be an integer. 
		Then, a singular point $p$ of $C$ is an {\it $m$-axis point} if
		\[
			\widehat{\OO}_{C,p} \cong k\llbracket x_1,\dotsc,x_m \rrbracket / (x_ix_j \, : \, 1 \le i < j \le m).
		\]
		We say $C$ is \emph{axis-like} if $\mathrm{Aut}(C)$ is finite and each of the singular points $p$ of $C$ is an  $m_p$-axis singularity. 
	\end{definition}
	Unlike cusps, $m$-axis singularities and nodal singularities share the common feature of being seminormal. Thus, they are, in many ways a more natural class to examine than cusps. 

	A subtle issue that arises in our work is that whether or not the Torelli map extends to an axis-like curve $C$ depends on the \emph{global geometry} of $C$. 
	Indeed, let us illustrate this through two examples of curves with a unique $4$-axis singularity, one where the Torelli map extends and one where it does not. 
	For the latter, any irreducible curve $C$ with finite automorphism group and a unique singular point which is a $4$-axis point admits two distinct stable models $C'_1$ and $C'_2$ in $\overline{\cM}_g$ whose canonically compactified Jacobians are non-isomorphic; 
	in fact, this is the same issue that arises for cusps in the aforementioned work of Alexeev and Brunyate \cite{AB12}, precluding the Torelli map from extending, see Example~\ref{ex:why_quasi-separating_axis-like} for more further details. 
	On the other hand, if $C$ is a curve constructed by gluing four non-isomorphic elliptic curves $E_1,\dotsc,E_4$ to a $4$-axis point of $C$, then the canonically compactified Jacobian of any stable model of $C$ is isomorphic to the product $\mathrm{Jac}(E_1) \times \dotsb \times \mathrm{Jac}(E_4)$. 
	Hence, the global geometry of curves with $m$-axis points influences whether the Torelli map extends. It is therefore necessary to impose additional constraints, which 
	leads to our definition of 
	\emph{separating (resp.~quasi-separating) axis-like singularities}; 
	these are $m$-axis singularities where $m \ge 3$ and 
	no branches (resp.~at most $3$ branches) of the singularity meet away from the singular point, see Definitions~\ref{def:separating_axis-like} and \ref{def:quasi-separating_axis-like}.

	%

	In this paper, we:\vspace{0.3em}

	\begin{enumerate}
		\item[(1)] Construct a normal alternative compactification $\overline{\cM}_{g,n}(\sF)$ of $\cM_{g,n}$ consisting of curves with either nodes or separating axis-like singularities (Proposition~\ref{prop:moduli_separating_axis-like}).\vspace{0.6em}

		\item[(2)] Prove the Torelli map extends to $\overline{\cM}_{g,n}(\sF)$ (Theorem~\ref{thm:main_thm_pointed_extension_axis_like}).\vspace{0.6em}

		\item[(3)] Prove that in \emph{every} alternative compactification of $\cM_{g,n}$ in the sense of Smyth \cite{Smy13}, the Torelli map extends to the locus of quasi-separating axis-like curves, which is normal (Theorem~\ref{thm:main_thm_partial_extension_axis_like}).\vspace{0.3em}
		%
	\end{enumerate}

	As far as we know, our work gives the first extension of the Torelli map beyond the nodal case. 
	Our construction of $\overline{\cM}_{g,n}(\sF)$ relies on the general theory of Smyth \cite{Smy13} where he classifies all alternative compactifications of $\cM_{g,n}$ that are proper Deligne--Mumford (DM) moduli stacks of curves. 
	In \cite[Theorem 1.9]{Smy13} he shows all such compactifications are given by $\cZ$-stable curves $\overline{\cM}_{g,n}(\cZ)$, where $\cZ$ is a combinatorial gadget recording how to replace a stable curve $C$ with another curve $C'$, possibly with worse singularities; see Section~\ref{subsec:Smyth_stability} for more details. 
	
	Our first main theorem is then:

	\begin{theorem} \label{thm:main_thm_pointed_extension_axis_like}
		For $g \ge 2$ or $g,n \ge 1$, the pointed Torelli map $t_{g,n} \colon \cM_{g,n} \to \cA_g$ extends uniquely to a morphism
		\[
			\overline{t}_{g,n}^{\sF} \colon \overline{\cM}_{g,n}(\sF) \to \overline{\cA}_g^{\mathrm{Ale}}
		\]
		of (proper DM) stacks, where $\overline{\cM}_{g,n}(\sF)$ denotes the moduli stack of stable $n$-pointed separating axis-like curves of genus $g$.
	\end{theorem}

	\begin{remark}
		As an immediate consequence, pulling back the universal family from 
		$\overline{\cA}_g^{\mathrm{Ale}}$ yields an alternative compactification $\overline{\cJ}^{g-1}_{g,n}(\sF)$ of the canonical pointed universal Jacobian $\cJ^{g-1}_{g,n}$ over $\overline{\cM}_g(\sF)$. 
		In particular, taking $n=0$, we obtain an alternative compactification of the canonical universal Jacobian $\cJ^{g-1}_g$.
	\end{remark}

	\begin{remark}
		One reason to consider the pointed case (i.e., $n>0$) is that our result also yields alternative compactifications of the pointed universal Jacobian $\cJ^d_{g,n}$ for all $d$. 
		Indeed, for any integer $d$ and any tuple $(d_1,\dotsc,d_n)$ of integers such that $d-g+1 = d_1 + \dotsb + d_n$, there is an isomorphism $\cJ^{g-1}_{g,n} \xrightarrow{\simeq} \cJ^d_{g,n}$ between universal Jacobians, where one twists the degree $g-1$ line bundle by $d_i$ copies of the $i$th marked point. 
		In particular, this is isomorphism is canonical when $n=1$. 
		Via this identification, we therefore also obtain an alternative compactification of the pointed universal Jacobian $\cJ^d_{g,n}$.
	\end{remark}

	Our next main theorem concerns all alternative compactifications $\overline{\cM}_{g,n}(\cZ)$ introduced by Smyth \cite{Smy13}. Recall that Alexeev--Brunyate \cite{AB12} showed that for $g \ge 2$, the Torelli map $t_g$ does not extend to a morphism from the moduli stack $\overline{\cM}_g^{\mathrm{ps}}$ of pseudostable curves to $\overline{\cA}_g^{\mathrm{Ale}}$; but the Torelli map does extend to a morphism $\cU_g^{\mathrm{ps}} \to \overline{\cA}_g^{\mathrm{Ale}}$ where $\cU_g$ is the open substack of $\overline{\cM}_g^{\mathrm{ps}}$ of nodal curves. This motivates the following question:

	\begin{question}\label{q:alternative-cmpt-extension}
		Given an alternative compactification $\overline{\cM}_{g,n}(\cZ)$ of $\cM_{g,n}$ in Smyth's sense \cite{Smy13}, what is the 
		largest open substack of $\overline{\cM}_{g,n}(\cZ)$ over which the pointed Torelli map $t_{g,n}^\cZ \colon \overline{\cM}_{g,n}(\cZ) \dashrightarrow \overline{\cA}_g^{\mathrm{Ale}}$ is a morphism?
	\end{question}

	\begin{remark}
		Notice that 
		this open substack can be much larger than the nodal locus, e.g., our Theorem~\ref{thm:main_thm_pointed_extension_axis_like} shows that the pointed Torelli map extends entirely to a morphism over the alternative compactification $\overline{\cM}_{g,n}(\sF)$.
	\end{remark}

	Our second main result is then:


	\begin{theorem} \label{thm:main_thm_partial_extension_axis_like}	
		Let $\overline{\cM}_{g,n}(\cZ)$ be an alternative compactification of $\cM_{g,n}$ in the sense of Smyth \cite{Smy13} with $g \ge 2$ or $g,n\ge1$. 
		Then the pointed Torelli map $t_{g,n} \colon \cM_{g,n} \to \cA_g$ extends uniquely to a morphism
		\[
			t_{g,n}^\cZ \colon \cM_{g,n}(\cZ)^{\textrm{qs-axis}} \to \overline{\cA}_g^{\mathrm{Ale}}
		\]
		where $\cM_{g,n}(\cZ)^{\textrm{qs-axis}}$ is the open substack of $\overline{\cM}_{g,n}(\cZ)$ parameterizing quasi-separating axis-like curves.
	\end{theorem}


	Similar to Theorem \ref{thm:main_thm_pointed_extension_axis_like}, our Theorem~\ref{thm:main_thm_partial_extension_axis_like} shows that $\cJ^d_{g,n}$ admits a partial compactification $\cJ^d_{g,n}(\cZ)^{\textrm{qs-axis}} \to \cM_{g,n}(\cZ)^{\textrm{qs-axis}}$ as a flat family of principally polarized stable semi-abelic pairs; 
	in other words, there is a consistent choice of a flat limit of $\mathrm{Pic}^{g-1}(C)$ as an underlying smooth $n$-pointed curve $(C;\sigma_1,\dotsc,\sigma_n)$ degenerates into an $n$-pointed $\cZ$-stable quasi-separating axis-like curve.

	\subsection*{Main techniques:~a stacky extension theorem} \label{subsec:intro_techniques}

	Theorem~\ref{thm:main_thm_partial_extension_axis_like} is a strict generalization of Theorem \ref{thm:main_thm_pointed_extension_axis_like} since $\cM_{g,n}(\sF)^{\textrm{qs-axis}}=\overline{\cM}_{g,n}(\sF)$. 
	We prove Theorem~\ref{thm:main_thm_partial_extension_axis_like} by first constructing the map 
	\[
		|t_{g,n}^\cZ|\colon |\cM_{g,n}(\cZ)^{\textrm{qs-axis}}| \to |\overline{A}_g^{\mathrm{Ale}}|
	\]
	on the level of isomorphism classes. 
	This is shown in Theorem~\ref{thm:extended_Torelli_Z-stable} and relies on Caporaso and Viviani's combinatorial characterization of the fibers of 
	$\overline{t}_{g,n} \colon \overline{\cM}_{g,n} \to \overline{\cA}_g^{\mathrm{Ale}}$ given in \cite{CV11}; 
	see Section~\ref{sec:extending_Torelli_map}, particularly Theorem~\ref{thm:Torelli_C1-equivalence}.

	To show that the set-theoretic map $|t_{g,n}^\cZ|$ lifts to a morphism of stacks, we prove and apply a new stack-theoretic extension theorem (Theorem \ref{thm:extension-stacks} below). 
	This result is inspired by the corresponding scheme-theoretic results \citelist{\cite{GG14}*{Theorem 7.3} \cite{AET23}*{Lemma 3.18}}, which say that if $X$ is normal, then a rational map $X \dasharrow Y$ of proper schemes extends to a morphism if and only if it extends as a map of sets and extends on the level of DVR-valued points. 
	We imagine our new stack-theoretic generalization will be useful beyond the context of our paper.

	We wish to emphasize that our stacky extension theorem (Theorem \ref{thm:extension-stacks}) is \emph{not} a simple generalization of the corresponding scheme-theoretic statement. 
	In fact, if one simply substitutes the words ``Deligne--Mumford stack'' in place of ``scheme/variety'' in the statements of \citelist{\cite{GG14}*{Theorem 7.3} \cite{AET23}*{Lemma 3.18}}, then the result is \emph{false}.

	\begin{example}
		Let $X$ be an affine cone over an elliptic curve $E$. If $0\in X$ denotes the cone point, then consider the map $\psi\colon X\setminus0\to B\mu_2$ induced by a $2$-isogeny $E'\to E$ of elliptic curves. 
		This defines a rational map $X\dasharrow B\mu_2$ which extends on the level of sets since $|B\mu_2|$ is a point. 
		We show in Example~\ref{ex:lasc_extension} that the above rational map extends to a morphism on the level of DVR points, however any extension of the double cover to $X$ is necessarily ramified. 
		Thus, $\psi$ does not extend to a morphism from $X$ and hence the naive generalization of \citelist{\cite{GG14}*{Theorem 7.3} \cite{AET23}*{Lemma 3.18}} to Deligne--Mumford stacks false.
	\end{example}

	Furthermore, to prove our stack-theoretic extension theorem, a completely new proof technique is needed. 
	Indeed, the proofs of the scheme-theoretic versions of the result proceed by considering $\overline{\Gamma}_f$, the closure of the graph of the morphism $f\colon X\dasharrow Y$, and applying Zariski's Main Theorem to prove $\overline{\Gamma}_f \to X$ is an isomorphism. 
	This proof simply does not work for stacks since $\overline{\Gamma}_f \to X$ may be non-representable.

	We note also that when one requires the source stack to be \emph{smooth}, Deopurkar and the first author in \cite[Lemma 7.4]{DH21} gave a stacky extension theorem, however their result is too restrictive for our purposes:~$\cM_{g,n}(\cZ)^{\textrm{qs-axis}}$ is normal but not smooth in general; 
	even worse, one expects $\overline{\cM}_{g,n}(\cZ)$ to be arbitrarily singular in light of Vakil's Murphy's Law \cite{Vak06}. 

	We prove our stacky extension theorem for rational maps $\cX\dasharrow\cY$ where $\cX$ is {\it locally algebraically simply connected}, a class of singularities introduced by Koll\'ar in \cite[Definition 7.1]{Kol93}. 
	A scheme $X$ is said to be locally algebraically simply connected if for every closed point $x\in X$, there is a resolution of singularities $Y_x \to \Spec \OO_{X,x}^h$ which has trivial \'etale fundamental group. 
	See Definition~\ref{def:lasc} and Definition~\ref{def:lasc_DMstack} for more details, including the stack-theoretic generalization of the definition.

	\begin{theorem}[Stacky Extension Theorem]\label{thm:extension-stacks}
		Let $\cX$ be a separated Deligne--Mumford stack that is locally of finite type and locally algebraically simply connected. 
		Let $\cY$ be a proper Deligne--Mumford stack, $\cU\subset\cX$ an open dense substack, and $f\colon\cU\to\cY$ a morphism such that
		\begin{enumerate}
			\item\label{extension:sets} $|f|\colon|\cU|\to|\cY|$ extends to a map $|g|\colon|\cX|\to|\cY|$ on the underlying sets of the associated topological spaces, and
			\item\label{extension:DVRs} (DVR extension property) for every DVR $R$ with fraction field $K$ and residue field $k$ and every commutative diagram
			\[
				\xymatrix{
					\Spec K\ar[r]\ar[d] & \cU\ar[dr]^-{f}\ar@{^{(}->}[d]^-{i} & \\
					\Spec R\ar[r]^-{h}\ar@/_1.1pc/@{-->}[rr]_-{\overline{h}} & \cX & \cY
				}
			\]
			of solid arrows, there exists $\overline{h}$ making the diagram commute such that $\overline{h}(\Spec k)=|g|(h(\Spec k))$ in $|\cY|$.
		\end{enumerate}
		Then there exists an extension $g\colon\cX\to\cY$ of $f$ which is unique up to $2$-isomorphism. 
	\end{theorem}

	We prove that $\cM_{g,n}(\cZ)^{\textrm{qs-axis}}$ is normal by using the normality of versal deformation spaces of $m$-axis singularities from \cite{PR24} (see Remark~\ref{rmk:deformation_space_reduced?}), then prove that $\cM_{g,n}(\cZ)^{\textrm{qs-axis}}$ is locally algebraically simply connected by analyzing the contracting locus from an open substack of $\overline{\cM}_{g,n}$ to $\cM_{g,n}(\cZ)^{\textrm{qs-axis}}$. 
	We verify the Torelli map satisfies the hypotheses of Theorem~\ref{thm:extension-stacks} in Theorem~\ref{thm:extended_Torelli_Z-stable}.

	

	\subsection*{Outline} \label{subsec:outline}
	Section~\ref{sec:prelim} gathers several preliminary results including properties of locally algebraically simply connected singularities introduced by Koll\'ar in \cite[\S 7]{Kol93}. 
	It includes Theorem~\ref{thm:asc_fib_lasc}, which is a tool to check when a scheme is locally algebraically simply connected. 
	In Section~\ref{sec:extension_DMstacks}, we prove the Stacky Extension Theorem (Theorem~\ref{thm:extension-stacks}) and provide an example (Example~\ref{ex:lasc_extension}) of why the naive generalization of the scheme-theoretic statement does not work (using Example~\ref{ex:lc_not_lasc_general}).
	
	Section~\ref{sec:Smyth_moduli} and Section~\ref{sec:qsep_axis-like} deal with alternative compactifications of $\cM_{g,n}$. 
	We summarize Smyth's combinatorial characterization of alternative compactifications of $\cM_{g,n}$ in Section~\ref{subsec:Smyth_stability}. 
	Then in Section~\ref{subsec:stable_axis-like}, we define stable separating axis-like curves (Definition~\ref{def:separating_axis-like}) and construct the moduli stack $\overline{\cM}_{g,n}(\sF)$ of $n$-pointed stable separating axis-like curves of genus $g$. 
	Various properties of $\overline{\cM}_{g,n}(\sF)$ are proved in the remainder of Section~\ref{subsec:stable_axis-like}, particularly Corollary~\ref{cor:moduli_FL_lasc} which verifies that $\overline{\cM}_{g,n}(\sF)$ is locally algebraically simply connected. 
	We define the open substack $\cM_{g,n}(\cZ)^{\textrm{qs-axis}}$ of any fixed alternative compactification $\overline{\cM}_{g,n}(\cZ)$ of $\cM_{g,n}$ in Section~\ref{sec:qsep_axis-like} then prove Corollary~\ref{cor:moduli_Z-stable_qs_axis-like_lasc} which verifies that $\cM_{g,n}(\cZ)^{\textrm{qs-axis}}$ is locally algebraically simply connected.
	
	Section~\ref{sec:extending_Torelli_map} starts with a summary of Caporaso and Viviani's combinatorial analysis of Alexeev's compactified Torelli map $\overline{t}_{g,n} \colon \overline{\cM}_{g,n} \to \overline{\cA}_g^{\mathrm{Ale}}$;
	it culminates with Theorem~\ref{thm:Torelli_C1-equivalence}. 
	We prove the that Torelli map extends to the quasi-separating axis-like locus in Theorem~\ref{thm:extended_Torelli_Z-stable} and Theorem~\ref{thm:compactified_Torelli_stable_separating_axis-like}, which immediately implies Theorem~\ref{thm:main_thm_pointed_extension_axis_like} and Theorem~\ref{thm:main_thm_partial_extension_axis_like} (the two main theorems from Section~\ref{sec:intro}).
	
	\subsection*{Conventions} \label{subsec:conventions}
	All schemes and stacks are defined over an algebraically closed field of characteristic zero (called the base field). 
	Similarly, fields are assumed to be field extensions of the base field. 
	A resolution (of singularities) of a locally integral locally Noetherian scheme $X$ is a proper birational morphism $\varphi \colon Y \to X$ such that $Y$ is regular (i.e., formally smooth over the base field), but not necessarily locally of finite type. 
	Given a scheme $X$, $\pi_0(X)$ means the number of connected components of $X$. 
	Given a morphism $f \colon X \to Y$ with a subscheme $T$ of $Y$, $f^{-1}(T)$ means the scheme-theoretic preimage of $T$ in $X$, and $\Supp (f^{-1}(T))$ refers to the support of the scheme-theoretic preimage.
	
	A curve means a connected reduced proper scheme of pure dimension one. 
	Genus of a curve (or its dual graph) means the arithmetic genus of a curve (or its dual graph) unless otherwise specified. 
	The pair $(g,n)$ of an arithmetic genus $g$ and the number $n$ of marked points of a curve is assumed to satisfy $2g-2+n>0$ unless otherwise specified. 
	Deligne--Mumford stack (DM stacks for short) means an algebraic stack with unramified diagonal as in \cite[Tag 04YW]{Stacks}.
	
	Graphs are assumed to be undirected and can have loops (self-edges) and half-edges. 
	Given a graph $\Gamma$, and $H$ is a subset of the set $V(\Gamma)$ of vertices of $\Gamma$, then $H$ as a subgraph of $\Gamma$ is understood as the complete subgraph of $\Gamma$ whose vertex set is $v(H)$, unless otherwise specified. 
	Two disjoint complete subgraphs $H_1, H_2$ of $\Gamma$ meet at $m$ number of edges if the number of edges of $\Gamma$ connecting $H_1$ and $H_2$ are exactly $m$.

	\section{Preliminaries} \label{sec:prelim}
	
	\subsection{(Relative) Normalizations of generically finite \'etale morphisms} \label{subsec:rel_normaliz}
	
	Here, we collect several basic lemmas about relative normalizations that will be used for the proof of Theorem \ref{thm:extension-stacks}. We include these for lack of a suitable reference.
	
	Recall from \cite[Tag 035H]{Stacks} that given a quasi-compact and quasi-separated morphism $f \colon Y \to X$ of schemes, the {\it relative normalization} of $X$ in $Y$ is the morphism $\nu \colon X' \to X$, where $X' \colonequals \Spec_X(\OO')$ and $\OO'$ is the integral closure of $\OO_X$ in $f_*\OO_Y$. When $f$ is proper, then \cite[Tag 03H2]{Stacks} implies that $X' \cong \Spec_X(f_*\OO_Y)$ and $f$ decomposes into the Stein factorization $Y \to X' \to X$.
	
	\medskip
	
	The following lemma shows that if $\pi \colon Y \to X$ is a finite generically \'etale morphism between normal schemes and $x \in X$, then the pullback of $\pi$ under the natural map $\Spec \OO_{X,x} \to X$ is also a finite generically \'etale morphism between normal schemes. In other words, for finite generically \'etale morphisms, relative normalization commutes with localization.
	
	\begin{lemma}\label{l:rel_normal_localization}
		Let $A$ be an integrally closed subring of a field $K$ of characteristic 0, and let $L$ be a finite separable field extension of $K$. If $B$ is the integral closure of $A$ in $L$ and $\mathfrak{p} \subset A$ is a prime ideal, then the integral closure of $A_\mathfrak{p}$ in $L$ is exactly $B \otimes_A A_\mathfrak{p}$.
	\end{lemma}
	\begin{proof}
		Observe that a subring $R$ of $L$ is the integral closure of $A$ in $L$ if and only if $R$ is finite over $A$, normal, and $\mathrm{Frac}(R) = L$.
		
		Notice that $B \otimes_A A_{\mathfrak{p}}$ is isomorphic to the localization of $B \subset L$ by a multiplicative subset $A \setminus \mathfrak p \subset L$, so $B \otimes_A A_{\mathfrak{p}}$ is also a subring of $L$.
		Since $B$ is the integral closure of $A$ in $L$, we see $B$ is finite over $A$, normal, and $\mathrm{Frac}(B) = L$.
		Then since $B \otimes_A A_{\mathfrak{p}}$ is the localization of $A \setminus \mathfrak p \subset L$, it is finite over $A_{\mathfrak p}$, normal, and $\mathrm{Frac}(B \otimes_A A_{\mathfrak{p}}) = L$.
		Hence, $B \otimes_A A_\mathfrak{p}$ is the integral closure of $A_{\mathfrak p}$ in $L$ by above.
	\end{proof}
	
	
	\begin{lemma}\label{l:torsor-extension-over-DVR}
		Let $D=\Spec R$ with $R$ a DVR and let $D^\circ=\Spec K$ be its generic point. Suppose that  $\pi\colon Q\to D$ is a finite morphism of normal schemes whose restriction to each irreducible component of $Q$ is surjective. If $\pi$ restricted to the generic fiber $Q^\circ \colonequals Q\times_D D^\circ$ is \'etale, then $\pi$ is the relative normalization of $D$ in $Q\times_D D^\circ$. In particular, if there exists a finite \'etale morphism $P\to D$ with $P\times_D D^\circ\simeq Q\times_D D^\circ$, then $\pi$ is \'etale.
	\end{lemma}
	\begin{proof}
		Observe that $Q^\circ \cong \sqcup_i \Spec L_i$ with each $L_i$ a finite separable field extension of $K$. If $Q_0$ is a connected component of normal $Q$, then it must be irreducible, and $\pi$ restricts to a surjection $Q_0 \to D$ by the assumptions. This implies that $Q_0$ contains $\Spec L_i$ for exactly one $i$. Note that $Q_0$ is irreducible and affine because $Q$ is affine as a finite cover of affine $D$. As $Q_0$ is a normal affine scheme that contains $\Spec L_i$ as its generic point, $H^0(Q_0)$ is integrally closed in $L_i$. As $Q_0 \to D$ finite implies that the pullback $R \to H^0(Q_0)$ is an integral ring extension, $H^0(Q_0)$ must be the integral closure of $R$ in $L_i$. Hence, $Q_0$ is the relative normalization of $D$ in $\Spec L_i$; as a result, $Q$ is also the relative normalization of $D$ in $Q^\circ$.
		
		In particular, if $\eta \colon P \to D$ is finite \'etale, then $P$ is regular because $D$ is as well. Then, the above argument applied to $\eta$ implies that $P$ is also the relative normalization of $D$ in $P\times_D D^\circ \cong Q^\circ$, so that $P \cong Q$ by the uniqueness of the relative normalization. This implies that $\pi \colon Q \to D$ must be \'etale as well.
	\end{proof}

	Combining Lemma~\ref{l:rel_normal_localization} and Lemma~\ref{l:torsor-extension-over-DVR}, we obtain the following corollary:
	\begin{corollary}\label{cor:rel_normal_DVR}
		Suppose that $f \colon Y \to X$ is a finite generically \'etale morphism of normal schemes. Let $x$ be a codimension $1$ point of $X$, so that $\OO_{X,x}$ is a DVR with an induced map $i \colon D \colonequals \Spec \OO_{X,x} \to X$. Then the pullback $f' \colon Y' \to D$ of $f$ along $i$ is finite generically \'etale. Furthermore, if $p \colon P \to D$ is finite \'etale whose generic fiber is isomorphic to the generic fiber of $f'$, then $p$ is isomorphic to $f'$ (i.e., $P \cong Y'$).
	\end{corollary}

	\begin{proof}
		Assume that $X_0$ is an open affine subscheme of the connected component of $X$ containing $x$; since $X$ is normal, $X_0$ is a normal irreducible affine scheme. Define $f_0 \colon Y_0 \to X_0$ to be the restriction of $f$ over $X_0$. By the assumptions on $f$, $f_0$ is also a finite morphism between normal schemes, so that $Y_0$ is affine as well. Generic \'etaleness of $f$ implies that $f_0$ is also finite generically \'etale, so that each irreducible component of $Y_0$ surjects onto $X_0$ via $f_0$. Since $i \colon D \to X$ factors through $X_0$ via $i_0 \colon D \to X_0$, $f'$ is isomorphic to the pullback of $f_0$ along $i$, which is finite with $Y_0$ normal by Lemma~\ref{l:rel_normal_localization}. Because the generic point of $D$ maps isomorphically to the generic point of $X_0$, $f'$ is also generically \'etale as the pullback of $f_0$. This proves the first assertion.
		
		The second assertion follows from Lemma~\ref{l:torsor-extension-over-DVR} because $f' \colon Y' \to D$ is a finite generically \'etale morphism of normal schemes whose generic fiber is isomorphic to $p \colon P \to D$.
	\end{proof}
	
	\subsection{Locally algebraically simply connected schemes} \label{subsec:alsc}
We say a connected scheme is \emph{algebraically simply connected} if there is no nontrivial connected finite \'etale cover. 
There are local versions of simple connectedness by \cite{Kol93, Tak03}. The algebraic analog of locally simply connected is described in \cite[\S 7]{Kol93} by first defining and proving properties for normal analytic spaces over $\mathbb{C}$ then describing how those notions naturally extend to normal schemes defined over any algebraically closed field of characteristic zero. In that regard, we describe the scheme-theoretic versions of definitions and properties in \cite[\S 7]{Kol93}:
	
	\begin{definition}[\cite{Kol93}*{Definition 7.1}] \label{def:lasc}
		A normal scheme $X$ locally of finite type is {\em locally algebraically simply connected} ({\em lasc} for short) if for every closed point $x$ of $X$, there is a resolution of singularities $Y_x \to \Spec \OO_{X,x}^h$ such that $\pi^{\text{\'et}}_1(Y_x) = 1$.
	\end{definition}

	As explained in \cite[Definition 7.1]{Kol93}, $\pi_1^{\text{\'et}}(Y_x)$ is independent of the choice of resolution $f \colon Y_x \to \Spec \OO_{X,x}^h$.
	
	Theorem~\ref{thm:klt_implies_lasc} below provides a large class of examples that are locally algebraically simply connected. Recall that the pair $(X,\Delta)$ of a normal scheme $X$ and an effective $\QQ$-divisor $\Delta$ is {\em Kawamata log terminal} ({\em klt} for short) if there exists a log resolution of singularities $f \colon Y \to X$ such that $f^*(K_X+\Delta) =_{\QQ} K_Y + \sum a_i E_i$ as $\QQ$-Cartier divisors, where $=_{\QQ}$ means the equivalence in $\Div(Y) \otimes \QQ$, each $E_i$ is either an exceptional divisor of $f$ or a proper transform of an irreducible component of $\Supp \Delta$, and $a_i < 1$ for each $i$. Recall that log canonical ({\em lc} for short) means that $a_i \le 1$ for each $i$. Using the discrepancy interpretation $a_i = -a(E_i,X,\Delta)$ as in \cite[Definition 2.25]{KM98}, the definition for klt/lc is independent of the choice of a log resolution $f$. It turns out that klt provides a nice criterion for being locally algebraically simply connected:
	
	\begin{theorem}[{\cite[Theorem 7.5]{Kol93}}] \label{thm:klt_implies_lasc}
		Let $X$ be a normal scheme locally of finite type. If $(X,\Delta)$ is a klt pair for some effective $\QQ$-divisor $\Delta$, then $X$ is locally algebraically simply connected.
	\end{theorem}
	
	When $X$ as in Theorem~\ref{thm:klt_implies_lasc} is defined over $\mathbb{C}$, then Takayama showed in \cite[Theorem 1.1]{Tak03} that any resolution $f \colon Y \to X$ induces an isomorphism $f_* \colon \pi_1(Y) \to \pi_1(X)$. We only make use of $\pi_1^{\text{\'et}}$ as we only deal with $G$-torsors for finite group $G$ for the purpose of applications. The algebraic version of Takayama's \cite[Theorem 1.1]{Tak03} is the following by \cite[Lemma 7.2]{Kol93}:
	
	\begin{lemma}[\cite{Kol93}*{Lemma 7.2}] \label{l:lasc_et}
		Suppose that $X$ is a normal scheme locally of finite type. If $X$ is locally algebraically simply connected and $f \colon Y \to X$ is a resolution of singularities, then the pushforward $f_* \colon \pi_1^{\text{\'et}}(Y) \to \pi_1^{\text{\'et}}(X)$ is an isomorphism.
	\end{lemma}

	If we instead use the definition of locally simply connected for normal varieties defined over $\mathbb{C}$ as in \cite[Definition 7.1]{Kol93}, then the pushforward induces an isomorphism between the topological fundamental groups by \cite[Lemma 7.2]{Kol93}, which is stronger than Lemma~\ref{l:lasc_et}.
	
	An immediate consequence of Lemma~\ref{l:lasc_et} is the following:
	
	\begin{proposition} \label{prop:lasc_rel_normaliz_et}
		Let $f \colon Y \to X$ as in Lemma~\ref{l:lasc_et}, and let $p \colon P \to Y$ be a finite degree $d$ \'etale cover of $Y$. Consider the following commutative diagram:
		\[
			\xymatrix{
				P \ar[r]^{f'} \ar[d]_{p} & Q \ar[d]^{p'}\\
				Y \ar[r]_{f} & X,
			}
		\]
		where $p' \circ f'$ is the Stein factorization of $f \circ p$. Then $p'$ is a finite degree $d$ \'etale cover.
	\end{proposition}

	\begin{proof}
		Without loss of generality, assume that $X$ is connected, i.e., irreducible as $X$ is normal. By Lemma~\ref{l:lasc_et}, there exists a finite degree $d$ \'etale cover $p'' \colon P' \to X$ such that the following is a Cartesian diagram:
		\[
			\xymatrix{
				P \ar[r]^{f''} \ar[d]_{p} & P' \ar[d]^{p''}\\
				Y \ar[r]_{f} & X.
			}
		\]
		
		Since $Q$ is a relative normalization of $X$ in $P$ by the Stein factorization property (see \cite[Tag 03H0]{Stacks}), the universal property of relative normalization as in \cite[Tag 035I]{Stacks} implies that there is a morphism $h \colon Q \to P'$ such that the following diagram commutes:
		\[
		\xymatrix{
			P \ar[r]^{f''} \ar[d]_{f'} & P' \ar[d]^{p''}\\
			Q \ar[r]_{p'} \ar[ru]^{h} & X.
		}
		\]
		Since $p''$ is finite and $X$ is locally of finite type, $P'$ is also locally of finite type, hence $P'$ is a Nagata scheme by \cite[Tag 035B]{Stacks}. Moreover, \cite[Tag 035I]{Stacks} implies that $h$ is the relative normalization of $P'$ in $P$, which is a birational morphism as $f''$ is proper birational as a resolution of singularities of $P'$. 
		Since the conditions of \cite[Tag 03GR]{Stacks} are satisfied, $h$ is also finite. Hence, $h \colon Q \to P'$ is an isomorphism by \cite[Tag 0AB1]{Stacks}.
	\end{proof}

	Proposition~\ref{prop:lasc_rel_normaliz_et} has the following useful consequence:~suppose $X$ is a normal scheme which is locally of finite type, $f\colon Y\to X$ is a resolution, and $p \colon P \to Y$ is a finite \'etale cover such that the relative normalization $p' \colon Q \to X$ of $X$ in $P$ is finite but not \'etale, then $X$ is not locally algebraically simply connected. 
	Using this criterion, Example~\ref{ex:lc_not_lasc_general} 
	gives examples of normal schemes that are not locally algebraically simply connected.
	
	\begin{example}[{Non-lasc elliptic/rational singularity}]
	\label{ex:lc_not_lasc_general}
		We provide two examples of normal schemes that are not locally algebraically simply connected:~one example has an elliptic singularity and the other a rational singularity. The argument, motivated by the previous paragraph and \cite[Example 3]{Sti17}, is via constructing a non-\'etale finite degree 2 cover of $X$ that is a relative normalization of a finite \'etale cover of the resolution of singularities of $X$.
		
		Consider the case when $\rho \colon \tilde{E} \to E$ is an \'etale degree $2$ cover between connected smooth projective varieties, where $K_E$ is a numerically trivial divisor. Suppose that $\cL$ is a very ample line bundle on $E$. By \cite[II. Ex. 5.14]{Har77}, we can replace $\cL$ by a suitably large positive tensor power so that the affine cone of the closed embedding of $E$ (resp., $\tilde{E}$) into $\PP H^0(\cL)$ (resp., $\PP H^0(\rho^*\cL)$) is normal; in this case, the corresponding affine cones are
		\[
			X_E \colonequals \Spec \bigoplus_{m \ge 0} H^0(\cL^m), \;\;\; X_{\tilde{E}} \colonequals \Spec \bigoplus_{m \ge 0} H^0(\rho^*\cL^m).
		\]
		Note that $\rho$ induces natural maps $H^0(\cL^m) \to H^0(\rho^*\cL^m)$ for each $m \ge 0$, which induces a finite morphism $\pi \colon X_{\tilde{E}} \to X_E$; this morphism is \'etale of degree $2$ away from the cone point $0 \in X_E$. Because $\Supp (\rho^{-1}(0))$ is the cone point of $X_{\tilde{E}}$ and $\pi$ is finite of degree $2$, $\pi$ is not \'etale over the cone point $0 \in X_E$.
		
		To see that $X_E$ is not locally algebraically simply connected at $0 \in X_E$, first define $X'_E \colonequals \Spec_E \mathrm{Sym}^* \cL$ and $X'_{\tilde E} \colonequals \Spec_{\tilde{E}} \mathrm{Sym}^* \rho^*\cL$. Note that $X'_E$ (resp., $X'_{\tilde E}$) is the total space of the line bundle $\cL^\vee$ (resp., $\rho^*\cL^\vee$), which is a smooth variety. Then, $\rho$ induces a degree 2 \'etale morphism $\pi' \colon X'_{\tilde E} \to X'_E$. Moreover, the natural morphisms $H^0(\cL^m) \otimes \OO_{E} \to \cL$ of sheaves over $E$ induces a birational proper morphism $\varphi \colon X'_E \to X_E$ (in fact, it is a resolution); similarly, there is an induced birational proper morphism $\varphi' \colon X'_{\tilde{E}} \to X_{\tilde{E}}$. Since $\varphi^{-1}(0) \cong E$ the zero section of $X'_E$ and $K_{X'_E}+E$ is numerically trivial along $E$, $X_E$ has a log canonical singularity at $0 \in X_E$. Since $\varphi$ and $\pi'$ are proper and $X_E$ is affine, $H^0((\varphi \circ \pi')_*\OO_{X'_{\tilde{E}}}) \cong \oplus_{m \ge 0}\, H^0 (\rho^*\cL^m)$. Then by the description of the Stein factorization as in \cite[Tag 03H0]{Stacks}, $X_{\tilde{E}}$ is the relative normalization of $X_E$ in $X'_{\tilde{E}}$, and $\pi \circ \varphi'$ is the Stein factorization of $\varphi \circ \pi$. Since $\pi$ is not \'etale, Proposition~\ref{prop:lasc_rel_normaliz_et} implies that $X_E$ is not locally algebraically simply connected.
		
		There are two notable examples where we can find such $\rho \colon \tilde{E} \to E$. The first example is when $\rho$ is a 2-isogeny between smooth elliptic curves, which is also mentioned in \cite[Example 3]{Sti17}. In this case, $X_E$ has elliptic singularity at the cone point $0$. The next example is when $E$ is an Enriques surface defined over $\mathbb{C}$ and $\tilde{E}$ is a corresponding K3 surface as an \'etale degree $2$ cover (this is not in \cite[Example 3]{Sti17}). In this case, $X_E$ has rational singularities because $H^0(R^i\varphi_*\OO_{X'_E}) \cong H^i(\oplus_{j \ge 0} \cL^j)=0$ whenever $i>0$ by the Hodge diamond of Enriques surfaces when $j=0$ and by the Kodaira vanishing theorem applied to ample line bundles $\cL^j \otimes \omega_E^{-1}$ when $j>0$.
	\end{example}

	If the exceptional locus $E$ of a given resolution $\varphi \colon Y \to X$ over an isolated singular point $p \in X$ is a smooth projective variety but $\pi_1^{\text{\'et}}(E) \neq 1$, then 
	arguments analogous to those in Example~\ref{ex:lc_not_lasc_general} implies $X$ is not locally algebraically simply connected. Theorem \ref{thm:asc_fib_lasc} below handles the case when $\pi_1^{\text{\'et}}(E) = 1$; this result will be useful later on in our study of stable axis-like curves. 
With this notion, we prove the following theorem, which is an application of \cite[Lemma 2]{Sti17}.
	
	\begin{theorem}\label{thm:asc_fib_lasc}
		Suppose that $X$ is a normal scheme locally of finite type. If $\varphi \colon Y \to X$ is a resolution of singularities such that $\varphi^{-1}(x)$ is algebraically simply connected for every closed point $x$ of $X$, then $X$ is locally algebraically simply connected.
	\end{theorem}

	\begin{proof}
		By Definition~\ref{def:lasc}, it suffices to consider a resolution of $\Spec \OO^h_{X,x}$ for every closed point $x$ of $X$ and then check that the \'etale fundamental group of the resulting regular scheme is trivial. 
		
		First, notice that $X$ is a normal scheme locally of finite type defined over an algebraically closed field of characteristic zero, so that $\OO^h_{X,x}$ is a normal Noetherian local ring by \cite[Tags 0CBM and 06LJ]{Stacks} and the map $\psi' \colon \Spec \OO^h_{X,x} \to \Spec \OO_{X,x}$ is formally \'etale;
		furthermore, $\psi'$ induces an isomorphism between the closed points. Since the induced map $j_x \colon \Spec \OO_{X,x} \to X$ is also formally \'etale, the composition $\psi \colonequals j_x \circ \psi'$ 
		is a formally \'etale morphism between normal locally Noetherian schemes. Note that the closed point of $\Spec \OO^h_{X,x}$ is mapped isomorphically to $x \in X$ by $\psi$, whose residue field is algebraically closed. 

		Let $x$ also refer to the closed point of $\Spec \OO^h_{X,x}$ and consider the pullback $\varphi_x \colon Y_x \colonequals Y \times_X \Spec \OO^h_{X,x} \to \Spec \OO^h_{X,x}$ of $\varphi$ by $\psi$; we claim that $\varphi_x$ is a resolution. Since $\varphi$ is a proper surjection, so is its pullback $\varphi_x$. Also, $\Spec \OO^h_{X,x}$ is irreducible if and only if $\Spec \OO_{X,x}$ is irreducible by \cite[Tag 06DH]{Stacks}. Since $\mathrm{Frac}(\OO^h_{X,x}) / \mathrm{Frac}(\OO_{X,x})$ is a separable field extension by definition, the generic point of $\Spec \OO^h_{X,x}$ maps to the generic point of $\Spec \OO_{X,x}$. Therefore, the image under $\psi$ of the generic point of $\Spec \OO^h_{X,x}$ is the generic point of $X$, so $\varphi_x$ is also birational. 
		Moreover, $Y_x$ is 
		regular by construction, so $\varphi_x$ is a resolution.
		
		It remains to prove $\pi_1^{\text{\'et}}(Y_x)=1$. Fix any finite \'etale cover $\pi \colon P \to Y_x$, and let $d \colonequals \deg \pi$ the degree of $\pi$. Then, there is a corresponding Stein factorization of $\varphi_x \circ \pi$ by \cite[Tag 03H0]{Stacks}
		\begin{equation}\label{eq:Stein_fac_et_resol}
			\xymatrix{
				P \ar[r]^{\varphi_x'} \ar[d]_{\pi} & P' \ar[d]^(.46){\pi'}\\
				Y_x \ar[r]_(.3){{\varphi_x}} & \Spec \OO^h_{X,x},
			}
		\end{equation}
		where $\pi'$ is finite. By construction, $\pi \colon P \to Y_x$ is also the normalized pullback of $\pi' \colon P' \to \Spec \OO^h_{X,x}$ under the resolution $\varphi_x \colon Y_x \to \Spec \OO^h_{X,x}$ because $\varphi_x$ is birational, $\pi'$ is generically \'etale, and $\pi$ is \'etale with $P$ regular. 
		
		Note that the second paragraph of the proof of \cite[Lemma 2]{Sti17}, which relies on \cite[\S 5.7]{MO15}, can be slightly modified by using the openness of the \'etale condition (as \cite[\S 5.7]{MO15} verifies \'etaleness at a point of the base). As a result, we obtain a slight extension of \cite[Lemma 2]{Sti17}: 
		$\pi'$ is \'etale if and only if $\pi$ is \'etale and the restriction $(\varphi_x \circ \pi)^{-1}(x) \to \varphi_x^{-1}(x)$ of $\pi$ is a trivial degree $d$ cover, where $x$ is the closed point of $\Spec \OO^h_{X,x}$ (which is also geometric by the convention that base field is algebraically closed). 
		
		We claim that the above modification of \cite[Lemma 2]{Sti17} implies that $\pi'$ is \'etale. First, $\varphi_x^{-1}(x) \cong \varphi^{-1}(x)$ is algebraically simply connected by assumption, so the restriction $(\varphi_x \circ \pi)^{-1}(x) \to \varphi_x^{-1}(x)$ of $\pi$ is a trivial \'etale cover. 
		Our modification of \cite[Lemma 2]{Sti17} then implies $\pi'$ is \'etale. Since $\pi_1^{\text{\'et}}(\OO^h_{X,x})=1$ by definition of $\OO^h_{X,x} \cong \OO^{sh}_{X,x}$ (because $\kappa(x)$ is algebraically closed), $\pi \colon P' \to \Spec \OO^h_{X,x}$ is also a trivial degree $d$ cover. As $\varphi_x' \colon P \to P'$ is proper with geometrically connected fibers by the construction of the Stein factorization in \eqref{eq:Stein_fac_et_resol}, $\pi_0(P)=\pi_0(P')=d$. Since $Y_x$ is irreducible because it is birational to the irreducible $\Spec \OO^h_{X,x}$, $\pi \colon P \to Y_x$ is also a trivial degree $d$ cover by using similar arguments as for $\widehat{\pi}'$.
%
	\end{proof}

Theorem~\ref{thm:asc_fib_lasc} provides an alternative criterion to Theorem~\ref{thm:klt_implies_lasc} where one may show $X$ is locally algebraically simply connected.


	\begin{example}[{Lasc lc singularity}]
	 \label{ex:lc_K3_lasc}
		In contrast to Example~\ref{ex:lc_not_lasc_general}, we construct an example of log canonical singularity that is locally algebraically simply connected. Consider a K3 surface $E$ defined over $\mathbb{C}$, with $\varphi \colon X'_E \to X_E$ as in Example~\ref{ex:lc_not_lasc_general}. By an analogous argument, $X_E$ has log canonical singularity at the cone point $0 \in X_E$. Since $\varphi^{-1}(0) \cong E$ is the zero section of the total space $X'_E$ of line bundle $\cL^\vee$ and $\pi_1(E)=1$, Theorem~\ref{thm:asc_fib_lasc} implies that $X_E$ is locally algebraically simply connected.
	\end{example}

	Thus, Example~\ref{ex:lc_not_lasc_general} implies log canonical, elliptic, and rational singularities are not locally algebraically simply connected in general. However, some log canonical singularities are locally algebraically simply connected by Example~\ref{ex:lc_K3_lasc}.

	\section{Extension theorem for stacks} \label{sec:extension_DMstacks}
	
	
	
	The goal of this section is to state and prove the extension theorem (Theorem~\ref{thm:extension-stacks}) for separated DM stacks, which is a main technical tool to extend a birational maps. 
	It applies to DM stack that are locally algebraically simply connected, which we now define: 
	
	
	\begin{definition} \label{def:lasc_DMstack}
		Let $\cX$ be a separated normal DM stack. Then, $\cX$ is {\em locally algebraically simply connected} if there exists an \'etale cover $X \to \cX$ by a scheme such that $X$ is a locally algebraically simply connected scheme.
	\end{definition}

	The following result shows that a separated normal DM stack $\cX$ is locally algebraically simply connected if and only if \emph{every} \'etale cover is locally algebraically simply connected.
	
	\begin{lemma}\label{l:well_def_lasc_sing}
		If $f \colon Y \to X$ is an \'etale cover of a normal scheme $X$, then $X$ is locally algebraically simply connected if and only if $Y$ is locally algebraically simply connected.
	\end{lemma}
	\begin{proof}
	Since $X$ and $Y$ are locally finitely presented over our characteristic $0$ base field $k = \overline{k}$, all residue fields of closed points of $X$ and $Y$ are isomorphic to $k$. Thus, for any closed point $y$ of $Y$, $\OO^h_{Y,y} \cong \OO^{sh}_{Y,y}$ and $\OO^h_{X,f(y)} \cong \OO^{sh}_{X,f(y)}$. Since \'etaleness of $f$ implies that $f^* \colon \OO^h_{Y,y} \xrightarrow{\cong} \OO^h_{X,x}$, and Definition~\ref{def:lasc} only depends on spectrums of henselization of local rings of closed points, the 
	result follows.
	\end{proof}

The following example shows that the Stacky Extension Theorem (Theorem~\ref{thm:extension-stacks}) is false without some assumptions on the singularities of $\cX$.

	\begin{example}[Counter-examples without lasc condition]
 \label{ex:lasc_extension}
		In Example~\ref{ex:lc_not_lasc_general}, we exhibited two examples of affine cones $X_E$ over projective varieties $E$ defined over $\mathbb{C}$ that are not locally algebraically simply connected at the cone point $0 \in X_E$. Furthermore, we have the following commutative diagram
		\[
			\xymatrix{
				X'_{\tilde{E}} \ar[r]^{\varphi'} \ar[d]_{\pi'} & X_{\tilde{E}} \ar[d]^{\pi}\\
				X'_E \ar[r]_{\varphi}  & X_E,
			}
		\]
		where $\pi'$ is a $\mu_2$-torsor, $\pi$ is a finite degree $2$ morphism that is not \'etale precisely at $0 \in X_E$, and $\varphi,\varphi'$ are resolutions of singularities.
		
		The above diagram can be reinterpreted as a morphism $\psi_U \colon U \colonequals X_E \setminus \{0\} \to B\mu_2$ and $\psi' \colon X'_E \to B\mu_2$ that agrees over $U$ via $\varphi$. Observe that Property \ref{extension:sets} of Theorem~\ref{thm:extension-stacks} is satisfied for $\psi_U$ because $|B\mu_2|$ is a single point set corresponding to a trivial $\mu_2$-cover over the base field. To see that property \ref{extension:DVRs} of Theorem~\ref{thm:extension-stacks} is satisfied, take any $h \colon \Spec R \to X_E$ as in property \ref{extension:DVRs}. Since $\varphi$ is proper and is an isomorphism over $U$, $h$ lifts to $h' \colon \Spec R \to X'_E$ by the valuative criterion for properness. As $\pi'$ is a $\mu_2$-torsor, so is the pullback of $\pi'$ via $h'$, which corresponds to a morphism $\overline{h} \colon \Spec R \to B\mu_2$ satisfying property \ref{extension:DVRs}.
		
		If the conclusion of the theorem were true for $\psi_U$, then $\psi_U$ would extend to $\psi \colon X_E \to B\mu_2$, which is the existence of a $\mu_2$-torsor over $X_E$ that is birational to $X_{\tilde E}$. Since this $\mu_2$-torsor and $X_{\tilde E}$ are both normal and finite over $X_E$, they must be isomorphic. However, $\pi$ has generic degree 2 while the support of $\pi^{-1}(0)$ is the single cone point of $X_{\tilde E}$, which contradicts that $\pi$ must be \'etale by Lemma~\ref{l:torsor-extension-over-DVR}. Hence, there is no extension $\psi \colon X_E \to B\mu_2$ of $\psi_U \colon U \to B\mu_2$. Hence, the locally algebraically simply condition cannot be removed from Theorem~\ref{thm:extension-stacks}.
	\end{example}

	\begin{proof}[Proof of Theorem~\ref{thm:extension-stacks}]
		By \cite[Proposition A.1]{FMN10}, if $g$ exists, then it is unique up to 2-isomorphism. 
		
		For the existence of $g$, We first reduce to the case where $\cX$ is a scheme. Since $\cX$ is a stack of finite type, choose an \'etale cover $\rho\colon X\to\cX$ with $X$ of finite type, and let $U=\rho^{-1}(\cU)$. Then $X$ is normal and $U\subset X$ is open dense. We obtain an induced map $|X|\to|\cX|\to|\cY|$; furthermore, since for every DVR $R$, an $R$-point of $X$ induces an $R$-point of $\cX$, property \ref{extension:DVRs} holds for $(U,X)$. Thus, assuming the result where the source is a scheme, there exists a unique map $X\to\cY$ extending $U\to\cY$. Since $\cX$ has proper unramified diagonal, the diagonal is also finite, thus $X\times_\cX X$ is a scheme; applying the same argument to $X\times_\cX X$, and using uniqueness, we have a map $\cX\to\cY$ by descent.

		We may therefore assume $\cX=X$ is a scheme. Since $\cY$ has finite diagonal, it has a coarse space $\cY\to Y$. By \citelist{\cite{GG14}*{Theorem 7.3} \cite{AET23}*{Lemma 3.18}}, we obtain a map $X\to Y$; although the statements of the results assume $X$ is proper, the proofs work when $X$ is only assumed to be seprated and locally of finite type. Now, by the uniqueness statement, it suffices to look \'etale locally on $Y$, so we may assume $\cY=[V/G]$ with $G$ a finite group and $V$ affine. Since $X$ is separated and locally of finite type, we can look Zariski locally on $X$ again by the uniqueness statement, so assume $X$ is also affine of finite type. By property \ref{extension:DVRs} and uniqueness, we may assume $X\setminus U$ has codimension at least $2$. Then $H^0(\OO_U)=H^0(\OO_X)$ since $X$ is normal. We have a $G$-torsor $p\colon P\to U$ and a $G$-equivariant map $P\to V$. Denoting by $i \colon U \to X$ the inclusion, let $Q$ be the relative normalization of $X$ in $P$ via $i \circ p \colon P \to X$; $Q$ is also normal by \cite[Tag 035L]{Stacks}. Since $X$ is a normal affine scheme of finite type over the base field by assumption, 
		the induced map $\pi \colon Q \to X$ is finite by \cite[Tag 035B, 03GR]{Stacks}. Finiteness of $\pi$ implies that $Q$ is affine because $X$ is affine as well. Then, the definition of relative normalization as in \cite[Tag 035H]{Stacks} implies that $H^0(\OO_Q)$ is the integral closure of $H^0(\OO_U)$ in $H^0(i_*p_*\OO_P) \cong H^0(\OO_P)$. Because $p \colon P \to U$ is a finite morphism between normal schemes, the integral closure of $H^0(\OO_X) \cong H^0(\OO_U)$ in $H^0(\OO_P)$ is $H^0(\OO_P)$ by \cite[Tag 035F]{Stacks}. Therefore, $Q \cong \Spec H^0(\OO_P)$, and the $G$-action on $P$ induces a $G$-action on $Q$. Furthermore, since $V$ is affine, the map $P\to V$ extends to a $G$-equivariant map $Q\to V$. Also note that since $p \colon P \to U$ is a $G$-torsor, $H^0(\OO_Q)^G \cong H^0(\OO_P)^G \cong H^0(\OO_U) \cong H^0(\OO_X)$ and so $X \cong Q/G$.
		
		Since $X$ is normal, we may shrink $U$ without changing the codimension of $X \setminus U$ and assume $U$ is contained in the smooth locus of $X$. Choose any closed point $x$ of $X \setminus U$. By assumption on the base field, the residue field $\kappa(x)$ is algebraically closed of characteristic zero. Fix a resolution $\varphi: X' \to X$ of singularities of $X$. Observe that $U$ is an open subscheme of the smooth locus of $X$, which is of finite type over the base field, and $\pi$ restricted to $U$ is a $G$-torsor, then the relative normalization of $X'$ in $P$ via $P \xrightarrow{p} U \hookrightarrow X'$ is a finite morphism $\pi'\colon Q' \to X'$ by \cite[Tag 035B, 03GR]{Stacks}, equipped with an induced $G$-equivariant map $Q' \to V$. Since $P$ is smooth, $Q'$ is also normal by \cite[Tag 035L]{Stacks}.
		
		To see that $\pi'$ is \'etale, first choose any codimension $1$ point $x'$ of $X'$. Then, using the induced morphism $i_{D'} \colon D' \to \Spec \OO_{X',x'} \hookrightarrow X'$, consider the pullback $\pi'_{D'} \colon Q'_{D'} \rightarrow D'$ of $\pi$. Notice that property \ref{extension:DVRs} implies that there is a $G$-torsor $P_{D'} \to D'$ whose generic fiber is isomorphic to that of $\pi'_{D'}$. Thus, Corollary~\ref{cor:rel_normal_DVR} implies that $Q'_{D'} \cong P_{D'}$ so that $\pi'_{D'}$ is finite \'etale, because $\pi' \colon Q' \to X'$ is finite, $X'$ is regular, and $Q'$ is normal. Since $x'$ runs over arbitrary codimension $1$ points of $X'$, if $\pi'$ is not \'etale over $X'$ at $x'$, then it contradicts the above conclusion that $\pi'_{D'}$ is a $G$-torsor. Hence, $\pi'$ is a $G$-torsor away from codimension $2$, so $\pi'$ is \'etale by purity of the branch locus \cite[Exp. X, Theorem 3.1]{Gro63SGAI}.
		
		Now, taking the Stein factorization as in \cite[Tag 03H0]{Stacks} of $\varphi \circ \pi'$, we obtain $Q' \to Q'' \to X$ with $Q''$ the relative normalization of $X$ in $Q'$; in this case, $Q'' \cong \Spec_X (\varphi \circ \pi')_*\OO_{Q'}$. Since $Q'$ is the relative normalization of $X'$ in the normal scheme $P$, $(\pi')_*\OO_{Q'}$ is the integral closure of $\OO_{X'}$ in $(i' \circ p)_*\OO_P$, where $i' \colon U \hookrightarrow X'$ is the induced open inclusion. Because $\varphi \colon X' \to X$ is a proper morphism of normal schemes of finite type which is an isomorphism over an open dense subscheme $U$ of $X$, $\varphi_*\OO_{X'} \cong \OO_X$ by the proof of Zariski's Main Theorem as in \cite[III. Corollary 11.4]{Har77}. Hence, the integral closure of $\varphi_*\OO_{X'} \cong \OO_X$ in $(\varphi \circ i' \circ p)_*\OO_P$ is $(\varphi \circ \pi')_*\OO_{Q'}$, so $Q'' \cong \Spec_X (\varphi \circ \pi')_*\OO_{Q'}$ must be the relative normalization of $X$ in $P$ as well; then $Q \cong Q''$ because $Q$ is the relative normalization of $X$ in $P$ by above construction. Denote $\mu \colon Q' \to Q$ to be the associated Stein morphism so that $\varphi \circ \pi' = \pi \circ \mu$.
		
		Since $X$ is locally algebraically simply connected, Proposition~\ref{prop:lasc_rel_normaliz_et} applied to the resolution $\varphi \colon X' \to X$ implies that $\pi \colon Q \to X$ is also finite \'etale of degree $d$, and so $\pi' \colon Q' \to X'$ is the pullback of $\pi$ by $\varphi$. Since $\pi'$ is also a $G$-torsor, $\pi$ must be a $G$-torsor as well by Lemma~\ref{l:lasc_et}. Using the $G$-equivariant map $Q \to V$ from above, we have obtained the desired extension $g \colon X \to \cY = [V/G]$ from the map $f \colon U \to [V/G]$ corresponding to the $G$-torsor $p \colon P \to U$.
	\end{proof}

	\section{Moduli stack of stable axis-like curves} \label{sec:Smyth_moduli}
	
	\subsection{Smyth's $\cZ$-stability for curves}
	\label{subsec:Smyth_stability}
	
	In this subsection, we recall Smyth's construction of the moduli stack $\overline{\cM}_{g,n}(\cZ)$ of $\cZ$-stable curves in \cite[\S 1]{Smy13}. 
	Since we only consider the DM case, assume from now on that the pair $(g,n)$ satisfies $2g-2+n >0$ so that $\overline{\cM}_{g,n}$ is a DM stack. 
	
	\medskip
	
	Recall that an $n$-pointed curve $(C;p_1,\dotsc,p_n)$ is a curve $C$ with $n$ distinct points in the smooth locus; this convention is more restrictive than \cite[\S 1]{Smy13}. An $n$-pointed curve $(C;p_1,\dotsc,p_n)$ is {\it prestable} if for each irreducible component $C_i^\nu$ of the normalization $C^\nu$ of $C$, the number of distinguished points (which consists of the support of the marked points and the preimages of singular locus $C_{\mathrm{sing}}$ of $C$ in $C_i^\nu$) in $C_i^\nu$ is at least three if $g(C_i^\nu)=0$ and is at least one if $g(C_1^\nu)=1$. Note that if $C$ is nodal and each $p_i$ is distinct and lie in the smooth locus $C_{\mathrm{sm}}$, then $C$ being prestable is equivalent to $C$ being stable. In addition, \cite[Remark 1.4]{Smy13} implies that the automorphism group of a prestable $n$-pointed curve is always finite.
	
	To define the stability condition, we first need some notation. Let $G_{g,n}$ be the set of all stable $n$-pointed dual graphs of genus $g$ (i.e., each graph is isomorphic to a dual graph of a stable $n$-pointed curve of genus $g$); as a convention, each marked point $p_i$ of a stable $n$-pointed curve $(C;p_1,\dotsc,p_n)$ corresponds to a labelled half-edge of the corresponding dual graph $\Gamma_{(C;p_1,\dotsc,p_n)} \in G_{g,n}$. We will often abbreviate $\Gamma_{(C;p_1,\dotsc,p_n)}$ as $\Gamma_C$ if there is no confusion.
	
	Given $\Gamma_1,\Gamma_2 \in G_{g,n}$, a {\it degeneration} $\Gamma_1 \rightsquigarrow \Gamma_2$ (called {\it specialization} in \cite[\S 1]{Smy13}) is induced by a flat family $\cC \to \Spec R$ of stable $n$-pointed curves defined over the spectrum of a DVR $R$; specifically, the dual graph of the geometric generic fiber is $\Gamma_1$ and the dual graph of the geometric special fiber is $\Gamma_2$. In fact, one need only replace $\Spec R$ with a suitable finite cyclic base change (rather than pass to geometric fibers) to obtain $\Gamma_1$ (resp., $\Gamma_2$) as the dual graphs of the generic (resp., special) fiber. In this case, each vertex $v \in V(\Gamma_1)$ corresponds to an irreducible component $C_v$ of the generic fiber of $\cC \to \Spec R$. Then, $v \rightsquigarrow M_v$ means that $M_v$ is a complete subgraph of $\Gamma_2$ which is the dual graph of the limit of $C_v$ (i.e., the restriction, to the special fiber, of closure of the $C_v$ in $\cC$).
	
	Now we are ready to recall $\cZ$-stability from \cite[\S 1]{Smy13}:
	
	\begin{definition}[\cite{Smy13}*{Definition 1.5}] \label{def:extremal_assignment}
		Fix $(g,n)$ and consider an enumeration $\Gamma_1,\dotsc,\Gamma_N$ of $G_{g,n}$. Then, an assignment $\cZ$ associating to each $\Gamma_i$ a subset $\cZ(\Gamma_i) \subset V(\Gamma_i)$ is {\it extremal} if
		\begin{enumerate}
			\item for each $i$, $\cZ(\Gamma_i)$ is an $\Aut(\Gamma_i)$-invariant proper subset of $V(\Gamma_i)$, and
			\item for each degeneration $\Gamma_j \rightsquigarrow \Gamma_l$ and each $v \in V(\Gamma_j)$, we have $v \in \cZ(\Gamma_j)$ if and only if $V(M_v) \subset \cZ(\Gamma_l)$ where $v \rightsquigarrow M_v$ is induced by $\Gamma_j \rightsquigarrow \Gamma_l$.
		\end{enumerate}
		If $(C;p_1,\dotsc,p_n)$ is a stable $n$-pointed genus $g$ curve, then $\cZ(C)$ is the reduced closed subscheme of $C$, which is the union of irreducible components of $C$ corresponding to points of $\cZ(\Gamma_C)$. 
	\end{definition}

	\begin{definition}[\cite{Smy13}*{Definition 1.8}] \label{def:Z-stability}
		Given an extremal assignment $\cZ$ on $G_{g,n}$, a smoothable $n$-pointed curve $(C;p_1,\dotsc,p_n)$ of genus $g$ is $\cZ$-stable if there exists a stable curve $(C^s;p_1^s,\dotsc,p_n^s)$ and a map $\phi \colon C^s \to C$ such that
		\begin{enumerate}
			\item $\phi$ is surjective with connected fibers,
			\item $\phi$ restricted to $C^s \setminus \cZ(C^s)$ is an isomorphism onto its image, and
			\item if $Z_1,\dotsc,Z_l$ is the list of all connected components of $\cZ(C^s)$, then $\phi(Z_j) \in C$ is a singular point of genus $g_j = p_a(Z_j)$ with $m_j = \#\left(Z_j \cap \overline{C^s \setminus Z_j}\right)$ for each $j$. In this case, $\phi(Z_j) \in C_{\mathrm{sing}}$ is a singularity of type $(g_j,m_j)$, and $(Z_j,P_j \cup (Z_j \cap \overline{C^s \setminus Z_j}))$ is a stable $(m_j + n_j)$-pointed curve of genus $g_j$, where $P_j \colonequals Z_j \cap \{p_1,\dotsc,p_n\}$ is the set of marked points in $Z_j$ and $n_j \colonequals \# P_j$.
		\end{enumerate}
	\end{definition}
	
	Note that for the dual graph $\Gamma$ of a smooth $n$-pointed curve of genus $g$, $\cZ(\Gamma)=\emptyset$ for any extremal assignment $\cZ$ by Definition~\ref{def:extremal_assignment}; thus, any smooth $n$-pointed curve of genus $g$, where $(g,n) \neq (0,0),(0,1),(0,2),(1,0)$ is automatically $\cZ$-stable by Definition~\ref{def:Z-stability}. 
	Moreover, the notion of $\cZ$-stability is only defined for smoothable pointed curves. 
	Throughout the paper, we do not consider non-smoothable curves.
	
	Smyth in \cite[Theorem 1.9]{Smy13} characterized every integral proper DM compactification of $\cM_{g,n}$ (i.e, proper DM compactifications without embedded points) by using Definition~\ref{def:Z-stability}. To explain this, define $\mathcal{V}_{g,n}$ be the unique irreducible component of the algebraic stack $\mathfrak{M}_{g,n}$ (called $\mathcal{U}_{g,n}$ in \cite[Corollary B.4]{Smy13}) of all (flat families of) $n$-pointed genus $g$ curves containing $\cM_{g,n}$; by definition, $\mathcal{V}_{g,n}$ is an integral algebraic stack. This means that given a scheme $T$, $\mathcal{V}_{g,n}(T)$ is a subgroupoid of the groupoid $\mathfrak{M}_{g,n}(T)$ of flat families of $n$-pointed genus $g$ curves. When the image of $T \to \mathfrak{M}_{g,n}$ in $|\mathfrak{M}_{g,n}|$ consists of curves in $|\mathcal{V}_{g,n}|$ with reduced versal deformation spaces, then by definition, this map factors through $\mathcal{V}_{g,n}$; when $T$ is reduced, showing that the image is in $|\mathcal{V}_{g,n}|$ is enough to guarantee that the map factors through $\mathcal{V}_{g,n}$. Theorem 1.9 of \cite{Smy13} shows that the $\overline{\cM}_{g,n}(\cZ)$ as in the following definition are all of the integral proper DM compactifications of $\cM_{g,n}$.
	
	\begin{definition} \label{def:stack_Z-stable}
		Given an extremal assignment $\cZ$ of $G_{g,n}$, let $\overline{\cM}_{g,n}(\cZ)$ be the open substack of $\mathcal{V}_{g,n}$ containing $\cM_{g,n}$ as an open dense substack such that $\overline{\cM}_{g,n}(\cZ)$ is proper and $|\overline{\cM}_{g,n}(\cZ)|$ consists of $\cZ$-stable curves. In other words, $\overline{\cM}_{g,n}(\cZ)$ is an integral proper DM stack that is also the reduced substack of the moduli stack of $\cZ$-stable curves.
	\end{definition}
	
	\begin{remark} \label{rmk:Smyth_moduli}
		Given an extremal assignment $\cZ$ of $G_{g,n}$, the moduli stack $\overline{\cM}_{g,n}(\cZ)$ from Definition~\ref{def:stack_Z-stable} can be thought of as the reduction of a moduli stack $\overline{\cM}_{g,n}(\cZ)'$ parameterizing flat families of $n$-pointed $\cZ$-stable curves (so smoothable) of genus $g$; so $\overline{\cM}_{g,n}(\cZ)'$ is a substack of $\mathfrak{M}_{g,n}$.
		In particular, $|\overline{\cM}_{g,n}(\cZ)|$ coincides with $|\overline{\cM}_{g,n}(\cZ)'|$ as topological subspaces of $|\mathfrak{M}_{g,n}|$, and the structure sheaf $\mathcal{O}_{\overline{\cM}_{g,n}(\cZ)}$ on  $|\overline{\cM}_{g,n}(\cZ)|$ is the quotient of $\mathcal{O}_{\overline{\cM}_{g,n}(\cZ)'}$ by its nilradical.
		When every $\cZ$-stable curves admits integral versal deformation space, then $\overline{\cM}_{g,n}(\cZ) = \overline{\cM}_{g,n}(\cZ)'$.
	\end{remark}
	
	\subsection{Axis-like curves}
	\label{subsec:stable_axis-like}

	In this subsection, we define separating axis-like stability condition $\sF$, leading to the moduli stack $\overline{\cM}_{g,n}^{\mathrm{s-axis}} \colonequals \overline{\cM}_{g,n}(\sF)$ of stable separating axis-like curves, see Proposition \ref{prop:moduli_separating_axis-like}. Our main goal is to show there exists a representable proper birational morphism $F \colon \overline{\cM}_{g,n} \to \overline{\cM}_{g,n}(\sF)$ where we explicitly describe the contracted locus, see Proposition \ref{prop:forget_DM_stable_to_stable_axis-like}. We also prove that $F$ is a (quasi-)resolution (i.e., resolution for stacks that's representable but not necessarily schematic), see Corollary \ref{cor:moduli_FL_lasc}.

	\medskip
	
	
	
	Recall from Definition~\ref{def:axis_sing} that $m$-axis singularities locally look like the axes in $m$-space; they are called rational $m$-fold points in \cite[Definition 1.16]{Smy13}. Note that $2$-axis points are nodes, and $m$-axis points are exactly singularities of type $(0,m)$ by \cite[Lemma 1.17(a)]{Smy13}. Moreover, any $m$-axis point is smoothable by \cite[Lemma 1.17(b)]{Smy13}. This implies that a single $m$-axis point comes from contracting a subcurve $D$ of a stable curve $C^s$ such that (i) $D$ is a rational tree, and (ii) $\# (D \cap \overline{C^s \setminus D})=m$. Motivated by this observation, we define the following, which generalizes the notion of {\it rational (chain) backbones} from \cite[\S 5.A]{HM98}:
	
	\begin{definition} \label{def:ratl_bridge}
		A subcurve $D$ of a prestable $n$-pointed curve $(C;p_1,\dotsc,p_n)$ is a {\it rational $m$-bridge} if $D$ does not contain any of the marked points $p_i$, and the tuple $(D;\Supp (D \cap \overline{C \setminus D}))$ is a stable $m$-pointed curve of genus zero. Furthermore, such a $D$ is {\it separating} if the number of connected components of $\overline{C \setminus D}$ is equal to the number of points in $\Supp (D \cap \overline{C \setminus D})$.
	\end{definition}

	From now on, {\it rational multibridge} means rational $m$-bridge for some $m \ge 1$.
	
	\begin{definition} \label{def:separating_axis_sing}
		If $m \ge 3$, then an $m$-axis point $p$ of a curve $C$ is {\it separating} if the normalization of $C$ at $p$ has exactly $m$ connected components.
	\end{definition}

	The following definition is the pointed analog of Definition~\ref{def:axis_sing}, which is consistent when $n=0$:
	
	\begin{definition} \label{def:axis-like_curves}
		A prestable $n$-pointed curve $(C;\sigma_1,\dotsc,\sigma_n)$ is {\it axis-like} if each singular point $p$ of $C$ is an $m_p$-axis point for some integer $m_p \ge 2$.
	\end{definition}

	\begin{definition} \label{def:separating_axis-like}
		An $n$-pointed axis-like curve $(C;p_1,\dotsc,p_n)$ is {\it stable separating} if 
		\begin{enumerate}
			\item every $m$-axis point of $C$ is separating whenever $m \ge 3$ and avoids marked points $p_1,\dotsc,p_n$, and
			\item $C$ has no separating rational $m$-bridges for every positive integer $m$.
		\end{enumerate}
	\end{definition}

	By drawing an analogy between stable curves and their dual graphs, we define the following: given a dual graph $G \in G_{g,n}$, a {\it separating rational $m$-bridge} is a complete subtree $H \subset G$ of $m$-pointed genus zero (i.e., treat edges connecting $H$ and $G \setminus H$ as half edges of $H$, then $H \in G_{0,m}$) such that the number of connected components of $G \setminus H$ is equal to $m$. Using this, we first show that stable separating axis-like curves form a proper DM stack by finding an appropriate stability condition:
	
	\begin{lemma} \label{l:stability_separating_axis-like}
		Let $\sF$ be an assignment on $G_{g,n}$ where $\sF(\Gamma) \subset \Gamma$ is the set of separating rational $m$-bridges of $\Gamma \in G_{g,n}$ for any $m$. Then, $\sF$ is an extremal assignment.
	\end{lemma}
	
	\begin{proof}
		By definition, the set of separating rational bridges is $\Aut(\Gamma)$-invariant, which verifies the first condition of Definition~\ref{def:extremal_assignment}.
		
		It remains to check the second condition of Definition~\ref{def:extremal_assignment}. Consider any degeneration $\Gamma_0 \rightsquigarrow \Gamma_1$. 
		As noted earlier, we can assume we have a flat family $\pi \colon \cC \to \Spec R$ of stable $n$-pointed genus $g$ curves over a DVR $R$ such that the dual graph of the generic (resp., special) fiber of $\pi$ is $\Gamma_0$ (resp., $\Gamma_1$). Moreover, for any vertex $v \in V(\Gamma_0)$, the induced degeneration $v \rightsquigarrow M_v$ corresponds to the connected flat subfamily of $\pi$ whose generic fiber corresponds to the irreducible component $\cC_v$ of the generic fiber of $\pi$; in this case, the restriction of the closure $\overline{\cC_v}$ to the special fiber is the connected subcurve $C_{M_v}$ whose dual graph is $M_v$.
		
		Now take any $v \in V(\Gamma_0)$. If $v \in \sF(\Gamma_0)$, then the genus of $v$ is zero, has no half-edges, and the number $e_v \ge 3$ of edges at $v$ is equal to the number of connected components of $\Gamma_0 \setminus v$. Also, the subscheme $\cC_{\Gamma_0 \setminus v}$ corresponding to the dual graph $\Gamma_0 \setminus v$ has exactly $e_v$ connected components; as $\overline{\cC}_{\Gamma_0 \setminus v}$ is flat over $\Spec R$, the number of connected components of the corresponding special fiber $C_{\Gamma_1 \setminus M_v}$ (whose dual graph is indeed $\Gamma_1 \setminus M_v$) is $e_v$ as well. In this regard, $(\overline{\cC}_v, \overline{\cC}_v \cap \overline{\cC}_{\Gamma_0 \setminus v})$ is a flat family of stable $e_v$-pointed curves of genus zero, whose complement in $\cC$ is flat over $\Spec R$ with $e_v$ connected components, hence $(\overline{\cC}_v, \overline{\cC}_v \cap \overline{\cC}_{\Gamma_0 \setminus v})$ is a flat family of separating rational $e_v$-bridges in $\cC$. So, $M_v \subset \Gamma_1$ is also a separating rational $e_v$-bridge. As a result, $M_v \subset \sF(\Gamma_1)$.
		
		Conversely, if $M_v \subset \sF(\Gamma_1)$, then a similar argument implies that $M_v$ is a separating rational $e_v$-bridge and $(\overline{\cC}_v, \overline{\cC}_v \cap \overline{\cC}_{G_0 \setminus v})$ is a flat family of separating rational $e_v$-bridges in $\cC$, so $v$ is a separating rational $e_v$-bridge in $\Gamma_0$; this implies that $v \in \sF(\Gamma_0)$.
	\end{proof}
	
	\begin{remark}\label{rmk:deformation_space_reduced?}
		To establish the normality of the stack of stable (pointed) separating axis-like curves in the next proposition, we must consider the local properties of these curves.
		Polishchuk and Rains showed in \cite[Theorem B]{PR24} that the versal deformation space of every $m$-axis singularity ($m$-fold singularity in loc. cit.) is normal (and Cohen--Macaulay).
		Hence, the versal deformation spaces of prestable axis-like curves (e.g., stable (pointed) separating axis-like curves) are normal because they are formally smooth over the product of the versal deformation spaces of $m$-axis singularities for various values of $m$ (cf. \cite[Lemma 2.1]{Smy11b}).
	\end{remark}
	
	\begin{proposition} \label{prop:moduli_separating_axis-like}
		$\overline{\cM}_{g,n}(\sF)$ is a normal proper DM moduli stack, whose points are stable $n$-pointed separating axis-like curves of genus $g$.
	\end{proposition}

	\begin{proof}
		Since $\sF$ is an extremal assignment by Lemma~\ref{l:stability_separating_axis-like}, $\overline{\cM}_{g,n}(\sF)$ from Definition~\ref{def:stack_Z-stable} is a proper DM stack. 
		If $\overline{\cM}_{g,n}(\sF)$ is the moduli stack of stable separating axis-like curves, then $\overline{\cM}_{g,n}(\sF)$ is normal by Remark~\ref{rmk:deformation_space_reduced?}.
		So it remains to show that a prestable $n$-pointed curve $(C;p_1,\dotsc,p_n)$ is stable separating axis-like if and only if $C$ is $\sF$-stable.
		
		Suppose that $(C;p_1,\dotsc,p_n)$ is a stable separating axis-like curve defined over an algebraically closed field $k$. Since $m$-axis singularities are smoothable by \cite[Lemma 1.17(b)]{Smy13}, there exists a smoothing $\cC \to \Spec R$ of $C$ with the sections $\varphi_i \colon \Spec R \to \cC$ of $\cC/\Spec R$ that lift the marked points $p_i \in C$; in other words, $R$ is a DVR and $\cC/\Spec R$ is a flat family of prestable curves whose generic fiber is a smooth $n$-pointed curve and special fiber is isomorphic to $C$. Let $(\widetilde{\cC}/\Spec \widetilde{R};\widetilde{\varphi}_1,\dotsc,\widetilde{\varphi}_n)$ be the stable reduction of the flat family $(\cC/\Spec R; \varphi_1,\dotsc, \varphi_n)$ of prestable $n$-pointed curves
		where $\Spec \widetilde{R} \to \Spec R$ is finite generically \'etale. As $(C;p_1,\dotsc,p_n)$ is prestable by assumption, the induced rational map $\widetilde{\cC}_0 \dashrightarrow C$ of special fibers is a proper surjective morphism. In fact, this morphism $\phi \colon \widetilde{\cC}_0 \to C$ is a contraction of a pure dimension one subscheme $D \subset \widetilde{\cC}_0$. Since an $m$-axis singularity is a singularity of type $(0,m)$ by \cite[Lemma 1.17(a)]{Smy13}, $D$ is a disjoint union of rational subcurves of $\widetilde{\cC}_0$. Because the normalization of $C$ at every $m$-axis points with $m \ge 3$ is isomorphic to $\overline{\widetilde{\cC}_0 \setminus D}$ and $C$ is separating axis-like, $D$ is a disjoint union of separating rational bridges. Since $C$ does not contain any separating rational bridges by assumption, $D$ is exactly the disjoint union of all separating rational bridges of $\widetilde{\cC}_0$. Hence, $D=\sF(\widetilde{\cC}_0)$ and the contraction $\phi \colon \widetilde{\cC}_0 \to C$ of $D$ implies that $C$ is $\sF$-stable.
		
		Suppose instead that $C$ is $\sF$-stable. By Definition~\ref{def:Z-stability}, there is a surjection $\phi \colon C^s \to C$ from a stable $n$-pointed genus $g$ curve $C^s$ such that $\phi$ is the contraction of $\sF(C^s)$. Since $\sF(C^s)$ is exactly the disjoint union of all separating rational bridges of $C^s$, $C$ does not contain any separating rational bridges. Furthermore, non-nodal singularities of $C$ are of type $(0,m)$ for some $m \ge 3$ by the construction of $\phi$,
		hence an $m$-axis singularity by \cite[Lemma 1.17(a)]{Smy13}. 
		Because $C^\nu \cong \overline{C^s \setminus \sF(C^s)}$ and $\sF(C^s)$ is the disjoint union of separating rational bridges, every non-nodal singular point of $C$ is a separating $m$-axis points for some $m \ge 3$. Thus, $C$ is a stable separating axis-like curve.
	\end{proof}

	\begin{example}\label{ex:moduli_stable_sep_axis_like_genus_3}
		When $(g,n)=(3,0)$, $\sF(C)$ is nonempty for every stable curve $C$ of genus three if and only if $C$ is of the form $R \cup E_1 \cup E_2 \cup E_3$, where $R$ is a stable separating rational $3$-bridge of $C$ and $E_i$'s are elliptic tails for every $i$. This choice of extremal assignment first appeared in \cite[Example A.3]{Smy13}, where he showed that $\overline{\cM}_3(\sF)$ parameterizes stable separating axis-like curves of genus three.
	\end{example}

	Notice that $\overline{\cM}_{g,n}(\sF)$ is a proper birational model of $\overline{\cM}_{g,n}$. The geometry of $\overline{\cM}_{g,n}(\sF)$ can be described from that of $\overline{\cM}_{g,n}$ by constructing a forgetful map $F \colon \overline{\cM}_{g,n} \to \overline{\cM}_{g,n}(\sF)$ as an application of the Stacky Extension Theorem (Theorem \ref{thm:extension-stacks}):

	\begin{proposition} \label{prop:forget_DM_stable_to_stable_axis-like}
		There exists a representable proper birational morphism $F \colon \overline{\cM}_{g,n} \to \overline{\cM}_{g,n}(\sF)$, where $F(C)$ is the contraction of all separating rational bridges of a given stable $n$-pointed curve $C$ of genus $g$.
	\end{proposition}

	\begin{proof}
		Notice that the identity map on $\cM_{g,n}$ defines a rational map $F' \colon \overline{\cM}_{g,n} \dashrightarrow \overline{\cM}_{g,n}(\sF)$. Since $\overline{\cM}_{g,n}$ is a smooth proper DM stack, 
		it is also locally algebraically simply connected by Definition~\ref{def:lasc_DMstack}. Then, by Theorem~\ref{thm:extension-stacks}, well-definedness of $F$ follows from the construction of a set-theoretic map $|F| \colon |\overline{\cM}_{g,n}| \to |\overline{\cM}_{g,n}(\sF)|$ that agrees with the assertion of the Proposition and satisfying \ref{extension:DVRs} of Theorem~\ref{thm:extension-stacks}.
		
		First, consider any stable $n$-pointed curve $(C;p_1,\dotsc,p_n)$ defined over an algebraically closed field of characteristic zero. Let $Z_1,\dotsc,Z_l$ be the enumeration of every connected component of $\sF(C)$; then for each $i$, $Z_i$ is a separating rational $m_i$-bridge with $m_i \ge 3$. Since $m$-axis singularities are seminormal by \cite[Corollary 1(2)]{Dav78}, axis-like curves are characterized by their normalizations and gluing data by \cite[Tag 0EUS]{Stacks}, i.e., the local neighborhood of an $m$-axis point $t$ of $E$ is characterized by the normalization $E'$ of $E$ at $t$ and the preimage $\{s_1,\dotsc,s_m\}$ of $t$ in $E'$. Denoting 
		\[
			\{q_{i,1},\dotsc,q_{i,m_i}\} \colonequals \overline{C \setminus Z_i} \cap Z_i
		\]
		for each $i$, the previous sentence implies that the data, consisting of $\overline{C \setminus \sF(C)}$ and $\{q_{1,1},\dotsc,q_{1,m_i}\}$, $\dotsc$, $\{q_{l,1},\dotsc,q_{l,m_l}\}$, uniquely characterize a stable separating axis-like curve $C'$. Moreover, the induced map $C \to C'$ is the contraction of $\sF(C)$ which is the locus of all separating rational multi-bridges. Define $|F|(C) \colonequals C'$, which agrees with the desired assertion on the level of set-theoretic maps.
		
		To verify property \ref{extension:DVRs} of Theorem~\ref{thm:extension-stacks} for $F'$ and $|F|$, fix any DVR $R$ with maximal ideal $\fmm$ and the residue field $k$. Fix any one-parameter family $(\pi \colon \cC \to \Spec R;\sigma_1,\dotsc,\sigma_n)$ of stable $n$-pointed curves over $\Spec R$ whose generic fiber is a smooth $n$-pointed curve of genus $g$. Let $Z_1,\dotsc,Z_l$ be the enumeration of connected components of $\sF(\cC_0)$, where $\cC_0$ is the special fiber of $\pi$. Since $\cC$ is normal and $Z_i$ is a rational $m_i$-bridge with $m_i \ge 3$ for every $i$, $K_\pi+\sum \frac{m_i-2}{m_i} Z_i$ is a big and $\pi$-nef $\QQ$-Cartier divisor on $\cC$ which is $\pi$-ample away from $\sF(\cC_0)$ and trivial on $\sF(\cC_0)$. Let $\widetilde{\cC}/\Spec R$ be the relative log canonical model of $(\cC, \sum \frac{m_i-2}{m_i} Z_i)/\Spec R$. By construction, the generic fibers of $\widetilde{\cC}$ and $\cC$ are isomorphic and the induced $\Spec R$-morphism $\cC \to \widetilde{\cC}$ is exactly the contraction of $\sF(\cC_0)$. By the previous paragraph, the special fiber $\widetilde{\cC}_0$ of $\widetilde{\cC}$ must be isomorphic to $|F|(\cC_0)$. Hence, property \ref{extension:DVRs} of Theorem~\ref{thm:extension-stacks} is satisfied. Combining everything from above, well-definedness of $F$ follows from Theorem~\ref{thm:extension-stacks} applied to $F'$ and $|F|$.
		
		For properties of $F$, observe that $\overline{\cM}_{g,n}$ is a smooth proper DM stack, 
		so that it is regular. Moreover, $F$ is an isomorphism over the open locus of $\overline{\cM}_{g,n}(\sF)$ parameterizing stable nodal curves; the openness follows from the fact that nodes cannot deform to $m$-axis points with $m \ge 3$ by looking at the versal deformation spaces. Because $\overline{\cM}_{g,n}$ is proper and $\overline{\cM}_{g,n}(\sF)$ is separated by Proposition~\ref{prop:moduli_separating_axis-like}, $F$ is proper by \cite[Tag 0CPT]{Stacks}.
		
		It remains to show representability of $F$. Using the notation from the second paragraph of this proof, it suffices to show that $F$ induces an inclusion of stabilizer groups $\Aut(C) \hookrightarrow \Aut(F(C))$ for every geometric point $C$ of $\overline{\cM}_{g,n}$ by \cite[Tag 04Y5]{Stacks}. Notice that for every $g \in \Aut(C)$, $g(\sF(C))=\sF(C)$ by Definition~\ref{def:extremal_assignment}. Thus, $\Aut(C)$ is a subgroup of $G'$, where $G'$ is a subgroup of $\Aut(\overline{C \setminus \sF(C)})$ such that for every $i$, $\{q_{i,1},\dotsc,q_{i,m_i}\}$ is mapped to $\{q_{j,1},\dotsc,q_{j,m_j}\}$ for some $j$ with $m_i=m_j$. By the characterization of seminormal singularities as in \cite[Tag 0EUS]{Stacks}, 
		$G'$ is exactly $\Aut(F(C))$. Therefore, $F$ is representable.
	\end{proof}

	\begin{remark}\label{rmk:extension_universal_family}
		In fact, a small modification of the proof of Proposition~\ref{prop:forget_DM_stable_to_stable_axis-like} implies $F$ lifts to a morphism $F^u \colon \overline{\cC}_{g,n} \to \overline{\cC}_{g,n}(\sF)$ between the universal curves $u \colon \overline{\cC}_{g,n} \to \overline{\cM}_{g,n}$ and $u_{\sF} \colon \overline{\cC}_{g,n}(\sF) \to \overline{\cM}_{g,n}(\sF)$. 
		Here, $F^u$ sends any pair $((C^s;\sigma_1,\dotsc,\sigma_n),q)$ of a stable $n$-pointed curve $(C^s;\sigma_1^s,\dotsc,\sigma_n^s)$ of genus $g$ and a point $q \in C^s$ to the corresponding stable separating axis-like curve $(C;\sigma_1,\dotsc,\sigma_n)$ and a point $q' \in C$; the two curves are related by the $\sF(C^s)$-contraction $\phi \colon C^s \to C$ and $q' = \phi(q)$ as in Definition~\ref{def:Z-stability}.
		
		To see why this claim is true, let us explain the salient parts of the proof, in analogy with the proof of Proposition~\ref{prop:forget_DM_stable_to_stable_axis-like}. 
		Recall that $u \colon \overline{\cC}_{g,n} \cong \overline{\cM}_{g,n+1} \to \overline{\cM}_{g,n}$ is exactly the $(n+1)$th projection morphism: 
		given $(n+1)$-pointed stable curve $(C;\sigma_1,\dotsc,\sigma_{n+1})$ of genus $g$, forget the $(n+1)$th marked point $\sigma_{n+1}$ and then take the stable model of semistable $(C;\sigma_1,\dotsc,\sigma_n)$, which coincides with contracting the irreducible rational component containing $\sigma_{n+1}$ that has at most two other special points (if such a component exists).
		Thus, $\overline{\cC}_{g,n}$ is smooth proper and locally algebraically simply connected. 
		
		Since $u_{\sF}$ is a proper morphism and $\overline{\cC}_{g,n}(\sF)$ is a separated DM stack as the universal family of prestable curves with finite automorphism groups, $\overline{\cC}_{g,n}(\sF)$ is also an integral proper DM stack parameterizing certain $(n+1)$-pointed prestable curves of genus $g$. Therefore, $\overline{\cC}_{g,n}(\sF) \cong \overline{\cM}_{g,n+1}(u\sF)$ for some extremal assignment $u\sF$ by \cite[Theorem 1.9]{Smy13}. Note that $u$ as the $(n+1)$th projection induces the map $p_{n+1} \colon G_{g,n+1} \to G_{g,n}$ that sends $\Gamma \in G_{g,n+1}$ to $p_{n+1}(\Gamma)$, which is the result of forgetting the $(n+1)$th half-edge of $\Gamma$ and then contracting the non-stable subgraph $B_\Gamma$ of $\Gamma$ (as a semistable $n$-pointed dual graph of genus $g$) that used to contain $(n+1)$th half-edge; so $B_\Gamma$ is either empty or a one-vertex graph.
		From this description, $u\sF(\Gamma) \subset \Gamma$ is the complete subgraph of $\Gamma$ generated by the preimage of $\sF(p_{n+1}(\Gamma))$ and $B_\Gamma$; the verification of this description of $u\sF$ is similar to the proof of Proposition~\ref{prop:moduli_separating_axis-like}.
		
		From the description of $\overline{\cM}_{g,n+1}$ and $\overline{\cM}_{g,n+1}(u\sF)$, if $K/k$ is an algebraically closed field extension (where $k$ is our algebraically closed base field), then $F^u(C)$ for each $K$-curve $C$ must correspond to the contraction of the irreducible rational component corresponding to $B_{\Gamma_C}$ (if nonempty) together with any remaining separating rational multibridges. This gives the description of the set-theoretic map $|F^u| \colon |\overline{\cM}_{g,n+1}| \to |\overline{\cM}_{g,n+1}(u\sF)|$. 
		Then similar arguments from the rest of the proof of Proposition~\ref{prop:forget_DM_stable_to_stable_axis-like} imply this set-theoretic map extends to the desired lift $F^u$ of $F$ as a morphism between the universal curves over $\overline{\cM}_{g,n}$ and $\overline{\cM}_{g,n}(\sF)$.
	\end{remark}

	If $\cY$ is a locally integral and locally Noetherian algebraic stack, we say a morphism $\cX \to \cY$ 
	is a {\it resolution} if it is schematic (i.e., representable by schemes), proper, and birational with $\cX$ regular. We will also need to consider the following more general definition.
	
	\begin{definition} \label{def:quasi-resolution}
		Let $f \colon \cX \to \cY$ be a representable morphism of separated DM stacks where $\cY$ is locally integral and locally Noetherian. Then, $f$ is a {\it quasi-resolution} if $f$ is proper birational and there exists an \'etale cover $T \twoheadrightarrow \cY$ by a locally integral scheme $T$ such that the pullback $\cX \times_{\cY} T \to T$ of $f$ is a morphism of schemes that is a resolution of $T$.  
	\end{definition}

\begin{remark}
Every quasi-resolution is a proper birational representable morphism by \'etale descent. By \cite[Example 2.3]{Ful10}, there are representable morphisms that are not schematic, even if there is an \'etale cover whose pullback is schematic. 
\end{remark}


The following lemma will be useful in proving $F$ from Proposition \ref{prop:forget_DM_stable_to_stable_axis-like} is a quasi-resolution.

	\begin{lemma} \label{l:image_point}
Let $\cX$ be a separated DM stack of finite type over our characteristic $0$ algebraically closed base field $k$. If $x \in \cX(k)$ and $\cZ_x$ is the scheme-theoretic image of $x$ in $\cX$, then the induced map $p_x \colon \Spec k \to \cZ_x$ is a finite \'etale surjection.
	\end{lemma}

	\begin{proof}
		Let $c \colon \cX \to X$ be the coarse map which exists by Keel--Mori.  
		Then $X$ is a finite type $k$-scheme and $c \circ x \colon \Spec k \to X$ is a proper morphism. Let $Z_x$ be the scheme-theoretic image of $c \circ x$ in $X$. Note $\cX$ has finite diagonal and it is defined over the spectrum of a characteristic zero field, so \cite[Theorem 11.3.6]{Ols16} implies the pullback $c_x \colon \cZ_x \to Z_x$ of $c$ along the closed immersion $\iota_{Z_x} \colon Z_x \hookrightarrow X$ is also the coarse map; the domain of $c_x$ being $\cZ_x$ follows from the definition of the scheme-theoretic image of a morphism of stacks (see \cite[Tag 0CMI]{Stacks}) and the fact that $c$ is a proper universal homeomorphism. Since $Z_x \cong \Spec k$ and $c_x \colon \cZ_x \to Z_x$ is also a proper universal homeomorphism as a coarse map, $|\cZ_x| \cong |Z_x| \cong |\Spec k|$ and so the induced map $p_x \colon \Spec k \to \cZ_x$ from $x \colon \Spec k \to \cX$ is a surjection.
		
		To see that $p_x$ is finite, first choose an \'etale cover $f \colon U \twoheadrightarrow \cX$ with $U$ a separated $k$-scheme of finite type; since $U$ and $\cX$ are of finite type, $f$ is also of finite type by \cite[Tag 050Y, 06U9]{Stacks}. Then, consider the following Cartesian diagram
		\[
			\xymatrix{
				T_x \ar[d]_{p'_x} \ar[r]^-{h} & \Spec k \ar[d]^{p_x}\\
				Z_x' \ar[d]_{j'_x} \ar[r]^-{g} & \cZ_x \ar[d]^{j_x}\\
				U \ar[r]_-{f} & \cX,
			}
		\]
		where $x$ is $2$-isomorphic to $j_x \circ p_x$, $j_x$ (which is the pullback of $\iota_{Z_x} \colon Z_x \hookrightarrow X$ under $c \colon \cX \to X$) and $j'_x$ are closed immersions, $p'_x$ is a surjection, and $g,h$ are \'etale surjections of finite type. Note that $k = \overline{k}$ implies that $T_x$ is isomorphic to $\sqcup_{i=1}^d \Spec k$ for some positive integer $d$. Moreover, $Z_x'$ is a closed subscheme of $U$, so that $Z'_x$ is a $k$-scheme of finite type. Since $f$ is flat and $\cZ_x$ is the scheme-theoretic image of $x \colon \Spec k \to \cX$, $Z'_x$ is also the scheme-theoretic image of $j'_x \circ p'_x \colon T_x \to U$ by \cite[Tag 0CMK]{Stacks}. Hence, $Z'_x$ is a reduced zero-dimensional scheme with a finite cover $p'_x \colon T_x \to Z'_x$; in fact, $k=\overline k$ with $T_x \cong \sqcup_{i=1}^d \Spec k$ implies that $p'_x$ is also \'etale. Then by descent, $p_x$ is a finite \'etale surjection.
	\end{proof}

	Now we are ready to prove that $F$ is a quasi-resolution, by first stating a general lemma then applying it to $F \colon \overline{\cM}_{g,n} \to \overline{\cM}_{g,n}(\sF)$ from Proposition~\ref{prop:forget_DM_stable_to_stable_axis-like}.
	
	\begin{lemma} \label{l:quasi-resolution_rel_ample}
		Let $f \colon \cX \to \cY$ be a representable proper birational morphism of finite type DM stacks, with $\cX$ smooth and separated, and $\cY$ separated and integral. Let $c_\cX \colon \cX \to X$ and $c_\cY \colon \cY \to Y$ be the coarse moduli maps, 
		and let $c(f) \colon X \to Y$ be the induced morphism. If there is a $c_f$-ample line bundle $L$ on $X$, then $f$ is a quasi-resolution.
	\end{lemma}
	
	\begin{proof}
		From the assumptions on $\cY$, there exists an \'etale surjection $h \colon T \twoheadrightarrow \cY$ with $T$ a reduced scheme of locally finite type by \cite[Tag 04YF]{Stacks}. Consider the pullback $f_h \colon U \colonequals T \times_{\cY} \cX \to T$ of $f$. As $f$ is representable, it suffices to show that $U$ is a scheme by Definition~\ref{def:quasi-resolution}. We will show that $f_h$ is a projective morphism, which implies that $U$ is a scheme.
		
		Since $T$ is locally of finite type over our algebraically closed base field $k$, there is a point $x \in T(k)$. Let $\Aut(h \circ x)$ be the stabilizer group of $h \circ x \colon \Spec k \to \cY$. Then, Lemma~\ref{l:image_point} implies $h \circ x$ decomposes into $\iota_{h \circ x} \circ p_{h \circ x}$, where $p_{h \circ x} \colon \Spec k \to \cZ_{h \circ x}$ is the finite \'etale surjection onto the scheme-theoretic image of $h \circ x$ and $\iota_{h \circ x} \colon \cZ_x \hookrightarrow \cY$ is the induced closed immersion. Consider the following commutative diagram
		\[
		\xymatrix{
			U_x \ar[r]^{p_{h \circ x}'} \ar[d]_{\pi_{h \circ x}} & U_{\cZ_x} \ar[r]^{\iota_{h \circ x}'} \ar[d]_{\pi_{h \circ x}'} & \cX \ar[r]^{c_\cX} \ar[d]^{f} & X \ar[d]^{c(f)}\\
			\Spec k \ar[r]_(.58){p_{h \circ x}} & \cZ_x \ar[r]_-(.5){\iota_{h \circ x}} & \cY \ar[r]_{c_{\cY}} & Y,
		}
		\]
		where the two leftmost squares are Cartesian; so $U_{\cZ_x} \colonequals \cZ_{h \circ x} \times_{\cY} \cX$ and $U_x \colonequals \Spec k \times_{\cZ_{h \circ x}} U_{\cZ_x}$.
		
		Consider the pullback $(c_\cX \circ \iota_{h \circ x}' \circ p_{h \circ x}')^*L$ on $U_x$. First, $c_\cX$ is proper and quasi-finite by \cite[Theorem 1.1]{Con05}. Then, $\iota_{h\circ x}$ being a closed immersion and $p_{h \circ x}$ being a finite \'etale surjection imply that $\iota'_{h \circ x}$ and $p'_{h \circ x}$ are proper quasi-finite morphisms. Thus, the composition $c_\cX \circ \iota_{h \circ x}' \circ p_{h \circ x}'$ is a proper quasi-finite morphism between schemes so it must be finite. 
		From the above commutative diagram, we have a morphism from $U_x$ to the fiber $H$ of $c(f) \colon X \to Y$ over $c_{\cY} \circ \iota_{h \circ x} \circ p_{h \circ x} \colon \Spec k \to Y$. Since $Y$ is a separated scheme of finite type over $\Spec k$, $c_{\cY} \circ \iota_{h \circ x} \to p_{h \circ x}$ is a closed immersion, so that $H$ is a closed subscheme of $X$. Thus, the canonical morphism $U_x \to H$ is also finite. Since $L$ is $c(f)$-ample, the restriction of $L$ to $H$ is ample; through the finite canonical morphism $U_X \to H$, the pullback of $L$ on $U_x$, which is isomorphic to $(c_\cX \circ \iota_{h \circ x}' \circ p_{h \circ x}')^*L$, is also ample.
		
		Since $U_x$ is isomorphic to the fiber of $f_h$ over $x \colon \Spec k \to T$ and $x \in T(k)$ is arbitrary, 
		we have shown that the line bundle $(c_\cX \circ h')^*L$ is ample when restricted to closed fibers of $f_h$, where $h' \colon U \to \cX$ is the pullback of $h$ along $f$. Thus, the line bundle $(c_\cX \circ h')^*L$ is $f_h$-ample by \cite[Tag 0D3D]{Stacks} because $U$ is an algebraic space and the scheme $T$ being locally of finite type (over $\Spec k$) implies that the set $T(k)$ of closed points of $T$ is Zariski dense. Since $f_h$ is proper, it is a projective morphism to a scheme $T$, hence $U$ is also a scheme.
	\end{proof}
	
	\begin{lemma} \label{l:quasi-resolution_moduli_stable_axis-like}
		The map $F \colon \overline{\cM}_{g,n} \to \overline{\cM}_{g,n}(\sF)$ from Proposition~\ref{prop:forget_DM_stable_to_stable_axis-like} is a quasi-resolution.
	\end{lemma}
	
	\begin{proof}
		Recall that $\overline{\cM}_{g,n}$ is a smooth proper DM stack and $\overline{\cM}_{g,n}(\sF)$ is a integral proper DM stack by Definition~\ref{def:stack_Z-stable}. By Proposition~\ref{prop:forget_DM_stable_to_stable_axis-like} and Lemma~\ref{l:quasi-resolution_rel_ample}, it remains to find a $c(F)$-ample line bundle on the projective coarse moduli space $\overline{M}_{g,n}$ of $\overline{\cM}_{g,n}$, where $c(F) \colon \overline{M}_{g,n} \to \overline{M}_{g,n}(\sF)$ is the induced morphism; since the two coarse moduli spaces are proper, $c(F)$ is a proper morphism. The projectivity of $\overline{M}_{g,n}$ implies that there is an ample line bundle $L$ on $\overline{M}_{g,n}$. Letting $k$ be the algebraically closed base field, every $k$-point of $\overline{M}_{g,n}(\sF)$ is closed, so every fiber of $c(F)$ over $k$ is a proper closed subscheme of $\overline{M}_{g,n}$ where the restriction of $L$ is ample. Then, \cite[Tag 0D3D]{Stacks} shows $L$ is a $c(F)$-ample line bundle.
	\end{proof}

	
	Before proving that $\overline{\cM}_{g,n}(\sF)$ is locally algebraically simply connected, the following proposition describes the $k$-fibers of $F \colon \overline{\cM}_{g,n} \to \overline{\cM}_{g,n}(\sF)$ 

	\begin{proposition} \label{prop:fiber_M0,mbar}
		Let $k$ be an algebraically closed field. Suppose that $C$ is a stable $n$-pointed separating axis-like curve of genus $g$ defined over $k$. Assume that $\{q_1,\dotsc,q_j\}$ is the set of non-nodal points of $C$, so that for every $i$, $q_i$ is a separating $m_i$-axis point with $m_i \ge 3$. If $x \colon \Spec k \to \overline{\cM}_{g,n}(\sF)$ is the geometric point corresponding to $C$, then
		\[
			F_x \cong \prod_{i=1}^j \overline{M}_{0,m_i},
		\]
		where $F$ is as defined in Proposition~\ref{prop:forget_DM_stable_to_stable_axis-like} and $F_x$ is the fiber of $F$ over $x$.
	\end{proposition}

	\begin{proof}
		Recall from the definition of fiber product of algebraic stacks as in \cite[Tag 003O, 04TE]{Stacks} that for any scheme $T$, any $T$-point of $F_x$ is characterized by the tuple $(\alpha, \beta, \gamma)$ where $\alpha \colon T \to \Spec k$ and $\beta \colon T \to \overline{\cM}_{g,n}$ are maps, and $\gamma \colon x \circ \alpha \to F \circ \beta$ is a $2$-isomorphism. Since $x$ corresponds to $C$, $x \circ \alpha \in \overline{\cM}_{g,n}(\sF)$ corresponds to the trivial family $\pi \colon C \times T \to T$. $\beta \in \overline{\cM}_{g,n}(T)$ corresponds to a flat family $\tau \colon \cC \to T$ of stable $n$-pointed curves and $F \circ \beta \in (\overline{\cM}_{g,n}(\sF))(T)$ corresponds to a flat family $\hat{\tau} \colon \hat{\cC} \to T$ of stable $n$-pointed separating axis-like curves. Note Remark~\ref{rmk:extension_universal_family} implies that there is a lift $F^u$ of $F$ as a morphism between universal curves over $\overline{\cM}_{g,n}$ and $\overline{\cM}_{g,n}(\sF)$, which induces a $T$-morphism $\xi \colon \cC \to \hat{\cC}$ such that for every geometric point $t$ of $T$, $\xi_t \colon \cC_t \to \hat{\cC}_t$ is the contraction of $\sF(\cC_t)$. As $\gamma \colon x \circ \alpha \to F \circ \beta$ corresponds to a $T$-isomorphism $f \colon C \times T \to \hat{\cC}$ as isotrivial families in $(\overline{\cM}_{g,n}(\sF))(T)$, we see $\xi$ is a contraction of separating rational bridges of fibers of $\tau$.
		
		Using the $T$-isomorphism $f \colon C \times T \to \hat\cC$, each $q_i$ lifts to a section $\sigma_i \colon T \to \hat{\cC}$ of $\hat{\cC}/T$, and $\hat{\cC}/T$ has seminormal singularities along $\sigma_i(T)$ by the proof of Proposition~\ref{prop:forget_DM_stable_to_stable_axis-like}. Defining $\cZ_i \colonequals \xi^{-1}(\sigma_i(T))$ for every $i$, then $\xi$ lifts to an isomorphism from $\overline{\cC \setminus \cZ_i}$ to the normalization of $\hat{\cC}$ at $\sigma_i(T)$ by \cite[Tag 0EUS]{Stacks}. Since the number of connected components of the normalization of $\hat{\cC}$ at $\sigma_i(T)$ for every $i$ is $m_i$ via the isomorphism $f \colon C \times T \to \hat\cC$, $\cZ_i \cap \overline{\cC \setminus \cZ_i}$ is a trivial family of $m_i$ distinct points over $T$; fix an ordering of those $m_i$-points by ordering connected components of the normalization of $C$ at $q_i$. Because every fiber of $\cZ_i$ is a separating rational $m_i$-bridge by the construction of $\xi$ with marked points induced by $\cZ_i \cap \overline{\cC \setminus \cZ_i}$, $(\alpha,\beta,\gamma) \in (\overline{\cM}_{g,n}(\sF))(T)$ induces flat families $(\cZ_i, \cZ_i \cap \overline{\cC \setminus \cZ_i})/T$ stable $m_i$-pointed curves of genus zero for each $i$. Denote these families by $(y_1,\dotsc,y_j) \in \prod \overline{M}_{0,m_i}(T)$. Using the definition of fiber product of algebraic stacks, it is easy to see that $(y_1,\dotsc,y_j)$ is independent of the choice of isomorphism classes of $(\alpha,\beta,\gamma)$. 
		
		
		By Lemma~\ref{l:quasi-resolution_moduli_stable_axis-like}, $x \colon \Spec k \to \overline{\cM}_{g,n}(\sF)$ factors through an \'etale cover $W \to \overline{\cM}_{g,n}(\sF)$ by a scheme $W$ and the pullback $F_W \colon W' \to W$ of $F$ along $W$ is a morphism of schemes; since $F_x$, the pullback of $F$ along $x$, is also isomorphic to the pullback of a morphism $F_W$ of schemes along a morphism $\Spec k \to W$, $F_x$ is also a scheme. Then the conclusion of the previous paragraph leads to a map $F_x(T) \to \prod \overline{M}_{0,m_i}(T)$ of sets. By construction, 
		this map is injective. To see that this map is surjective, fix arbitrary $(y_1,\dotsc,y_j) \in \prod \overline{M}_{0,m_i}(T)$. Then, there is a corresponding family of stable $n$-pointed curves of genus $g$ by gluing the families of stable pointed genus zero curves coming from $(y_1,\dotsc,y_j)$ and the normalization of $C \times T$ along trivial sections corresponding to $p_i$ for every $i$; by respecting the orders of marked points/sections, there exists $(\alpha,\beta,\gamma) \in F_x(T)$ that maps to $(y_1,\dotsc,y_j)$.
		
		Since the choice of $T$ is arbitrary, the isomorphism $F_x(T) \cong \prod \overline{M}_{0,m_i}(T)$ lifts to an isomorphism $F_x \cong \prod \overline{M}_{0,m_i}(T)$.
	\end{proof}
	
	As an application of Lemma~\ref{l:quasi-resolution_moduli_stable_axis-like}, Proposition~\ref{prop:fiber_M0,mbar}, and Theorem~\ref{thm:asc_fib_lasc}, the following corollary shows that $\overline{\cM}_{g,n}(\sF)$ is locally algebraically simply connected.
	
	\begin{corollary} \label{cor:moduli_FL_lasc}
		$\overline{\cM}_{g,n}(\sF)$ is locally algebraically simply connected. 
	\end{corollary}

	\begin{proof}

		Let $k$ be the algebraically closed base field and take any geometric point $x \colon \Spec k \to \overline{\cM}_{g,n}(\sF)$.
		Then, the fiber $F_x$ of $F$ over $x \colon \Spec k \to \overline{\cM}_{g,n}(\sF)$ is a product of moduli space of stable pointed curves of genus zero by Proposition~\ref{prop:fiber_M0,mbar}.
		Since moduli spaces of stable pointed curves of genus zero are smooth proper and rational by \cite[p. 561]{Kee92}, $F_x$ is as well.
		Then, $F_x$ is smooth, proper, and rationally connected, so that $\pi_1^{\text{\'et}}(F_x)=1$ by \cite[Proposition 2.3]{Kol00}. 
		Since $x$ is arbitrary, Theorem~\ref{thm:asc_fib_lasc} and Definition~\ref{def:lasc_DMstack} imply that $\overline{\cM}_{g,n}(\sF)$ is locally algebraically simply connected.
	\end{proof}
	
	\section{The locus of quasi-separating axis-like curves} \label{sec:qsep_axis-like}


	In this section, we define quasi-separating axis-like curves and the open substack $\cM_{g,n}(\cZ)^{\textrm{qs-axis}}$ of $\overline{\cM}_{g,n}(\cZ)$ parameterizing $\cZ$-stable quasi-separating axis-like curves, see Proposition \ref{prop:forget_stable_Z-qsep_to_Z-stable_qsep_axis-like}. We will see in Theorem~\ref{thm:extended_Torelli_Z-stable} that $\cM_{g,n}(\cZ)^{\textrm{qs-axis}}$ is a locus in $\overline{\cM}_{g,n}(\cZ)$ where the Torelli map $t_{g,n} \colon \cM_{g,n} \to \cA_g$ extends. Also see Example~\ref{ex:why_quasi-separating_axis-like} for why Torelli map does not extend over some non-quasi-separating axis-like curves. 
	
	\begin{definition} \label{def:qsep_ratl_bridge}
		A rational $m$-bridge $D$ of an $n$-pointed curve $(C;p_1,\dotsc,p_n)$ is {\it quasi-separating} if 
		\begin{enumerate}
			\item there is at most one connected component $C_0$ of $\overline{C \setminus D}$ with $\#\Supp (C_0 \cap D)>1$, and
			\item $\#\Supp (C_0 \cap D) \le 3$.
		\end{enumerate}
	\end{definition}

	\begin{definition} \label{def:qsep_axis}
		An $m$-axis point $p$ of an $n$-pointed curve $(C;p_1,\dotsc,p_n)$ with $m \ge 3$ is {\it quasi-separating} if $p \neq p_i$ for every $i$ and the normalization $\nu_p \colon C' \to C$ of $C$ at $p$ has
		\begin{enumerate}
			\item at most one connected component $C_0$ of $C'$ such that $\#\Supp (C_0 \cap \nu_p^{-1}(p)) > 1$, and
			\item $\#\Supp (C_0 \cap \nu_p^{-1}(p)) \le 3$.
		\end{enumerate}
	\end{definition}

	Now we are ready to define quasi-separating axis-like curves:
	
	\begin{definition} \label{def:quasi-separating_axis-like}
		An $n$-pointed axis-like curve $(C;p_1,\dotsc,p_n)$ is {\it quasi-separating} if every $m$-axis point of $C$ is quasi-separating whenever $m \ge 3$ and avoids marked points $p_1,\dotsc,p_n$ of $C$.
	\end{definition}

	The stable curves associated with the definition of $\cZ$-stable quasi-separating axis-like curves are characterized as follows:
	
	\begin{definition} \label{def:stable_Z-quasi-separating_axis-like}
		Let $\cZ$ be an extremal assignment on $G_{g,n}$. Then, a stable $n$-pointed curve $(C;p_1,\dotsc,p_n)$ of genus $g$ is {\it $\cZ$-quasi-separating} if $\cZ(C)$ is a (possibly empty) finite disjoint union of quasi-separating rational multibridges.
	\end{definition}
	
	When $\cZ = \sF$, any stable $n$-pointed curve is $\sF$-quasi-separating by Lemma~\ref{l:stability_separating_axis-like}. By Proposition~\ref{prop:forget_DM_stable_to_stable_axis-like}, $\sF$-stable quasi-separating axis-like curves (which are exactly stable separating axis-like curves) come from stable $\sF$-quasi-separating curves. There is an analogous result relating stable $\cZ$-quasi-separating curves and $\cZ$-stable quasi-separating axis-like curves:
	
	\begin{lemma} \label{l:contract_Z-qs}
		Let $\cZ$ be an extremal assignment on $G_{g,n}$. Let $(C;p_1,\dotsc,p_n)$ be a $\cZ$-stable $n$-pointed genus $g$ quasi-separating axis-like curve. If $(C^s;p_1^s,\dotsc,p_n^s)$ is a stable $n$-pointed  genus $g$ curve and $\phi \colon C^s \to C$ is as in Definition~\ref{def:Z-stability}, then $C^s$ is a $\cZ$-quasi-separating curve.
	\end{lemma}

	\begin{proof}
		If $q$ is any non-nodal singular point of $C$, then $q \not \in \{p_1,\dotsc,p_n\}$ and $q$ is a quasi-separating $m$-axis point of $C$ by Definition~\ref{def:quasi-separating_axis-like}. By \cite[Lemma 1.17(a)]{Smy13}, it is also a singularity of type $(0,m)$ with $m \ge 3$, so the corresponding connected component $Z = \phi^{-1}(q)$ of $\cZ(C^s)$ is a rational $m$-bridge. 
		
		It suffices to show that $Z$ is a quasi-separating rational $m$-bridge. Let $B_1,\dotsc,B_l$ be the enumeration of the connected components of $\overline{C^s \setminus Z}$, and let $P_i \colonequals \Supp(B_i \cap Z)$ for every $i$. Since $q$ is a seminormal singular point by \cite[Corollary 1(2)]{Dav78}, the restriction of $\phi$ to $\overline{C^s \setminus Z}$ factors through the normalization of $C$ at $q$ by \cite[Tag 0EUS]{Stacks}. Letting $\nu_q \colon C' \to C$ be the normalization of $C$ at $q$, the induced map $\psi_q \colon \overline{C^s \setminus Z} \to C'$ (so that $\nu_q \circ \psi = \phi$ on $\overline{C^s \setminus Z}$) is an isomorphism on an open neighborhood of $\sqcup P_i \subset \overline{C^s \setminus Z}$ and $\psi_q$ induces bijections on connected components of $\overline{C^s \setminus Z}$ and $C'$; so define $B_i'$ to be the connected component of $C'$ corresponding to $B_i$ via $\psi_q$. Then $q$ being a quasi-separating $m$-axis point implies that there is at most one $j \in \{1,\dotsc,l\}$ such that $\# (B_j' \cap \nu_q^{-1}(q))$ is equal to $2$ or $3$ and for any $i \neq j$, $\# (B_i' \cap \nu_q^{-1}(q)) = 1$. Using the properties of $\psi_q$ mentioned before, this condition on $B_i \cap \nu_q^{-1}(q)$ implies that $\# P_j$ is equal to $2$ or $3$ if such $j$ exists, and $\# P_i = 1$ for any $i \neq j$. This condition coincides with $Z$ being a quasi-separating rational $m$-bridge.
	\end{proof}

	This suggests a set-theoretic correspondence between stable $\cZ$-quasi-separating curves and $\cZ$-stable quasi-separating axis-like curves, analogous to $|F| \colon |\overline{\cM}_{g,n}| \to |\overline{\cM}_{g,n}(\sF)|$. We construct an analogue $F_{\cZ} \colon \cM_{g,n}^{\cZ-\mathrm{qs}} \to \cM_{g,n}(\cZ)^{\textrm{qs-axis}}$ of $F \colon \overline{\cM}_{g,n} \to \overline{\cM}_{g,n}(\sF)$ through a series of steps, including Lemma~\ref{l:stable_Z-qs-axis} and Proposition~\ref{prop:forget_stable_Z-qsep_to_Z-stable_qsep_axis-like} that define $\cM_{g,n}^{\cZ-\mathrm{qs}}$ and $\cM_{g,n}(\cZ)^{\textrm{qs-axis}}$.

	\begin{lemma} \label{l:stable_Z-qs-axis}
		Given an extremal assignment $\cZ$ on $G_{g,n}$, the locus in $\overline{\cM}_{g,n}$, of $\cZ$-quasi-separating curves, forms an open substack $\cM_{g,n}^{\cZ-\mathrm{qs}}$ of $\overline{\cM}_{g,n}$, which is a separated smooth DM stack of finite type.
	\end{lemma}

	\begin{proof}
		Let $Y$ be a subset of $|\overline{\cM}_{g,n}|$ consisting of stable $\cZ$-quasi-separating $n$-pointed curves, and let $X$ be the complement of $Y$ in $|\overline{\cM}_{g,n}|$. By the definition of topology on $|\overline{\cM}_{g,n}|$ in \cite[(5.5)]{LMB00}, the existence of an open substack $\cM_{g,n}^{\cZ-\mathrm{qs}}$ of $\overline{\cM}_{g,n}$ follows from the openness of $Y$. It suffices to show $X$ is closed since $\cM_{g,n}^{\cZ-\mathrm{qs}}$ is an open substack of the smooth proper DM stack $\overline{\cM}_{g,n}$.
		
		Closedness of $X$ follows from constructibility of $X$ and stability of $X$ under specializations by \cite[\S 5]{HR17} (particularly the discussion right after Theorem 5.1 of loc. cit.). To see that $X$ is constructible, note \cite[\S 2]{Cap20} implies that the separated DM stack $\overline{\cM}_{g,n}$ admits a locally closed stratification by $\cM_\Gamma$, where $\Gamma$ is any element of the finite set $G_{g,n}$ and $|\cM_\Gamma|$ consists of stable curves whose dual graphs are isomorphic to $\Gamma$.\footnote{Although \cite[\S 2]{Cap20} uses the language of schemes, the construction applies to the corresponding stacks as well} By Definition~\ref{def:stable_Z-quasi-separating_axis-like}, the property that a curve is (not) $\cZ$-quasi-separating only depends on the dual graph $\Gamma_C$ of $C$. As a result, $X$ is a finite union of strata $|\cM_\Gamma|$ where $\cZ(\Gamma)$ is nonempty and not a finite disjoint union of quasi-separating rational multibridges; so $X$ is indeed constructible.
		
		To check that $X$ is stable under specializations, consider any DVR $R$ and a morphism $f \colon \Spec R \to \overline{\cM}_{g,n}$; this corresponds to a flat family $\pi \colon \cC \to \Spec R$ of stable $n$-pointed curves of genus $g$, where the families of points are characterized by sections $\varphi_i \colon \Spec R \to \cC$ of $\pi$ for all $i=1,\dotsc,n$. As in \S~\ref{subsec:Smyth_stability}, assume (up to finite cyclic base change on $\Spec \widetilde{R}$) that the dual graph $\Gamma_1$ (resp., $\Gamma_2$) of the generic (resp., special) fiber of $\pi$ coincides with that of the geometric generic (resp., special) fiber. Then $\pi$ corresponds to a degeneration $\Gamma_1 \rightsquigarrow \Gamma_2$. 
		
		Assume that $\cZ(\Gamma_1)$ is nonempty and not a finite disjoint union of quasi-separating rational multibridges (so that $|f| \colon |\Spec R| \to |\overline{\cM}_{g,n}|$ sends the generic point of $\Spec R$ to $X$). Then, it remains to show that $\cZ(\Gamma_2)$ is nonempty and not a finite disjoint union of quasi-separating rational multibridges (so that $|f|$ sends the special point of $\Spec R$ to $X$). Observe that the assumption on $\cZ(\Gamma_1)$ implies that there exists a connected component $Z$ of $\cZ(\Gamma_1)$ that is not a quasi-separating rational $m$-bridge for any $m \ge 3$. By Definition~\ref{def:Z-stability} and Definition~\ref{def:quasi-separating_axis-like}, it remains to show that $\cZ(\Gamma_2)$ also contains a connected component that is not a quasi-separating rational $m$-bridge for any $m \ge 3$.
		
		If the genus $g(Z)$ is larger than zero, then $\Gamma_1 \rightsquigarrow \Gamma_2$ implies that there is an induced degeneration $Z \rightsquigarrow Z_0$ with $Z_0 \subset \cZ(\Gamma_2)$ by Definition~\ref{def:extremal_assignment}. Since $Z$ is connected, so is $Z_0$, which implies that $Z_0$ is contained in a connected component $Z_0'$ of $\cZ(\Gamma_2)$. Because $0 < g(Z) = g(Z_0) \le g(Z_0')$, $Z_0'$ is not a rational multibridge.
		
		The remaining case is when $g(Z)=0$ but $Z \subset \Gamma_1$ is not a quasi-separating rational multibridge. If $Z$ has a half-edge that corresponds to some half-edge $p_s$ of $\Gamma_1$, then $Z_0$ also has the same half-edge $p_s$ of $\Gamma_2$. Since $Z_0 \subset Z'_0$, $Z'_0$ is a connected subgraph with a half-edge corresponding to $p_s$, $Z_0'$ is not a quasi-separating rational multibridge so we are done.
		
		Otherwise, assume that $g(Z)=0$, $Z \subset \Gamma_1$ is not a quasi-separating rational multibridge, and $Z$ does not contain any of the half-edges correpsonding to half-edges of $\Gamma_1$. Then, this further breaks into the two cases, where either there exists a connected component $H$ of $\Gamma_1 \setminus Z$ that meets $Z$ at $m_0 \ge 4$ number of edges, or there exists two distinct connected components $H,H'$ of $\Gamma_1 \setminus Z$ that meets $Z$ at $m_0, m_0' \ge 2$ number of edges respectively. Before checking each case, define $Z_0$ to be the subgraph of $\Gamma_2$ induced by degeneration of $Z$ and $Z_0'$ to be the unique connected component of $\cZ(\Gamma_2)$ containing $Z_0$. Assume that $g(Z_0')=0$, otherwise $Z_0'$ is not a quasi-separating rational multibridge and we are done.
		
		In the first case, let $H_0$ to be the subgraph of $\Gamma_2$ induced by degeneration of $H$; then there are $m_0 \ge 4$ number of edges connecting $H_0$ and $Z_0$. If $Z_0' \supset H_0$, then $Z_0' \supset Z_0 \cup H_0$, so that $Z_0 \cup H_0$ is a degeneration of $Z \cup H$ with $g(Z_0') \ge g(Z_0 \cup H_0) = g(Z \cup H) = g(Z) + g(H) + (m_0-1) \ge 3$, contradicting $g(Z_0')=0$. Thus, $H_0 \setminus Z_0'$ is nonempty, and $Z_0 \cup (H_0 \cap Z_0')$ is connected because $Z_0'$ is connected and $H_0$ is the connected component of $\Gamma_2 \setminus Z_0$. Let $H_{0,1},\dotsc,H_{0,l}$ be the enumeration of connected components of $H_0 \setminus Z_0'$, and define $e_i \ge 1$ be the number of edges connecting $H_{0,i}$ and $(Z_0 \cup H_0) \setminus H_{0,i}$. Then 
		\[
			g(Z_0 \cup H_0) = g(Z_0) + g(H_0) + (m_0 - 1) = g(H_0) + (m_0 -1) 
		\]
		by similar arguments from before, while
		\begin{equation} \label{eq:genus_graph_computation}
			g(Z_0 \cup H_0) = g(Z_0 \cup (H_0 \cap Z_0')) + \sum_{i=1}^l (g(H_{0,i}) + e_i -1) = \sum_{i=1}^l g(H_{0,i}) + \sum_{i=1}^l (e_i-1)
		\end{equation}
		because $Z_0 \cup (H_0 \cap Z_0') \subset Z_0'$ are connected curves with $g(Z_0') = 0$. Since $\sum g(H_{0,i}) \le g(H_0)$ and $g(Z_0 \cup H_0) = g(H_0) + (m_0 - 1)$ by above, equation~\eqref{eq:genus_graph_computation} implies that $3 \le m_0 -1 \le \sum (e_i - 1)$, so either there exists one $i$ such that $e_i \ge 4$, or there exists two distinct $i,j$ with $e_i, e_j \ge 2$. Then either there exists $i$ with $H_{0,i}$ meeting $Z_0'$ along $e_i \ge 4$ edges or there exists two distinct $i,j$ such that $H_{0,i}$ and $H_{0,j}$ meet $Z_0'$ along $e_i,e_j \ge 2$ edges, so that $Z_0'$ is not quasi-separating rational multibridge by Definition~\ref{def:qsep_ratl_bridge}; this proves the assertion in the first case.
		
		On the other hand, consider the second case. Define $Z_0, Z_0'$ as before, and let $H_0, H_0'$ be the subgraph of $\Gamma_2$ induced by degenerations of $H, H'$ respectively. Similar to the first case, assume that $g(Z_0')=0$. Since $g(Z_0 \cup H_0) = g(Z \cup H) = g(Z) + g(H) + (m_0 - 1) \ge 1$, $Z_0'$ does not contain $H_0$ because $g(Z_0') = 0$; similarly, $Z_0'$ does not contain $H_0'$. Let $H_{0,1},\dotsc,H_{0,l}$ (resp., $H_{0,1}',\dotsc,H_{0,l'}'$) be the enumeration of connected components of $H_0 \setminus Z_0'$ (resp., $H_0' \setminus Z_0'$), and define $e_i \ge 1$ (resp., $e_i' \ge 1$) be the number of edges connecting $H_{0,i}$ (resp., $H_{0,i}'$) and $(Z_0 \cup H_0) \setminus H_{0,i}$ (resp., $(Z_0 \cup H_0') \setminus H_{0,i}'$). Then, a similar argument as in the first case above (involving equations analogous to equation~\eqref{eq:genus_graph_computation}) implies that $1 \le (m_0 - 1) \le \sum (e_i-1)$ and similarly $1 \le (m_0' - 1) \le \sum (e_i' -1)$, which implies that there exists $i \in \{1,\dotsc,l\}$ and $j \in \{1,\dotsc,l'\}$ such that $e_i, e_j \ge 2$. Then $H_{0,i}$ and $H_{0,i'}$ meet $Z_0'$ along $e_i,e_j \ge 2$ edges, so that $Z_0'$ is not quasi-separating rational multibridge by Definition~\ref{def:qsep_ratl_bridge}.
	\end{proof}

	$\cM_{g,n}^{\cZ-\mathrm{qs}}$ allows us to construct an analog of $F \colon \overline{\cM}_{g,n} \to \overline{\cM}_{g,n}(\sF)$ from Proposition~\ref{prop:forget_DM_stable_to_stable_axis-like}:

	\begin{proposition} \label{prop:forget_stable_Z-qsep_to_Z-stable}
		Let $\cZ$ be an extremal assignment of $G_{g,n}$. Then, there exists a representable separated birational morphism $\overline{F}_{\cZ} \colon \cM_{g,n}^{\cZ-\mathrm{qs}} \to \overline{\cM}_{g,n}(\cZ)$ of finite type, where $\overline{F}_{\cZ}(C)$ is the contraction of $\cZ(C)$ into axis-like singularities for any curve $C$ in $\cM_{g,n}^{\cZ-\mathrm{qs}}$.
	\end{proposition}

	\begin{proof}
		Notice that the identity map on $\cM_{g,n}$ defines a rational map $F'_{\cZ} \colon \overline{\cM}_{g,n} \dashrightarrow \overline{\cM}_{g,n}(\cZ)$, because $\cZ(C) = \emptyset$ for any smooth $n$-pointed curve $(C;p_1,\dotsc,p_n) \in |\cM_{g,n}|$. Observe that Lemma~\ref{l:stable_Z-qs-axis} implies that $\cM_{g,n}^{\cZ-\mathrm{qs}}$ is a separated smooth DM stack of finite type, so that it is also locally algebraically simply connected by Definition~\ref{def:lasc} and Definition~\ref{def:lasc_DMstack}. We apply Theorem~\ref{thm:extension-stacks} to construct $\overline{F}_{\cZ}$.
The construction of $|\overline{F}_{\cZ}|$ and verification of \ref{extension:DVRs} of Theorem~\ref{thm:extension-stacks} follows from similar arguments in the proof of Proposition~\ref{prop:forget_DM_stable_to_stable_axis-like}. 
		
		For the properties of $\overline{F}_{\cZ}$, notice that the representability of $\overline{F}_\cZ$ is analogous to the proof of Proposition~\ref{prop:forget_DM_stable_to_stable_axis-like}. Because $\overline{\cM}_{g,n}(\cZ)$ is a proper DM stack by Definition~\ref{def:stack_Z-stable}, the diagonal $\Delta_{\overline{\cM}_{g,n}(\cZ)/\Spec k}$ of $\overline{\cM}_{g,n}(\cZ)$ is proper as well, where $k$ is the algebraically closed base field of characteristic zero. Then, separatedness of  $\overline{F}_\cZ$ follows from properness of $\Delta_{\overline{\cM}_{g,n}(\cZ)/\Spec k}$ and separatedness of $\cM_{g,n}^{\cZ-\mathrm{qs}}$ by \cite[Tag 050N]{Stacks}. Also, \cite[Tag 06U9, 050Y]{Stacks} imply that $\overline{F}_{\cZ}$ is of finite type because $\cM_{g,n}^{\cZ-\mathrm{qs}}$ and $\overline{\cM}_{g,n}(\cZ)$ are of finite type over the base scheme $\Spec k$.
	\end{proof}

	Note that when $\cZ = \sF$, then $\cM_{g,n}^{\cZ-\mathrm{qs}} \cong \overline{\cM}_{g,n}$ and $\overline{F}_{\sF} = F$ from Proposition~\ref{prop:forget_DM_stable_to_stable_axis-like}. However, $\overline{F}_\cZ$ is not proper whenever there exists a $\cZ$-stable curve that is not quasi-separating axis-like. For this exact reason, we need to characterize the image of $\overline{F}_\cZ$:

	\begin{proposition} \label{prop:forget_stable_Z-qsep_to_Z-stable_qsep_axis-like}
		Given an extremal assignment $\cZ$ of $G_{g,n}$, there is an open substack $\cM_{g,n}(\cZ)^{\textrm{qs-axis}}$ of $\overline{\cM}_{g,n}(\cZ)$ where $|\cM_{g,n}(\cZ)^{\textrm{qs-axis}}|$ coincides with the locus in $|\overline{\cM}_{g,n}(\cZ)|$ consisting of quasi-separating axis-like curves. Moreover, there is a quasi-resolution 
		\[
			F_{\cZ} \colon \cM_{g,n}^{\cZ-\mathrm{qs}} \to \cM_{g,n}(\cZ)^{\textrm{qs-axis}},
		\]
		such that $\overline{F}_{\cZ}$ is $2$-isomorphic to $\iota_{\cZ} \circ F_{\cZ}$, where $\iota_{\cZ} \colon \cM_{g,n}(\cZ)^{\textrm{qs-axis}} \to \overline{\cM}_{g,n}(\cZ)$ is the open immersion.
	\end{proposition}

	\begin{proof}
		Let $U$ be the image of $|\cM_{g,n}^{\cZ-\mathrm{qs}}|$ in $|\overline{\cM}_{g,n}(\cZ)|$ by $|\overline{F}_\cZ|$. Since $\overline{\cM}_{g,n}(\cZ)$ is an integral proper DM stack by Definition~\ref{def:stack_Z-stable}, there is an \'etale surjection $h \colon T \twoheadrightarrow \overline{\cM}_{g,n}(\cZ)$ where $T$ is a locally integral (i.e., reduced and locally irreducible) scheme of locally finite type by \cite[Tag 04YF]{Stacks}. Defining $U_T \colonequals |h|^{-1}(U) \subset |T|$, openness of $U$ is equivalent to the openness of $U_T$. To see that $U_T$ is an open subset of $|T|$, it suffices to show that for every open irreducible subscheme $S$ of $T$ that is also of finite type, the set $U_S \colonequals U_T \cap |S|$ is an open subset of $|T|$. Let $\nu_S \colon S^\nu \to S$ be the normalization of $S$ and let $U_{S^\nu} \colonequals |\nu_S|^{-1}(U_S) \subset |S^\nu|$. The assumptions on $S$ imply that $\nu$ is finite by \cite[Tag 035S]{Stacks}, so openness of $U_S$ follows from openness of $U_{S^\nu}$. 
		
		To see that $U_{S^\nu}$ is open in $|S^\nu|$, let $h_S \colon S^\nu \to \overline{\cM}_{g,n}(\cZ)$ be the composition of $\nu$, the open immersion $S \hookrightarrow T$, and $h$. Note that $h_S$ corresponds to a flat family $(\pi \colon \cC \to S^\nu; \sigma_1,\dotsc,\sigma_n)$ of $n$-pointed $\cZ$-stable curves of genus $g$ (here, each $\sigma_i \colon S^\nu \to \cC$ is a section of $\pi$) which are prestable by Definition~\ref{def:Z-stability}. Hence, \cite[Lemma 2.2]{Smy13} implies that there exists an alteration $\rho \colon \widetilde{S} \to S^\nu$ and a birational $\widetilde{S}$-morphism $\phi \colon \cC^s \to \widetilde{\cC}$ from a flat family $(\pi^s \colon \cC^s \to \widetilde{S}; \sigma_1^s,\dotsc,\sigma_n^s)$ of stable $n$-pointed curves of genus $g$ to the pullback $(\widetilde{\pi} \colon \widetilde{\cC} \to \widetilde{S}; \widetilde{\sigma}_1,\dotsc,\widetilde{\sigma}_n)$ of $(\pi \colon \cC \to S^\nu; \sigma_1,\dotsc,\sigma_n)$ (here, $\phi \circ \sigma_i^s = \widetilde{\sigma}_i$ for every $i$). By \cite[Corollary 2.10]{Smy13}, the pullback of $\phi$ to any geometric point $x$ of $\widetilde{S}$ induces a contraction $\phi_x \colon \cC^s_x \to \widetilde{\cC}_x$ of $\cZ(\cC^s_x)$ as in Definition~\ref{def:Z-stability}. Then Lemma~\ref{l:contract_Z-qs} and Lemma~\ref{l:stable_Z-qs-axis} imply that the preimage $U_{\widetilde{S}}$ of $U_{S^\nu}$ via $|\rho| \colon |\widetilde{S}| \to |S^\nu|$ is an open subset of $|\widetilde{S}|$; so properness of alteration $\rho$ implies that $U_{S^\nu}$ is open $|\widetilde{S}|$.
		
		The above arguments imply that $U \subset |\overline{\cM}_{g,n}(\cZ)|$ is open, so let $\cM_{g,n}(\cZ)^{\textrm{qs-axis}}$ be the open substack of $\overline{\cM}_{g,n}(\cZ)$ such that $|\cM_{g,n}(\cZ)^{\textrm{qs-axis}}| = U$ in $|\overline{\cM}_{g,n}(\cZ)|$, and let $\iota_\cZ \colon \cM_{g,n}(\cZ)^{\textrm{qs-axis}} \to \overline{\cM}_{g,n}(\cZ)$ be the induced open immersion. By Lemma~\ref{l:contract_Z-qs}, $\cM_{g,n}(\cZ)^{\textrm{qs-axis}}$ is the reduced substack of the moduli stack of $n$-pointed $\cZ$-stable quasi-separating axis-like curves of genus $g$, and the representable morphism $\overline{F}_{\cZ} \colon \cM_{g,n}^{\cZ-\mathrm{qs}} \to \overline{\cM}_{g,n}(\cZ)$ decomposes into $\iota_\cZ \circ F_\cZ$, where the induced $F_{\cZ} \colon \cM_{g,n}^{\cZ-\mathrm{qs}} \to \cM_{g,n}(\cZ)^{\textrm{qs-axis}}$ is a representable birational surjection. 
		
		For the properties of $F_\cZ$, notice that $\iota_{\cZ}$ is an open immersion, which is also an affine morphism; so the diagonal $\Delta_{\iota_{\cZ}}$ of $\iota_{\cZ}$ is proper. This and Lemma~\ref{l:stable_Z-qs-axis} imply that $F_\cZ$ is separated by \cite[Tag 050M]{Stacks}. Similarly, Lemma~\ref{l:stable_Z-qs-axis} and \cite[Tag 06U9, 050Y]{Stacks} imply that $F_\cZ$ is of finite type.
		
		To see that $F_\cZ$ is proper, it remains to check the valuative criteria of properness (see \cite[Proposition 7.12]{LMB00}). Fix a DVR $R$ and consider a flat family $(\pi_R \colon \cC_R \to \Spec R; \sigma_1,\dotsc,\sigma_n)$ of $n$-pointed $\cZ$-stable curves of genus $g$, where $\pi$ is a $\cZ$-stable quasi-separating axis-like curves. Denote $K \colonequals \mathrm{Frac}(R)$ so that $\eta \colon \Spec K \to \Spec R$ is the generic point. Assume that there is a contraction $\phi_K \colon \cC_K^s \to \cC_K$ of $\cZ(\cC_K)$, where $(\cC^s_K;\sigma_1',\dotsc,\sigma_n')$ is a stable $\cZ$-quasi-separating curve. Using \cite[Lemma 2.2]{Smy13} again, there is an alteration $\rho_R \colon S \to \Spec R$ and a birational $S$-morphism $\phi_R \colon \cC_R^s \to \widetilde{\cC_R}$ from a flat family $(\pi_R^s \colon \cC_R^s \to S; \sigma_1^s,\dotsc,\sigma_n^s)$ of stable $n$-pointed curves of genus $g$ to the pullback $(\widetilde{\pi}_R \colon \widetilde{\cC}_R \to \widetilde{S}; \widetilde{\sigma}_1,\dotsc,\widetilde{\sigma}_n)$ of $(\pi_R \colon \cC_R \to \Spec R; \sigma_1,\dotsc,\sigma_n)$; here, $\widehat{\pi}_R$ is a flat family of $\cZ$-stable quasi-separating axis-like curves. Let $\rho_K \colon S_\eta \to \Spec K$ be the generic fiber of $\rho$. By \cite[Corollary 2.10]{Smy13}, the restriction of $\phi_R$ to $S_\eta$ is isomorphic to the pullback of $\phi_K$ along $\rho_K$. Note that Lemma~\ref{l:stable_Z-qs-axis} implies that the $\cZ$-quasi-separating locus of $\pi_R^s$ in $S$ is $S$ itself. Since $\rho_R \colon S \to \Spec R$ is an alteration, it is a proper generically finite \'etale surjection with $S$ irreducible; this immediately implies that $\rho_K \colon S_\eta \to \Spec K$ is a finite \'etale morphism, so $S_\eta \cong \Spec L$ for some finite field extension $L/K$. Then, properness of $\rho_R$ implies that there is a DVR $R'$ with $L \cong \mathrm{Frac}(R')$ and a morphism $u \colon \Spec R' \to S$ such that $\rho_R \circ u \colon \Spec R' \to \Spec R$ is a generically finite \'etale surjection that sends the special point of $\Spec R'$ to that of $\Spec R$. Then the pullback of $\phi_R$ induces a birational morphism from a flat family of stable $\cZ$-quasi-separating curves over $\Spec R'$, whose generic fiber is isomorphic to the pullback of $\cC_K^s$ along $\rho_K$, to a flat family of $\cZ$-stable quasi-separating axis-like curves over $\Spec R'$, which is isomorphic to the pullback of $\pi_R$. This verifies the valuative criteria of properness on $F_\cZ$, hence $F_\cZ$ is proper.
		
		It remains to show that $F_\cZ$ is a quasi-resolution. Let $c(F_\cZ) \colon M_{g,n}^{\cZ-\mathrm{qs}} \to M_{g,n}(\cZ)^{\textrm{qs-axis}}$ be the induced morphism between coarse moduli spaces from the morphism $F_\cZ$ between open substacks of proper DM stacks. Observe that $M_{g,n}^{\cZ-\mathrm{qs}}$ (resp., $M_{g,n}(\cZ)^{\textrm{qs-axis}}$) as an open subspace of the coarse moduli space $\overline{M}_{g,n}$ (resp., $\overline{M}_{g,n}(\cZ)$) of a proper DM stack $\overline{\cM}_{g,n}$ (resp., $\overline{\cM}_{g,n}(\cZ)$) is separated of finite type. Also, properness of $F_\cZ$ and the coarse moduli maps imply that $c(F_\cZ)$ is also a proper morphism of algebraic spaces by \cite[Tag 0CQK]{Stacks}. Then consider the restriction $L$ on $M_{g,n}^{\cZ-\mathrm{qs}}$ of an ample line bundle on the projective scheme $\overline{M}_{g,n}$. Then properness of $c(F_\cZ)$ with $M_{g,n}(\cZ)^{\textrm{qs-axis}}$ being an algebraic space of finite type imply that $L$ is also $c(F_\cZ)$-ample by \cite[Tag 0D3D]{Stacks}. Therefore, existence of this $L$ implies that the representable proper birational morphism $F_\cZ$ is indeed a quasi-resolution by Lemma~\ref{l:quasi-resolution_rel_ample}.
	\end{proof}

	Just like $F \colon \overline{\cM}_{g,n} \to \overline{\cM}_{g,n}(\sF)$, geometric fibers of $F_{\cZ} \colon \cM_{g,n}^{\cZ-\mathrm{qs}} \to \cM_{g,n}(\cZ)^{\textrm{qs-axis}}$ from Proposition~\ref{prop:forget_stable_Z-qsep_to_Z-stable_qsep_axis-like} are products of moduli spaces of stable pointed curves of genus zero:
	
	\begin{proposition} \label{prop:fiber_M0,mbar_qsep}
		Let $k$ be an algebraically closed field, and $\cZ$ be an extremal assignment of $G_{g,n}$. Suppose that $C$ is a $\cZ$-stable quasi-separating axis-like curve defined over $k$. Assume that $\{q_1,\dotsc,q_j\}$ is the set of non-nodal points of $C$, so that for every $i$, $q_i$ is a quasi-separating $m_i$-axis point with $m_i \ge 3$. Then, if $x \colon \Spec k \to \overline{\cM}_{g,n}(\sF)$ is the geometric point corresponding to $C$, then
		\[
		(F_\cZ)_x \cong \prod_{i=1}^j \overline{M}_{0,m_i},
		\]
		where $F_\cZ$ is as defined in Proposition~\ref{prop:forget_stable_Z-qsep_to_Z-stable_qsep_axis-like} and $(F_\cZ)_x$ is the fiber of $F_\cZ$ over $x$.
	\end{proposition}

	\begin{proof}
		The arguments in the proof of Proposition~\ref{prop:fiber_M0,mbar} apply here as well, under the assumption that there are analogues of Proposition~\ref{prop:forget_DM_stable_to_stable_axis-like}, Remark~\ref{rmk:extension_universal_family}, and Lemma~\ref{l:quasi-resolution_moduli_stable_axis-like} for $F_\cZ$. Proposition~\ref{prop:forget_stable_Z-qsep_to_Z-stable_qsep_axis-like} is the analogue of Proposition~\ref{prop:forget_DM_stable_to_stable_axis-like} and Lemma~\ref{l:quasi-resolution_moduli_stable_axis-like}.
		
		It remains to explain the salient parts of the analogue of Remark~\ref{rmk:extension_universal_family}. Let $u_{\cZ} \colon \overline{\cC}_{g,n}(\cZ) \to \overline{\cM}_{g,n}(\cZ)$ be the universal family of stable $\cZ$-quasi-separating curves. Similar arguments as in Remark~\ref{rmk:extension_universal_family} imply there is an extremal assignment $u\cZ$ such that the $\overline{\cC}_{g,n}(\cZ) \cong \overline{\cM}_{g,n+1}(u\cZ)$. Using the $(n+1)$th projection $p_{n+1} \colon G_{g,n+1} \to G_{g,n}$ described in the proof of Remark~\ref{rmk:extension_universal_family}, for any $\Gamma \in G_{g,n+1}$, $u\cZ(\Gamma) \subset \Gamma$ is the complete subgraph of $\Gamma$ generated by the preimage of $\cZ(p_{n+1}(\Gamma))$ and a complete subgraph $B_\Gamma$ of $\Gamma$ (which is either empty or a one-vertex graph with exactly three half-edges, one of which is the $(n+1)$th marked half-edge of $\Gamma$).
		
		Using the openness of $\cM_{g,n}^{\cZ-\mathrm{qs}}$ and $\cM_{g,n}(\cZ)^{\textrm{qs-axis}}$ in $\overline{\cM}_{g,n}$ and $\overline{\cM}_{g,n}(\cZ)$ respectively (see Lemma~\ref{l:stable_Z-qs-axis} and Proposition~\ref{prop:forget_stable_Z-qsep_to_Z-stable_qsep_axis-like}), let $\cM_{g,n+1}^{u\cZ-u\mathrm{qs}}$ and $\cM_{g,n+1}(u\cZ)^{u\mathrm{qs-axis}}$ be the open substacks of the universal families $\overline{\cM}_{g,n+1}$ and $\overline{\cM}_{g,n+1}(u\cZ)$ that correspond to the preimages of $\cM_{g,n}^{\cZ-\mathrm{qs}}$ and $\cM_{g,n}(\cZ)^{\textrm{qs-axis}}$ respectively. Then similar arguments as in Remark~\ref{rmk:extension_universal_family} imply there is the desired lift $F^u_{\cZ} \colon \cM_{g,n+1}^{u\cZ-u\mathrm{qs}} \to \cM_{g,n+1}(u\cZ)^{u\mathrm{qs-axis}}$ of $F_{\cZ} \colon \cM_{g,n}^{\cZ-\mathrm{qs}} \to \cM_{g,n}(\cZ)^{\textrm{qs-axis}}$.
	\end{proof}

	We end this section by proving:
	\begin{corollary} \label{cor:moduli_Z-stable_qs_axis-like_lasc}
		$\cM_{g,n}(\cZ)^{\textrm{qs-axis}}$ is locally algebraically simply connected (so normal as well). 
	\end{corollary}

	\begin{proof}
		Consider $F_\cZ \colon \cM_{g,n}^{\cZ-\mathrm{qs}} \to \cM_{g,n}(\cZ)^{\textrm{qs-axis}}$ from Proposition~\ref{prop:forget_stable_Z-qsep_to_Z-stable_qsep_axis-like}, which is a quasi-resolution. 
		By Remark~\ref{rmk:deformation_space_reduced?}, $\cM_{g,n}(\cZ)^{\textrm{qs-axis}}$ is normal.
		Letting $k$ be the algebraically closed base field, Proposition~\ref{prop:fiber_M0,mbar_qsep} and the argument from the proof of Corollary~\ref{cor:moduli_FL_lasc} imply that the fiber of $F_\cZ$ over any $k$-point of $\cM_{g,n}(\cZ)^{\textrm{qs-axis}}$ is a smooth proper rational scheme (in fact, also a product of moduli spaces of stable pointed curves of genus zero), hence has trivial \'etale fundamental group by \cite[Proposition 2.3]{Kol00}. Then Theorem~\ref{thm:asc_fib_lasc} and Definition~\ref{def:lasc_DMstack} imply that $\cM_{g,n}(\cZ)^{\textrm{qs-axis}}$ is locally algebraically simply connected.
	\end{proof}

	\section{Extending the Torelli map 
to Alexeev compactifications} \label{sec:extending_Torelli_map}
	
	The goal of this section is to prove Theorem~\ref{thm:extended_Torelli_Z-stable} and Theorem~\ref{thm:compactified_Torelli_stable_separating_axis-like}. This is an application of the Stacky Extension Theorem~(Theorem \ref{thm:extension-stacks}).


	\medskip

	When $g \ge 1$, recall the Torelli map $t_g \colon \cM_g \to \cA_g$ from the moduli stack of smooth genus $g$ curves to principally polarized abelian varieties of dimension $g$, where $t_g$ sends a curve $C \in \cM_g$ to the pair $(\mathrm{Pic}^{g-1}(C), \Theta(C))$; from now on, we instead write $(\mathrm{Jac}(C) \curvearrowright \mathrm{Pic}^{g-1}(C),\Theta(C))$ in order to remember $\mathrm{Pic}^{g-1}(C)$ as a $\mathrm{Jac}(C)$-torsor. When $g=1$, let $\overline{\cM}_1$ be the (nonseparated and non-DM) moduli stack of semistable curves of genus $1$. 
	In the introduction, we explained Alexeev's {\it compactified Torelli map} $\overline{t}_g \colon \overline{\cM}_g \to \overline{\cA}_g^{\mathrm{Ale}}$ from \cite[\S 5]{Ale04}\footnote{Although the statement of \cite[Corollary 5.4]{Ale04} is written for coarse moduli spaces (good moduli spaces when $g=1$), \cite[Theorem 5.2, 5.3, 5.4]{Ale04} imply that the statement lifts as a morphism between moduli stacks.}, where $\overline{\cA}_g^{\mathrm{Ale}}$ is the irreducible component of the moduli stack $\overline{\cA}_g$ of principally polarized stable semi-abelic pairs containing $\cA_g$; see \cite[\S 1.2]{CV11} for a combinatorial description of $\overline{t}_g(C)$ for any stable curve $C \in \overline{\cM}_g$.
	
	
	
	Consider a pair $(g,n)$ with $2g-2+n >0$; for this paragraph, assume $g \ge 1$ and $n \ge 1$. If $g \ge 2$, then there is a projection $p_{g,n} \colon \overline{\cM}_{g,n} \to \overline{\cM}_g$, which sends any stable $n$-pointed curve $(C;\sigma_1,\dotsc,\sigma_n)$ of genus $g$ to a stable unpointed curve $C'$ of genus $g$; $C'$ is constructed by forgetting the $n$ points $\sigma_1,\dotsc,\sigma_n$ on $C$ and then repeatedly contracting rational tails and bridges one at a time until the resulting curve $C'$ is stable. Using this, the {\it compactified Torelli map} for $\overline{\cM}_{g,n}$ is $\overline{t}_{g,n} \colonequals \overline{t}_g \circ p_{g,n} \colon \overline{\cM}_{g,n} \to \overline{\cA}_g^{\mathrm{Ale}}$, which a proper morphism between integral proper DM moduli stacks. If $g=1$ instead, consider the projection $p_{1,n} \colon \overline{\cM}_{1,n} \to \overline{\cM}_{1,1}$ which is defined similarly as before, except that the data of the first marked point is not forgotten. Using the forgetful map 
$p_{1,1} \colon \overline{\cM}_{1,1} \to \overline{\cM}_1$, the {\it compactified Torelli map} for $\overline{\cM}_{1,n}$ is $\overline{t}_{1,n} \colonequals \overline{t}_1 \circ p_{1,1} \circ p_{1,n} \colon \overline{\cM}_{1,n} \to \overline{\cA}_1^{\mathrm{Ale}}$, a proper morphism between integral proper DM moduli stacks.
	
	Caporaso and Vivani in \cite[Theorem 2.1.7]{CV11} gave a criterion for the fibers of the set-theoretic compactified Torelli map $|\overline{t}_g| \colon |\overline{\cM}_g| \to \left|\overline{A}_g^{\mathrm{Ale}}\right|$, where their criterion also applies to the pointed cases $|\overline{t}_{g,n}|$. To understand their criteria, let us explain their notations. First, recall that a node $p$ of a curve $C$ is {\it separating} if the normalization of $C$ at $p$ is disconnected; analogously, a node $p$ of a finite disjoint union of curves is {\it separating} if $p$ is a separating node of the unique connected component containing $p$. Given a nodal curve $C$ (our convention implies that $C$ is connected, which is not the convention of \cite{CV11}), define $\widetilde{C}$ to be the normalization of $C$ at the set of separating nodes of $C$; so $\widetilde{C}$ is a finite disjoint union of nodal curves without separating nodes. Given a nodal curve $C_1$ free from separating nodes (i.e., without separating nodes), Caporaso and Viviani define the notion of {\it stabilization} in \cite[\S 1.3]{CV11}, turning $C_1$ into a ``stable'' curve; we note that the definition of ``stable'' curve in loc. cit. is different than the standard definition (via the ampleness of canonical divisor) because $\PP^1$ and nodal plane cubics are called ``stable'' in loc. cit. To avoid this confusion, we 
make the following definition:
	
	\begin{definition} \label{def:almost_polystable_curves}
		A nodal curve $C$ of arithmetic genus $g$ is {\it almost stable} if:
		\begin{enumerate}
			\item $g \ge 2$ and $C$ is stable (i.e., $\omega_C$ is ample),
			\item $g = 1$ and $C$ is irreducible (i.e., $C$ is a nodal plane cubic), or
			\item $g = 0$ and $C \cong \PP^1$.
		\end{enumerate}
		A reduced proper scheme $D$ of pure dimension one is {\it almost polystable} if $D$ is a finite disjoint union of almost stable curves.
	\end{definition}

	Noting that Definition~\ref{def:ratl_bridge} can be extended to a finite disjoint union of nodal curves, we say a {\it rational tail} (resp., {\it rational bridge}) is a rational $1$-bridge (resp., rational $2$-bridge); when a curve is prestable or almost stable, there are no rational tails and bridges.

	\begin{definition} \label{def:stabilization}
		Let $C$ be a finite disjoint union of nodal curves. Then a {\it stabilization} of $C$ is an almost polystable curve $C^{\mathrm{st}}$, where there exists a morphism $C \rightarrow C^{\mathrm{st}}$ which is a sequence of contractions of irreducible components $E \cong \PP^1$ that are rational tails or bridges.
	\end{definition}

	By definition of almost polystable curves, the isomorphism class of $C^{\mathrm{st}}$ is uniquely determined by $C$ while the contraction morphism $C \rightarrow C^{\mathrm{st}}$ is not, unless every connected component of $C$ of genus zero or one is almost stable. Motivated by \cite[Remark 1.3.1, Corollary 1.3.3]{CV11}, we extend the definition of $|\overline{t}_g|$ to a finite disjoint union of curves, where a product $(A \curvearrowright X,D) \times (B \curvearrowright Y,E)$ of pairs of torsors and their divisors are understood as $(A \times B \curvearrowright X \times Y, D \times Y + X \times E)$: 
	
	\begin{definition} \label{def:Torelli_almost_polystable}
		Let $C_1,\dotsc,C_m$ be 
the connected components of an almost polystable curve $C$ defined over an algebraically closed field $K$. If the genus of $C$ is $g = \sum g_i$ with $g_i \colonequals p_a(C_i)$ for every $i$, then
		\[
			\overline{t}_g(C) \colonequals \prod_{g_i>0} \overline{t}_{g_i}(C_i).
		\]
	\end{definition}

	Now \cite[Remark 1.3.1, Corollary 1.3.3]{CV11} is summarized as follows:
	
	\begin{lemma} \label{l:Torelli_normalization-sep_stabilization}
		Let $(C;\sigma_1,\dotsc,\sigma_n)$ be a stable $n$-pointed curve of genus $g$ defined over an algebraically closed field. Then $\overline{t}_{g,n}(C) \cong \overline{t}_{g}(C^{\mathrm{pst}})$, where $C^{\mathrm{pst}}$ is defined as the disjoint union of every connected component of $(\widetilde{C})^{\mathrm{st}}$ whose genus is greater than zero.
	\end{lemma}

	An immediate consequence of Lemma~\ref{l:Torelli_normalization-sep_stabilization} is if two stable $n$-pointed curves $C_1$ and $C_2$ of genus $g$ satisfy $C_1^{\mathrm{pst}} \cong C_2^{\mathrm{pst}}$, then $\overline{t}_{g,n}(C_1) \cong \overline{t}_{g,n}(C_2)$. In this case, note that $C_1^{\mathrm{pst}}$ is an almost polystable curve (with no connected component isomorphic to $\PP^1$) free from separating nodes.
	
	On the other hand, \cite[Theorem 2.1.7]{CV11} gives an equivalent condition for when $\overline{t}_g(C_1) \cong \overline{t}_g(C_2)$, which is \emph{a priori} weaker than the previous condition. In loc. cit., Caporaso and Viviani use the notion of {\it $\mathrm{C}1$-sets} from \cite[Definition 2.3.1]{CV10} and {\it $\mathrm{C}1$-equivalences} from \cite[Definition 2.1.5]{CV11}. Since we only need those notions for almost polystable curves that are free from separating nodes, we give the following equivalent definition from \cite[Lemma 2.3.2(ii)]{CV10} instead of explaining the graph theory involved in \cite[Definition 2.3.1]{CV10}:
	
	\begin{definition}[\cite{CV10}*{Lemma 2.3.2(ii)}] \label{def:C1-set}
		Let $C$ be a finite disjoint union of nodal curves without separating nodes. A set $S$ of nodes of $C$ (equivalently, a subset $S$ of the set $E(\Gamma_C)$ of edges of the dual graph $\Gamma_C$ of $C$) is a {\it $\mathrm{C}1$-set} if it is of the form $\{p\} \sqcup (S \setminus \{p\})$, where $p \in S$ and the preimage of $S \setminus \{p\}$ in the normalization $\nu_p \colon C_p^\nu \to C$ of $C$ at $p$ is exactly the set of separating nodes of $C_p^\nu$.
	\end{definition}
	
	Note that Definition~\ref{def:C1-set} is well-defined by \cite[Lemma 2.3.2(ii)]{CV10}, i.e., if $p,q$ are two distinct points of a $\mathrm{C}1$-set $S$, then $S \setminus \{p\}$ is the set of separating nodes of $C_p^\nu$ and $S \setminus \{q\}$ is also the set of separating nodes of $C_q^\nu$. This implies that the set $E(\Gamma_C)$, of nodes of a finite disjoint union $C$ of nodal curves without separating nodes, has a unique partition into $\mathrm{C}1$-sets. 
	
	To define {\it $\mathrm{C}1$-equivalence}, first denote $\mathrm{Set}^1C = \mathrm{Set}^1\Gamma_C$ to be the set of $\mathrm{C}1$-sets of $C$ (equivalently, of the dual graph $\Gamma_C$). Then the following is the extensions of \cite[Definition 2.1.5]{CV11} to the disconnected case:
	
	\begin{definition}[\cite{CV11}*{Definition 2.1.5}] \label{def:C1-equivalence}
		Let $C_1$, $C_2$ be finite disjoint unions of nodal curves without separating nodes. For each $i=1,2$, let $\nu_i \colon C_i^\nu \to C_i$ be the normalization. Then, $C_1$ and $C_2$ are {\it $\mathrm{C}1$-equivalent} if the following hold:
		\begin{enumerate}
			\item there is an isomorphism $\phi \colon C_1^\nu \xrightarrow{\cong} C_2^\nu$, and
			\item $\phi$ induces a bijection $\mathrm{Set}^1C_1 \to \mathrm{Set}^2C_2$ by sending $S \in \mathrm{Set}^1C_1$ to $S' \in \mathrm{Set}^1C_2$, where $\phi(\nu_1^{-1}(S)) = \nu_2^{-1}(S')$.
		\end{enumerate}
	\end{definition}

	\begin{example} \label{ex:not_C1-equivalent}
		This example is motivated by \cite[Example 2.1.2(2)]{CV11}. Let $K$ be an algebraically closed field. Fix $g \ge 1$, $(E;p,q,r,s) \in \cM_{g,4}(K)$, and $t_1 \neq t_2$ in $\PP^1_K(K) \setminus 
		\{0,1,\infty\}$ such that $(\PP^1_K;0,1,\infty,t_1) \not \cong (\PP^1_K;0,1,\infty,t_2)$ in $\cM_{0,4}(K)$. For each $i=1,2$, let $C_i \in \overline{\cM}_{g+3}(K)$ be a stable curve such that the normalization $\nu_i \colon C_i^\nu \cong E \sqcup \PP^1 \to C_i$ satisfies $\nu_i(p) = \nu_i(0)$, $\nu_i(q)=\nu_i(1)$, $\nu_i(r)=\nu_i(\infty)$, and $\nu_i(s)=\nu_i(t_i)$. Since every node of $C_i$ is not separating for any $i=1,2$, $C_i^\mathrm{pst} \cong C_i$. So if $C_1$ and $C_2$ were $\mathrm{C}1$-equivalent, then they must be isomorphic. However, their normalizations (as a disjoint union of $4$-pointed curves) are not isomorphic, so they are not $\mathrm{C}1$-equivalent.
	\end{example}

	Since $|\overline{t}_{g,n}| \colon |\overline{\cM}_{g,n}| \to \left|\overline{\cA}_g^{\mathrm{Ale}}\right|$ only depends on $C^\mathrm{pst}$ for every $C \in |\overline{\cM}_{g,n}|$ by Lemma~\ref{l:Torelli_normalization-sep_stabilization} (which differs from $\widetilde{C}^{\mathrm{st}}$ by removing $\PP^1$ components of $\widetilde{C}^{\mathrm{st}}$), the following is a slight extension of \cite[Theorem 2.1.7]{CV11} by using Lemma~\ref{l:Torelli_normalization-sep_stabilization}:
	
	\begin{theorem} \label{thm:Torelli_C1-equivalence}
		Suppose that $(C;\sigma_1,\dotsc,\sigma_n)$ and $(C';\sigma_1',\dotsc,\sigma_n')$ are almost stable $n$-pointed curves of genus $g > 0$ defined over an algebraically closed field $K$. Then, $\overline{t}_{g,n}(C) \cong \overline{t}_{g,n}(C')$ if and only if $C^\mathrm{pst}$ and $(C')^\mathrm{pst}$ from Lemma~\ref{l:Torelli_normalization-sep_stabilization} are C1-equivalent.
	\end{theorem}

	\begin{proof}
		The case $n=0$ and $g \ge 2$ follows from \cite[Theorem 2.1.7]{CV11}, and $(g,n)=(1,0)$ case follows from Lemma~\ref{l:Torelli_normalization-sep_stabilization}. When $g = 1$ and $n > 0$, then the projection $p_{1,1} \circ p_{1,n} \colon \overline{\cM}_{1,n} \twoheadrightarrow \overline{\cM}_{1,1}$ has the property that the irreducible curve $C^{\mathrm{pst}}$ (resp., $(C')^\mathrm{pst}$) is isomorphic to $(p_{1,1} \circ p_{1,n})(C)$ (resp., $(p_{1,1} \circ p_{1,n})(C')$), so the assertion follows again by \cite[Theorem 2.1.7]{CV11}.
		
		The remaining case is when $g \ge 2$ and $n > 0$. Let $p_{g,n} \colon \overline{\cM}_{g,n} \twoheadrightarrow \overline{\cM}_g$. Since $\overline{t}_{g,n}$ factors through $\overline{t}_g$ via a surjection $p_{g,n}$, it suffices to show that $(p_{g,n}(C;\sigma_1,\dotsc,\sigma_n))^{\mathrm{pst}}$ is isomorphic to $C^\mathrm{pst}$ by Lemma~\ref{l:Torelli_normalization-sep_stabilization} and \cite[Theorem 2.1.7]{CV11}; the following proof also applies to $(C';\sigma_1',\dotsc,\sigma_n')$. Recall that $(p_{g,n}(C;\sigma_1,\dotsc,\sigma_n))^{\mathrm{pst}}$ is created from $C$ by contracting all rational tails and bridges, normalizing all separating nodes, contracting rational tails and bridges one at a time, then deleting any connected components isomorphic to $\PP^1_K$. Notice that a node $q$ of $p_{g,n}(C;\sigma_1,\dotsc,\sigma_n)$ is separating if and only if each node in the preimage of $q$ in $C$ is also separating. Moreover, for any separating node $q'$ of $C$, $p_{g,n}(q')$ in $p_{g,n}(C;\sigma_1,\dotsc,\sigma_n)$ is either a separating node or a smooth point. Hence, $(p_{g,n}(C;\sigma_1,\dotsc,\sigma_n))^{\mathrm{pst}}$ isomorphic to $C^\mathrm{pst}$ as long as there is a stabilization $\widetilde{C} \to (\widetilde{C})^\mathrm{st}$ that contracts all rational tails and bridges of $C$.
		
		Observe that in $\widetilde{C}$, rational tails of $C$ correspond to a disjoint union of connected components of $\widetilde{C}$, each of which is isomorphic to $\PP^1_K$; so rational tails are not present in $C^{\mathrm{pst}}$ by Lemma~\ref{l:Torelli_normalization-sep_stabilization}. In $\widetilde{C}$, rational bridges of $C$ correspond to a disjoint union of 1) connected components of $\widetilde{C}$, each of which is isomorphic to $\PP^1_K$, and 2) at most one rational bridge $B$, which is properly contained in a connected component $C_1$. If such $B$ exist, then there is a stabilization of $C_1$ where $B$ is contracted (recall that stabilization map is not unique when $g(C_1) = 1$). Thus, there exists a stabilization $\widetilde{C} \to (\widetilde{C})^{\mathrm{st}}$ that contracts all rational bridges of $C$. Hence, this induces a desired isomorphism $(p_{g,n}(C;\sigma_1,\dotsc,\sigma_n))^{\mathrm{pst}} \cong C^\mathrm{pst}$.
	\end{proof}

The following example shows why we must consider the quasi-separating axis-like locus.

	\begin{example}[{Torelli map does not extend to all axis-like curves}] \label{ex:why_quasi-separating_axis-like}
		Let $C_i \in \overline{\cM}_{g+3}(K)$ for $i=1,2$ as in Example~\ref{ex:not_C1-equivalent}, so that the dual graphs of $C_1$ and $C_2$ are isomorphic. If there exists an extremal assignment $\cZ$ of $G_{g+3,0}$ such that $\cZ(\Gamma_{C_i})$ is the non-quasi-separating rational $4$-bridge, then \cite[Corollary 1(2)]{Dav78} and \cite[Tag 0EUS]{Stacks} imply that there are $\cZ(\Gamma_{C_i})$-contractions $C_i \to C'$ with $C'$ a $\cZ$-stable axis-like curve that is not quasi-separating. This implies that $C_1,C_2$ are non-isomorphic $K$-points of a substack $F_\cZ^{-1}(C')$ of $\cM_{g+3}^{\cZ-\mathrm{qs}}$. Since $\overline{t}_{g+3}(C_1) \not \cong \overline{t}_{g+3}(C_2)$ by Theorem~\ref{thm:Torelli_C1-equivalence}, the restriction of $|\overline{t}_{g+3}|$ to $|\cM_{g+3}^{\cZ-\mathrm{qs}}|$ cannot factor through $|F_\cZ|$ because $|\overline{t}_{g+3}|(|F_\cZ|^{-1}(C'))$ is not a single point of $\left|\overline{\cA}_{g+3}^\mathrm{Ale}\right|$. This justifies the notion of quasi-separating axis-like curves as in Definition~\ref{def:quasi-separating_axis-like}, assuming that it is reasonable to enforce marked points to be distinct and lie on the smooth locus.
	\end{example}

	Now we are ready to prove the following main theorem, which immediately implies Theorem~\ref{thm:main_thm_partial_extension_axis_like}:
	
	\begin{theorem} \label{thm:extended_Torelli_Z-stable}
		Assume that $(g,n)$ satisfies $2g-2+n>0$. On $\overline{\cM}_{g,n}(\cZ)$, the Torelli map $t_{g,n} \colon \cM_{g,n} \to \cA_g$ admits a partial extension $t_{g,n}^{\cZ} \colon \cM_{g,n}(\cZ)^{\textrm{qs-axis}} \to \overline{\cA}_g^\mathrm{Ale}$ such that the following diagram commutes:
		\[
			\xymatrix@R=0.7in@C=0.5in{
				\overline{\cM}_{g,n} \ar[r]^{\overline{t}_{g,n}} & \overline{\cA}_g^{\mathrm{Ale}}\\
				\cM_{g,n}^{\cZ-\mathrm{qs}} \ar@{^{(}->}[u]^{\iota'_\cZ} \ar@{>>}[r]_(0.35){F_\cZ} & \cM_{g,n}(\cZ)^{\textrm{qs-axis}} \ar[u]_{t_{g,n}^{\cZ}},
			}
		\]
		where $\iota'_\cZ \colon \cM_{g,n}^{\cZ-\mathrm{qs}} \to \overline{\cM}_{g,n}$ is the open immersion. 
		The map $t_{g,n}^{\cZ}$ is unique up to unique $2$-isomorphism.
	\end{theorem}

	A special case of Theorem~\ref{thm:extended_Torelli_Z-stable} is when $\cZ = \sF$. In this case, the inclusion $\iota'_\cZ$ is an isomorphism by Definition~\ref{def:stable_Z-quasi-separating_axis-like} and $F_{\sF} = F$ by Proposition~\ref{prop:forget_DM_stable_to_stable_axis-like} and Proposition~\ref{prop:forget_stable_Z-qsep_to_Z-stable_qsep_axis-like}. This with Theorem~\ref{thm:extended_Torelli_Z-stable} immediately implies the following theorem, which immediately implies Theorem~\ref{thm:main_thm_pointed_extension_axis_like}. 

	\begin{theorem} \label{thm:compactified_Torelli_stable_separating_axis-like}
		Assume that $(g,n)$ satisfies $2g-2+n>0$. On $\overline{\cM}_{g,n}(\sF)$, the Torelli map $t_{g,n} \colon \cM_{g,n} \to \cA_g$ admits a compactification $\overline{t}_{g,n}^{\sF} \colon \overline{\cM}_{g,n}(\sF) \to \overline{\cA}_g^\mathrm{Ale}$ such that the following diagram of proper DM stacks commutes
		\[
		\xymatrix@R=0.5in@C=0.1in{
			\overline{\cM}_{g,n} \ar[rr]^{\overline{t}_{g,n}} \ar@{>>}[dr]_(0.45){F} & & \overline{\cA}_g^{\mathrm{Ale}}\\
			& \overline{\cM}_{g,n}(\sF) \ar[ur]_(0.6){\overline{t}_{g,n}^{\sF}},
		}
		\]
		The map $\overline{t}_{g,n}^{\sF}$ is unique up to unique $2$-isomorphism. 
	\end{theorem}

	\begin{proof}[Proof of Theorem~\ref{thm:extended_Torelli_Z-stable}]
		Since $\overline{\cM}_{g,n}(\cZ)$ is a proper DM stack by Definition~\ref{def:stack_Z-stable}, $\cM_{g,n}(\cZ)^{\textrm{qs-axis}}$ satisfy conditions of Theorem~\ref{thm:extension-stacks} by Proposition~\ref{prop:forget_stable_Z-qsep_to_Z-stable_qsep_axis-like} and Corollary~\ref{cor:moduli_Z-stable_qs_axis-like_lasc}. Thus, the existence of $t_{g,n}^\cZ$ follows from the existence of a set-theoretic extension $|t_{g,n}^\cZ| \colon |\cM_{g,n}(\cZ)^{\textrm{qs-axis}}| \to \left| \overline{\cA}_g^\mathrm{Ale} \right|$ of $t_{g,n} \colon \cM_{g,n} \to \cA_g$ such that $|\overline{t}_{g,n}| \circ |\iota_\cZ'| = |t_{g,n}^\cZ| \circ |F_\cZ|$ that satisfies Property~\ref{extension:DVRs} of Theorem~\ref{thm:extension-stacks}.
		
		We claim that Property~\ref{extension:DVRs} of Theorem~\ref{thm:extension-stacks} follows directly from the existence of the desired set-theoretic extension $|t_{g,n}^\cZ|$ and the fact that $F_\cZ$ is a proper representable morphism of stacks. To see this, suppose that $R$ is a DVR with $L \colonequals \mathrm{Frac}(R)$. Let $\iota_\eta \colon \Spec L \hookrightarrow \Spec R$ to be the inclusion of the generic point of $\Spec R$. Suppose that $y \colon \Spec R \to \cM_{g,n}(\cZ)^{\textrm{qs-axis}}$ is a morphism such that $y \circ \iota_\eta$ represents a $L$-point of $\cM_{g,n}$. 
		Since $F_\cZ$ is a quasi-resolution by Proposition~\ref{prop:forget_stable_Z-qsep_to_Z-stable_qsep_axis-like}, it is a representable proper morphism; 
		so interpreting $y \circ \iota_\eta$ as a $L$-point $x_\eta \in \cM_{g,n}^{\cZ-\mathrm{qs}}(L)$ (so that $F_\cZ \circ x_\eta \cong y \circ \iota_\eta$) implies that there is a lift $x \colon \Spec R \to \cM_{g,n}^{\cZ-\mathrm{qs}}$ of $y \colon \Spec R \to \cM_{g,n}(\cZ)^{\textrm{qs-axis}}$, i.e., $F_\cZ \circ x \cong y$. Then the existence of $|t_{g,n}^\cZ|$ such that $|\overline{t}_{g,n}| \circ |\iota_\cZ'| = |t_{g,n}^\cZ| \circ |F_\cZ|$ implies that $\overline{t}_{g,n} \circ \iota_\cZ' \circ x \colon \Spec R \to \overline{\cA}_g^\mathrm{Ale}$ is a DVR extension of $t_{g,n} \circ y \circ \iota_\eta \colon \Spec L \to \overline{\cA}_g^\mathrm{Ale}$ such that $|\overline{t}_{g,n} \circ \iota_\cZ' \circ x| = |t_{g,n}^\cZ| \circ |F_\cZ| \circ |x| = |t_{g,n}^\cZ| \circ |y|$, which implies Property~\ref{extension:DVRs} of Theorem~\ref{thm:extension-stacks}.
		
		
		Therefore, it remains to construct our desired $|t_{g,n}^\cZ|$, i.e., $|\overline{t}_{g,n}| \circ |\iota_\cZ'| = |t_{g,n}^\cZ| \circ |F_\cZ|$.
		To obtain such $|t_{g,n}^\cZ|$, it suffices to prove that for every two distinct points of $|\cM_{g,n}^{\cZ-\mathrm{qs}}|$ whose images under $|F_\cZ|$ coincide in $|\cM_{g,n}(\cZ)^{\textrm{qs-axis}}|$, their images under $|\overline{t}_{g,n}| \circ |\iota_\cZ'|$ also coincide in $\left| \overline{\cA}_g^\mathrm{Ale} \right|$. 
		To see this, let $K$ be any algebraically closed field, and let $(C;\sigma_1,\dotsc,\sigma_n)$ be a $\cZ$-stable $n$-pointed quasi-separating axis-like curve of genus $g$ defined over $\Spec K$. Suppose that for every $i=1,2$, $\phi_i \colon C^{s,i} \to C$ is the contraction of $\cZ(C^{s,i})$, where $(C^{s,i};\sigma_1^{s,i},\dotsc,\sigma_n^{s,i})$ is a stable $n$-pointed $\cZ$-quasi-separating curve of genus $g$ defined over $\Spec K$ and $\phi_i(\sigma_j^{s,i})=\sigma_j$ for all $j$. Then $|t_{g,n}^\cZ|$ is well-defined at $(C;\sigma_1,\dotsc,\sigma_n)$ if and only if $\overline{t}_{g,n}(C^{s,1}) \cong \overline{t}_{g,n}(C^{s,2})$, which is equivalent to the claim that $(C^{s,1})^{\mathrm{pst}}$ and $(C^{s,2})^{\mathrm{pst}}$ are $\mathrm{C}1$-equivalent by Theorem~\ref{thm:Torelli_C1-equivalence}. If $(C;\sigma_1,\dotsc,\sigma_n)$ is also a stable nodal curve (that is also $\cZ$-stable), then Proposition~\ref{prop:forget_stable_Z-qsep_to_Z-stable_qsep_axis-like} implies that $\phi_i$ is an isomorphism for all $i$ by Proposition~\ref{prop:forget_stable_Z-qsep_to_Z-stable}, so $(C^{s,1})^{\mathrm{pst}} \cong C^{\mathrm{pst}} \cong (C^{s,2})^{\mathrm{pst}}$, which implies that $(C^{s,1})^{\mathrm{pst}}$ and $(C^{s,2})^{\mathrm{pst}}$ are $\mathrm{C}1$-equivalent. Thus, assume that $(C;\sigma_1,\dotsc,\sigma_n)$ is not a nodal curve.
		
		Let us fix any $i \in \{1,2\}$ and describe $C^{s,i}$ at separating nodes in terms of $C$ and $\cZ(C^{s,i})$. Let $p_1,\dotsc,p_u$ be an enumeration of the non-nodal singular points of $C$, so that for every $j$, $p_j$ is an $m_j$-axis point (that is also quasi-separating) of $C$ with $m_j \ge 3$. For every $1 \le j \le u$, let $q_{j,1},\dotsc,q_{j,m_j}$ be the enumeration of $\nu_\mathrm{axis}^{-1}(p_j)$, where $\nu_{\mathrm{axis}} \colon C^\nu_\mathrm{axis} \to C$ is the normalization of $C$ at $\{p_1,\dotsc,p_u\}$. For each $j$, define $e_j$ to be $1$ if $p_j$ is a separating $m_j$-axis point of $C$; otherwise, define $e_j$ to be the maximum, among connected components $E$ of the normalization $\nu_{p_j} \colon C^\nu_{p_j} \to C$ of $C$ at $p_j$, of the number $\# \Supp (E \cap \nu_{p_j}^{-1}(p_j))$; by Definition~\ref{def:separating_axis_sing}, $e_j \ge 2$ if $p_j$ is not separating. Since each $p_j$ is a quasi-separating $m_j$-axis point of $C$, Definition~\ref{def:qsep_axis} implies that $e_j \le 3$ and $e_j=1$ if and only if $p_j$ is a separating $m_j$-axis point. Whenever $e_j >1$, assume, without loss of generality, that the image of $q_{1,1},\dotsc,q_{1,e_j}$ under the induced morphism $C^\nu_{\mathrm{axis}} \to C^\nu_{p_j}$ belong to the same connected component of $C^\nu_{p_j}$. Recalling the proof of Proposition~\ref{prop:fiber_M0,mbar_qsep} (which draws an analogy with the proof of Proposition~\ref{prop:fiber_M0,mbar}), $C^{s,i}$ is isomorphic to the nodal union of $C^\nu_\mathrm{axis}$ and quasi-separating rational $m_j$-bridges $(Z^i_j;q^i_{j,1},\dotsc,q^i_{j,m_j})$ (which is a stable $m_j$-pointed curve of genus zero) for every $1 \le j \le u$ by identifying $q_{j,v} \in C^\nu_\mathrm{axis}$ with $q^i_{j,v} \in Z^i_j$ for every $1 \le v \le m_j$; here, $\cZ(C^{s,i})$ is exactly the disjoint union of $Z^i_j$'s for every $1 \le j \le u$.
		
		We claim that the geometry of the normalization $\nu_{s,i} \colon \widetilde{C}^{s,i} \to C^{s,i}$ at separating nodes of $C^{s,i}$ can be described in terms of $\widetilde{C}$ and $Z^i_j$'s. To see this, suppose that $x$ is a separating node of $C^{s,i}$, so that $C^{s,i}$ decomposes into a union of two connected components $C^{s,i}_{x,1}$ and $C^{s,i}_{x,2}$ that intersect exactly at $x$. If $x$ is not contained in $\cZ(C^{s,i})$, then $\phi_i(x)$ is a node of $C$, and Definition~\ref{def:Z-stability} implies that $\phi_i(C^{s,i}_{x,1})$ and $\phi_i(C^{s,i}_{x,2})$ form the decomposition of $C$ into two connected components, such that their intersection is exactly $\phi_i(x)$. Hence, $\phi_i(x)$ is also a separating node of $C$. Similarly, if $x$ is a node of $C^{s,i}$ such that $\phi_i(x)$ is a separating node of $C$, then $x$ is also a separating node of $C^{s,i}$. 
		
		The remaining case is when $x$ is contained in $Z^i_j$ for some $j$. Assume that $x = q^i_{j,v}$ for some $v$. If $e_j=1$, then $Z^i_j$ is a separating rational $m_j$-bridge of $C^{s,i}$, thus $x = q^i_{j,v}$ is a separating node of $C^{s,i}$. If $e_j > 1$, then $Z^i_j$ is a quasi-separating rational $m_j$-bridge that is not a separating rational $m_j$-bridge. Let $E_j$ be the closure in $C$ of the connected component of $C \setminus \{p_j\}$ such that the closure $E^{s,i}_j$ of $\phi_i^{-1}(E_j \setminus \{p_j\})$ meets $Z^i_j$ at exactly $\{q^i_{j,1},\dotsc,q^i_{j,e_j}\}$. If $v > e_j$, then $x = q^i_{j,v}$ is not in $E^{s,i}_j$, so let $E^i_x$ be the connected component of $\overline{C^{s,i} \setminus Z_j}$ that contains $x$. Then, $E^i_x \cap Z^i_j = \{x\}$, and $\overline{C^{s,i} \setminus E_x}$ is connected such that $Z^i_j$ is also contained in $\overline{C^{s,i} \setminus E_x}$. Thus, $x$ is a separating node of $C^{s,i}$. The remaining case is when $v \le e_j$; then $x \in E^{s,i}_j \cap Z^i_j$. By Definition~\ref{def:Z-stability}, connectedness of $E_j \setminus \{p_j\}$ implies that $E^{s,i}_j$ is also connected, so that the dual graph $\Gamma^{s,i}_j$ of the subcurve $E^{s,i}_j \cup Z^i_j$ of $C^{s,i}$ contains a cycle $H$, where $E(H)$ contains $x$ (as an edge of $\Gamma^{s,i}_j$). Since $\Gamma^{s,i}_j$ is the complete subgraph of the dual graph of $C^{s,i}_j$, the existence of cycle $H$ with $x \in E(H)$ implies that $x$ is not a separating node of $C^{s,i}$. If $x$ is not equal to $q^i_{j,v}$ for any $v$, then $x$ is a node of $Z^i_j$. If $e_j=1$, then $Z^i_j$ is a separating rational $m_j$-bridge, so that every node of $Z^i_j$ is a separating node of $C^{s,i}$. On the other hand, if $e_j > 1$, then there is a minimal subcurve $\widehat{Z}^i_j$ of $Z^i_j$ that contains $\{q^i_{j,1},\dotsc,q^i_{j,e_j}\}$. In this case, $x$ (as a node of $Z^i_j$) is a separating node of $C^{s,i}$ if and only if $x$ is not contained in $\widehat{Z}^i_j$.
		
		Recalling that $(C^{s,i})^\mathrm{pst}$ is the removal of all connected components of genus zero from $(\widetilde{C}^{s,i})^{\mathrm{st}}$ by Lemma~\ref{l:Torelli_normalization-sep_stabilization}, define $\widehat{C}^{s,i}$ as the disjoint union of connected components of $\widetilde{C}^{s,i}$ of genus greater than zero. Also, let $\widehat{C}'$ be the normalization of $C$ at the set of separating nodes of an axis-like curve $C$. Note that $\widehat{C}'$ still contains $p_j$ as a quasi-separating $m_j$-axis point for every $j$. Define  $\widehat{C}_\mathrm{axis}$ as the intersection of $\widehat{C}^{s,i}$ and the normalization $\widehat{C}'_\mathrm{axis}$ of $\widehat{C}'$ at $\{p_1,\dotsc,p_u\}$; $\widehat{C}'_\mathrm{axis}$ being a reduced closed subscheme of $\widetilde{C}^{s,i}$ follows from the analysis of separating nodes of $C^{s,i}$ from the previous two paragraphs. By construction, a connected component $D$ of $\widehat{C}'_\mathrm{axis}$ is not in $\widehat{C}_\mathrm{axis}$ if and only if the image $D'$ of $D$ in $\widehat{C}'$ is a connected component of genus zero or the image $D''$ of $D$ in $C^{s,i}$ is a curve of genus zero that meets $\sqcup_{j=1}^u Z^i_j$ exactly once; the latter condition is equivalent to the image $D'''$ of $D$ in $C$ being a curve of genus zero that meets the set $\{p_1,\dotsc,p_u\}$; so $D'''$ is not contained in any of the $E_j$'s for any $1 \le j \le u$. This implies that the isomorphism class of $\widehat{C}_\mathrm{axis}$ is independent of the choice of $i \in \{1,2\}$ (in fact, $\widehat{C}_\mathrm{axis}$ only depends on $C$). From this, the analysis of separating nodes of $C^{s,i}$ from the previous two paragraphs implies that $\widehat{C}^{s,i}$ is the nodal union of $\widehat{C}_\mathrm{axis}$ and $\widehat{Z}^i_j$'s for all $1 \le j \le u$, where $\widehat{Z}^i_j$ is understood as the empty set whenever $e_j = 1$. Observe that for each $1 \le j \le u$, when $\widehat{Z}^i_j$ is nonempty (i.e., $e_j > 1$), then the intersection of $\widehat{C}_\mathrm{axis}$ and $\widehat{Z}^i_j$ in $\widehat{C}^{s,i}$ consists of $e_j$ number of nodes of $\widehat{C}^{s,i}$, corresponding to the identification of $q_{j,v} \in \widehat{C}_\mathrm{axis}$ (by interpreting the induced morphism $\widehat{C}_\mathrm{axis} \to C^\nu_\mathrm{axis}$ as a partial normalization and then abusing notation on $q_{j,v} \in \widehat{C}_\mathrm{axis}$) and $q^i_{j,v} \in \widehat{Z}^i_j$ for each $1 \le v \le e_j$.
		
		By the construction of $\widehat{C}^{s,i}$, $(C^{s,i})^\mathrm{pst}$ is isomorphic to $(\widehat{C}^{s,i})^\mathrm{st}$. Since stabilization commutes with disjoint unions, it suffices to describe the stabilization of each connected component $T$ of $\widehat{C}^{s,i}$ and show that the resulting connected curve $T^\mathrm{st}$ is independent of the choice of $i \in \{1,2\}$, i.e., there is a description of $T^\mathrm{st}$ only in terms of $\widehat{C}_\mathrm{axis}$. There are two cases to consider: either $T$ is disjoint from $\sqcup_{j=1}^u \widehat{Z}^i_j$, or $T$ contains $\sqcup_{j \in J} \widehat{Z}^i_j$ for some nonempty subset $J$ of $\{1,\dotsc,u\}$ such that for every $j \in J$, $e_j > 1$. First, assume that $T$ is disjoint from $\sqcup_{j=1}^u \widehat{Z}^i_j$. Then $T$ is a connected component of $\widehat{C}_\mathrm{axis}$ that is disjoint from $\{q_{j,v} \in \widehat{C}_\mathrm{axis} \; : \; 1 \le j \le u, \; e_j > 1, \; 1 \le v \le e_j \}$. Since $\widehat{C}_\mathrm{axis}$ is uniquely determined from $C$ by the construction, $T^\mathrm{st}$ is a connected component of $(C^{s,1})^\mathrm{pst}$ and $(C^{s,2})^\mathrm{pst}$ that is entirely constructed from $C$.
		
		The remaining case is when $T$ contains $\sqcup_{j \in J} \widehat{Z}^i_j$ for some nonempty subset $J$ of $\{1,\dotsc,u\}$, where $e_j > 1$ for every $j \in J$. Recall by construction that $g(T) \ge 1$. By construction, the image $T'$ of $T$ in $C$ is a subcurve that is contained in $\cap_{j \in J} E_j$ and contains $\{p_j \; : \; j \in J\}$. In fact, $T'$ is identified with a connected component of the normalization of $\cap_{j \in J} E_j$ at its separating nodes; so $T'$ is also a subcurve of $\widehat{C}'$. The induced map $T \to T'$ is a contraction of $\sqcup_{j \in J} \widehat{Z}^i_j$. If $g(T)=1$, then as a nodal curve without separating nodes, the dual graph $\Gamma_T$ of $T$ is a cycle with $w \ge 2$ number of vertices. By construction, $\Gamma_T \supsetneq \Gamma_{\sqcup_{j \in J} \widehat{Z}^i_j}$ where $\Gamma_{\sqcup_{j \in J} \widehat{Z}^i_j}$ is the dual graph of $\sqcup_{j \in J} \widehat{Z}^i_j$, and $e_j = 2$ for every $j \in J$. Therefore, $\sqcup_{j \in J} \widehat{Z}^i_j$ is the union of some rational bridges of $T$. So there exists an irreducible component $T_0 \cong \PP^1$ of $\overline{T \setminus \sqcup_{j \in J} \widehat{Z}^i_j}$. By construction, the image of $T_0$ in $C$ is an irreducible component of $C$. Thus, it is possible to fix a choice of $T_0$ independent of $i \in \{1,2\}$. Then, there is a stabilization $T \to T^\mathrm{st}$ such that it factors through the induced $\sqcup_{j \in J} \widehat{Z}^i_j$-contraction $T \to T'$, and the restriction of $T \to T^\mathrm{st}$ to $T_0$ is a surjection. Therefore, $T^\mathrm{st}$ is isomorphic to the irreducible semistable curve of arithmetic genus one (and geometric genus zero), which is a connected component of $(C^{s,1})^\mathrm{pst}$ and $(C^{s,2})^\mathrm{pst}$ that is entirely constructed from $T_0 \subset C$.
		
		Finally, if $g(T) > 1$, then the stabilization is unique, i.e., there is a unique stabilization morphism $\psi: T \to T^\mathrm{st}$ as a contraction of rational bridges (here, rational tails do not appear because $T$ is free from separating nodes). Notice that for every $j \in J$, $\widehat{Z}^i_j$ is a rational tree that is attached to $\overline{T \setminus \widehat{Z}^i_j}$ along $e_j$ number of distinct smooth points of $\widehat{Z}^i_j$. Thus, $\widehat{Z}^i_j$ is either a rational bridge or $3$-bridge depending on whether $e_j$ equals two or three. If $\widehat{Z}^i_j$ is a rational bridge, then it contracts to a point under $\psi$. If $\widehat{Z}^i_j$ is a rational $3$-bridge, then $\psi(\widehat{Z}^i_j) \subset T^\mathrm{st}$ is an irreducible rational $3$-bridge $(B^i_j;q^i_{j,1},q^i_{j,2},q^i_{j,3})$ that is isomorphic to $(\PP^1_K;0,1,\infty)$; here, by abuse of notation, $q^i_{j,v}$ refers to $\psi(q^i_{j,v}) \in T^\mathrm{st}$, where $q^i_{j,v} \in T$ refers to a point of $\widehat{Z}^i_j \subset \widehat{C}^{s,i}$. Let $T'$ to be the image of $T$ under the normalization map $\widehat{C}^{s,i} \to \widehat{C}'$ and let $J' \colonequals \{j \in J \mid e_j = 3\}$. Consider the nodal union $T''$ of the partial normalization $(T')^\nu_{J'}$ of $T' \subset \widehat{C}'$ along $\{p_j \in \widehat{C}' \mid j \in J'\}$ (so $\overline{T \setminus \sqcup_{j \in J'} \widehat{Z}^i_j} \cong (T')^\nu_{J'}$) and $\sqcup_{j \in J'} (B^i_j;q^i_{j,1},q^i_{j,2},q^i_{j,3})$, where each $q_{j,v} \in (T')^\nu_{J'}$ for every $j \in J'$ and $1 \le v \le 3$ (coming from $q_{j,v} \in \widehat{C}_\mathrm{axis}$ via the induced morphism $\widehat{C}_\mathrm{axis} \to \widehat{C}'$) is glued to $q^i_{j,v} \in B^i_j$. Since $T^\mathrm{st}$ is isomorphic to $(T'')^\mathrm{st}$ by construction and the isomorphism class of $(T'')^\mathrm{st}$ is uniquely determined by $C$ (not $C^{s,1}$ and/or $C^{s,2}$), $T^\mathrm{st}$ is a connected component of $(C^{s,1})^\mathrm{pst}$ and $(C^{s,2})^\mathrm{pst}$ that is entirely constructed from $C$.
		
		Combining all of the observations from the above, $(C^{s,1})^{\mathrm{pst}} \cong (C^{s,2})^{\mathrm{pst}}$, which implies that $(C^{s,1})^{\mathrm{pst}}$ and $(C^{s,2})^{\mathrm{pst}}$ are $\mathrm{C}1$-equivalent. Therefore, the desired assertion that $\overline{t}_{g,n}(C^{s,1}) \cong \overline{t}_{g,n}(C^{s,2})$ holds by Theorem~\ref{thm:Torelli_C1-equivalence}, which proves the statement of the theorem.
	\end{proof}
	
	\section*{Acknowledgements}
	We thank Donu Arapura, Dori Bejleri, Dawei Chen, Dan Edidin, Patricio Gallardo, Joe Harris, Aaron Landesman, James McKernan, David McKinnon, Akash Sengupta, and Chenyang Xu for enlightening discussions and encouragement.
	We also thank J\'anos Koll\'ar 
	for reading the draft carefully and giving insightful feedbacks.
	

\end{document}